\documentclass{tran-l}

\usepackage{verbatim}
\usepackage{appendix}
 \usepackage{amssymb}
\usepackage{xcolor}
\newtheorem{theorem}{Theorem}[section]
\newtheorem{lemma}[theorem]{Lemma}
\newtheorem{proposition}[theorem]{Proposition}
\newtheorem*{claim}{Claim}

\newtheorem{corollary}[theorem]{Corollary}
\theoremstyle{definition}
\newtheorem{definition}[theorem]{Definition}
\newtheorem{example}[theorem]{Example}

\theoremstyle{remark}
\newtheorem{remark}[theorem]{Remark}
\numberwithin{equation}{section}

\let\origcolor\color

\usepackage{marginnote}               
\usepackage{tikz}
\usetikzlibrary{calc,chains,fit}
\newcommand{\changed}[1]{\textcolor{blue}{#1}}
\newcommand{\changednew}[1]{{\origcolor{purple}#1}}

\newif\ifpubinfo      \pubinfofalse       
\newif\ifannotations  \annotationsfalse 

\usepackage{xparse} 
\makeatletter
\ifannotations
  \let\refchange\refchange@full
\else
  \let\refchange\relax
  \@ifundefined{refchange}{%
    \NewDocumentCommand{\refchange}{ O{} m +m }{#3}%
  }{%
    \RenewDocumentCommand{\refchange}{ O{} m +m }{#3}%
  }%
\fi
\makeatother

\makeatletter
\providecommand{\refchange@full}[3][]{#3}
\makeatother
\providecommand{\refchange}[2]{#2}
\makeatletter
\ifannotations
  \let\refchange\refchange@full
\else
    \renewcommand{\color}[2][]{\relax}
    \renewcommand{\textcolor}[2]{#2}

\fi
\makeatother

\newcommand{\refstagger}{1.2ex}        
\newcounter{refchangeid}

\newcount\rcSideToggle
\makeatletter
\renewcommand{\refchange@full}[3][]{%
  \stepcounter{refchangeid}%
  \def\rcOptSkip{#1}%
  \ifx\rcOptSkip\@empty
    \ifodd\value{refchangeid}%
      \def\rcOptSkip{\refstagger}%
    \else
      \def\rcOptSkip{-\refstagger}%
    \fi
  \fi
  \def\rcContent{%
    \tikz[remember picture,baseline,start chain=going right,node distance=2pt]{
      \def\rcfirstname{}\def\rclastname{}%
      \foreach \n [count=\rcidx] in {#2}{
        \expandafter\node\expandafter[on chain,draw=black,rounded corners,
             inner sep=2pt,fill=white,text=black]
             (rcbox-\therefchangeid-\rcidx) {\textbf{\footnotesize \n}};
        \ifx\rcfirstname\empty \xdef\rcfirstname{rcbox-\therefchangeid-\rcidx}\fi
        \xdef\rclastname{rcbox-\therefchangeid-\rcidx}%
      }%
      \node[inner sep=0pt,fit=(\rcfirstname) (\rclastname),name=mref-\therefchangeid]{};
      \path (mref-\therefchangeid.center) coordinate (mrefstart-\therefchangeid);
    }%
  }%
  #3%
  \llap{\tikz[remember picture,baseline]\node (rctxt-\therefchangeid) {}; }%
  \advance\rcSideToggle by 1\relax
  \begingroup
    \ifodd\rcSideToggle
      \normalmarginpar        
    \else
      \reversemarginpar       
    \fi
    \marginnote{\rcContent}[\rcOptSkip]%
  \endgroup
  \expandafter\ifx\csname refdraw@\the\value{refchangeid}\endcsname\relax
    \expandafter\gdef\csname refdraw@\the\value{refchangeid}\endcsname{1}%
    \begin{tikzpicture}[remember picture,overlay]
      \draw[-stealth,gray,shorten <=6pt,line cap=round]
        (mrefstart-\therefchangeid) -- (rctxt-\therefchangeid);
    \end{tikzpicture}%
  \fi
}
\makeatother

\newif\ifannotationsnew  \annotationsnewfalse 
\makeatletter
\ifannotationsnew
  \let\refchangenew\refchangenew@full
\else
  \let\refchangenew\relax
  \@ifundefined{refchangenew}{%
    \NewDocumentCommand{\refchangenew}{ O{} m +m }{#3}%
  }{%
    \RenewDocumentCommand{\refchangenew}{ O{} m +m }{#3}%
  }%
\fi
\makeatother
\makeatletter
\providecommand{\refchangenew@full}[3][]{#3}
\makeatother
\providecommand{\refchangenew}[2]{#2}
\makeatletter
\ifannotationsnew
  \let\refchangenew\refchangenew@full
\else
    \renewcommand{\origcolor}[2][]{\relax}

\fi
\makeatother

\newcommand{\refstaggernew}{1.2ex}        
\newcounter{refchangenewid}

\newcount\rcSideToggle
\makeatletter
\renewcommand{\refchangenew@full}[3][]{%
  \stepcounter{refchangenewid}%
  \def\rcOptSkip{#1}%
  \ifx\rcOptSkip\@empty
    \ifodd\value{refchangenewid}%
      \def\rcOptSkip{\refstaggernew}%
    \else
      \def\rcOptSkip{-\refstaggernew}%
    \fi
  \fi
  \def\rcContent{%
    \tikz[remember picture,baseline,start chain=going right,node distance=2pt]{
      \def\rcfirstname{}\def\rclastname{}%
      \foreach \n [count=\rcidx] in {#2}{
        \expandafter\node\expandafter[on chain,draw=black,rounded corners,
             inner sep=2pt,fill=white,text=black]
             (rcbox-\therefchangenewid-\rcidx) {\textbf{\footnotesize \n}};
        \ifx\rcfirstname\empty \xdef\rcfirstname{rcbox-\therefchangenewid-\rcidx}\fi
        \xdef\rclastname{rcbox-\therefchangenewid-\rcidx}%
      }%
      \node[inner sep=0pt,fit=(\rcfirstname) (\rclastname),name=mref-\therefchangenewid]{};
      \path (mref-\therefchangenewid.center) coordinate (mrefstart-\therefchangenewid);
    }%
  }%
  #3%
  \llap{\tikz[remember picture,baseline]\node (rctxt-\therefchangenewid) {}; }%
  \advance\rcSideToggle by 1\relax
  \begingroup
    \ifodd\rcSideToggle
      \normalmarginpar        
    \else
      \reversemarginpar       
    \fi
    \marginnote{\rcContent}[\rcOptSkip]%
  \endgroup
  \expandafter\ifx\csname refdraw@\the\value{refchangenewid}\endcsname\relax
    \expandafter\gdef\csname refdraw@\the\value{refchangenewid}\endcsname{1}%
    \begin{tikzpicture}[remember picture,overlay]
      \draw[-stealth,gray,shorten <=6pt,line cap=round]
        (mrefstart-\therefchangenewid) -- (rctxt-\therefchangenewid);
    \end{tikzpicture}%
  \fi
}
\makeatother

\makeatletter
\ifpubinfo
\else
  \let\ISSN\@empty
  \def\@serieslogo{}%
  \def\@setcopyright{}%
  \let\copyrightinfo\@gobbletwo
  \gdef\@volume{}%
  \gdef\@number{}%
  \gdef\@pages{}%
  \gdef\@date{}%
  \def\ps@firstpage{\ps@empty}%
\fi
\makeatother

\newif\ifshowstuff
   \showstufftrue    

\usepackage{comment}

\ifshowstuff
  \includecomment{showonly}     
  \newcommand{\maybe}[1]{#1}   
\else
  \excludecomment{showonly}    
  \newcommand{\maybe}[1]{}     
\fi

\begin{document}

\title{Critical Exponent Rigidity for $\Theta$-positive Representations}

\author{\maybe{Zhufeng Yao}}

\maybe{
  \thanks{The author was supported by the NUS-MOE grants  A-8000458-00-00 and A-8001950-00-00.}
}


\date{}

\dedicatory{}

\begin{abstract}
We prove that for a $\Theta$-positive representation into a simple Lie group from \refchange{2}{{\color{blue}a non-elementary discrete subgroup}} $\Gamma \subset \mathsf{PSL}(2,\mathbb{R})$, the critical exponent for any $\alpha \in \Theta$ is at most one. When $\Gamma$ is geometrically finite, equality holds if and only if $\Gamma$ is a lattice. This paper is a \refchangenew{1}{\changednew{generalization of the result by Canary–Zhang–Zimmer~\cite{CaZhZi}}}, \refchange{1}{{\color{blue}which concerns a special class of positive representations, namely Hitchin representations into $\mathsf{PSL}(d,\mathbb{R})$ arising from geometrically finite Fuchsian groups.}}
\end{abstract}

\maketitle
\setcounter{tocdepth}{1}
\tableofcontents

\section{Introduction}\label{sec1}
Let $S$ be a closed, oriented and connected hyperbolic surface. Hitchin~\cite{HITCHIN1992449} introduced a special connected component of 
\refchange{3}{{\color{blue}the representation variety modulo conjugation 
$Hom(\pi_1(S), \mathsf{PSL}(d,\mathbb{R}))/\mathsf{\refchangenew{3}{\changednew{PGL}}}(d,\mathbb{R})$}},
now known as the \emph{Hitchin component}, using the theory of Higgs bundles. When $d=2$, this component coincides with the classical Teichmüller space \refchange{4}{{\color{blue}of $S$}}. More generally, the Hitchin component is topologically a cell and provides a natural generalization of Teichmüller theory to higher rank Lie groups, initiating the field of higher Teichmüller theory. Representations in this component are called Hitchin representations.

Labourie \cite{labourie2005anosovflowssurfacegroups} showed that Hitchin representations are {\color{blue} \refchange{5}{discrete and faithful, and} possess a strong dynamical property which is called the \refchange[1.0\baselineskip]{6}{\emph{Anosov property}}}. This property generalizes convex cocompactness in rank one Lie groups. On a different but complementary front, Fock–Goncharov \cite{fock2006modulispaceslocalsystems} took an algebraic perspective and showed that Hitchin representations are exactly those admitting a continuous, equivariant positive map from the Gromov boundary $\partial \pi_1(S)$ of $\pi_1(S)$ to the total flag manifold $\mathcal{F}_d$. These results position Hitchin representations as natural higher-rank analogues of Fuchsian representations, mirroring many of the geometric and dynamical features of classical Teichmüller theory.

Another prominent family of higher Teichmüller spaces arises from the work of Burger–Iozzi–Wienhard \cite{maximalearlier, burger2008surfacegrouprepresentationsmaximal}, inspired by earlier work of Goldman \cite{Goldman1988}. These are the maximal representations—representations of $\pi_1(S)$ into Hermitian Lie groups of tube type that maximize the Toledo invariant. Maximal representations form connected components in the corresponding character variety, and, like Hitchin representations, they are discrete, faithful, and Anosov. A characterization parallel to that of Hitchin representations holds: a representation is maximal if and only if it admits a continuous, equivariant positive map to the Shilov boundary of the Hermitian symmetric space.

The central role of positivity in both Hitchin and maximal representations motivated Guichard–Wienhard \cite{GW2} to develop a unifying framework: the theory of $\Theta$-positivity. Generalizing Lusztig’s notion of total positivity~\cite{Lusztig1998TotalPI}, they defined the notion of positivity for a family of pairs $(G,\Theta)$, where $G$ is a semisimple Lie group and $\Theta$ is a nonempty subset of the set of simple roots \refchange{7}{{\color{blue}(note that not every pair $(G,\Theta)$ admits a positive structure)}}. In this setting, one can define positive configurations in the associated $\Theta$-flag manifold $\mathcal{F}_{\Theta}$. A representation is called $\Theta$-positive if it admits a continuous, equivariant map from $\partial\pi_1(S)$ to $\mathcal{F}_{\Theta}$ that sends cyclically oriented $n$-tuples to $\Theta$-positive configurations. This notion encompasses both Hitchin and maximal representations and is consistent with their respective positivity structures. {\color{blue}\refchangenew{9}{\changednew{More recently}}, it has been shown that $\Theta$-positive representations form connected components of the character variety 
$\mathrm{Hom}(\pi_1(S),G)/G$ (see Guichard–Labourie–Wienhard~\cite{GLW} and Beyer–Guichard–Labourie–Pozzetti–Wienhard~\cite{beyrer2024positivitycrossratioscollarlemma}).  One of the key steps is to show that the positive equivariant map on \refchange{8}{$\partial \pi_1(S)$} has good dynamical properties, which imply that $\Theta$-positive representations are $\Theta$-Anosov.

\refchange{9}{In this paper}, we further analyze the dynamics and regularity properties of the associated positive equivariant maps. As an application, our main result establishes a rigidity phenomenon for the critical exponent associated to each simple root in $\Theta$: it is always bounded \refchange{10}{from above} by one. When $\Gamma$ is geometrically finite, the equality holds if and only if $\Gamma$ is \refchange{11}{a lattice}. This generalizes earlier results of Potrie–Sambarino~\cite{potrie2017eigenvaluesentropyhitchinrepresentation}, Pozzetti–Sambarino–Wienhard~\cite{Liplim}, and Canary–Zhang–Zimmer~\cite{CaZhZi}. The first paper proved that for a Hitchin representation into $\mathsf{PSL}(d,\mathbb{R})$ from a closed surface group, the critical exponent of each root is $1$; the second paper proved that for an $\alpha$-Anosov representation from a closed surface group whose $\alpha$-limit set is a Lipschitz submanifold, the critical exponent of $\alpha$ is always $1$; and the third paper proved that for a Hitchin representation into $\mathsf{PSL}(d,\mathbb{R})$ from a discrete subgroup of $\mathsf{PSL}(2,\mathbb{R})$, the critical exponent of each root $\alpha$ equals the Hausdorff dimension of the $\alpha$-conical \refchange{12}{limit set}.}

\refchangenew{9}{\changednew{In the Introduction, we make the following assumptions:}}
\changednew{
\begin{enumerate}
    \item $G$ is a non-compact simple Lie group with finitely many components.
    \item The identity component $G^{\circ}$ has finite center.
    \item The conjugation action of $G$ on the Dynkin diagram is trivial (see Section~\ref{assumption}).
\end{enumerate}
}
These hypotheses will be in force throughout most of the paper, although we will restate the assumptions on $G$ explicitly in each section as needed.

\subsection{Critical Exponent Rigidity}

Let $\Gamma \subset \mathsf{PSL}(2,\mathbb{R})$ be a \refchange{13}{{\color{blue}non-elementary}} discrete subgroup acting on the hyperbolic disk $\mathbb{D}$, and let $\Lambda(\Gamma)\subset \partial \mathbb{D} = S^1$ denote its limit set. Let $G$ be a simple Lie group with Lie algebra $\mathfrak{g}$. Fix a Cartan decomposition $\mathfrak{g} = \mathfrak{k} \oplus \mathfrak{p}$, choose a Cartan subspace $\mathfrak{a} \subset \mathfrak{p}$, a positive Weyl chamber $\mathfrak{a}^+ \subset \mathfrak{a}$ \refchangenew{19}{\changednew{(which is open)}}, and let $\Delta$ denote the corresponding set of positive simple roots. Let $\kappa: G \to \overline{\mathfrak{a}^+}$ be the Cartan projection, and let $(\mathfrak{a}^*)^+ \subset \mathfrak{a}^*$ denote the cone of functionals positive on $\mathfrak{a}^+$ \refchangenew{19}{\changednew{(which is a closed cone with $\{0\}$ removed)}}.

Given a non-empty subset $\Theta \subset \Delta$, the corresponding flag variety and its opposite are denoted by $\mathcal{F}_\Theta,\mathcal{F}_{\Theta}^{opp}$. A representation $\rho: \Gamma \to G$ is called \emph{$\Theta$-transverse} if there exists a pair of continuous, $\rho$-equivariant, transverse, and strongly dynamics-preserving map $(\xi^\Theta,\xi^{\Theta,opp}): \Lambda(\Gamma) \to \mathcal{F}_\Theta\times \mathcal{F}_{\Theta}^{opp}$, referred to as the \emph{limit map} of $\rho$ (see Definition \ref{Anosov}). \refchange{17}{{\color{blue}Moreover, we will see that as a property of transverse representations, such a limit map is unique.}} When $\Gamma$ is geometrically finite, such representations are equivalent to relatively $\Theta$-Anosov representations (see Canary–Zhang–Zimmer~\cite{CaZhZi}, \cite{CaZhZi2} and \cite{CaZhZi3}).

\refchange{14}{{\color{blue}Suppose that}} $G$ admits a $\Theta$-positive structure; one can define positivity for $n$-tuples in $\mathcal{F}_\Theta$ {\color{blue}(for example, when $G=\mathsf{PSL}(d,\mathbb{R})$\changednew{,} $\Theta$ consists of all simple roots, and as mentioned above, positive tuples are described in Lusztig~\cite{Lusztig1998TotalPI}\refchangenew{14}{\changednew{)}}}. A representation $\rho: \Gamma \to G$ is \emph{$\Theta$-positive} if there exists a continuous and $\rho$-equivariant map $\xi^\Theta: \Lambda(\Gamma) \to \mathcal{F}_\Theta$ that sends \refchange{15}{{\color{blue}all}} cyclically ordered $n$-tuples in $\Lambda(\Gamma) \subset S^1$ to positive $n$-tuples in $\mathcal{F}_\Theta$. One can show that (see Guichard-Wienhard~\cite[Theorem~3.4]{GW2}) when $G$ admits a $\Theta$-positive structure, the set $\Theta$ is symmetric under any involution of the Dynkin diagram. Consequently, the map $\xi^{\Theta}$ tautologically induces a map 
\(
\xi^{\Theta,opp} : \Lambda(\Gamma)\;\longrightarrow\; \mathcal{F}_{\Theta}^{opp}.
\)

In the case where $\Gamma$ is a closed surface group, Guichard, Labourie and Wienhard showed in \cite{GLW} that $\Theta$-positive representations are $\Theta$-Anosov. \refchange{16}{{\color{blue}In the case where $\Gamma$ is a non-elementary geometrically finite subgroup of $\mathsf{PSL}(2,\mathbb{R})$, Canary, Zhang and Zimmer showed in \cite{CaZhZi2} that Hitchin representations from $\Gamma$ to $\mathsf{PSL}(d,\mathbb{R})$ are relatively Anosov with respect to all simple roots. We extend these results:}}

\begin{proposition}[Proposition \ref{thm17}]\label{propB}
Let $\Gamma \subset \mathsf{PSL}(2,\mathbb{R})$ be a \refchange{16}{{\color{blue}non-elementary}} discrete subgroup. If $\rho: \Gamma \to G$ is a $\Theta$-positive representation, then $\rho$ is $\Theta$-transverse. Moreover, the positive map $\xi^\Theta$ coincides with the limit map associated with the $\Theta$-transverse representation, \refchange{17}{{\color{blue}which is unique as a property of $\Theta$-transverse representation.}}
\end{proposition}

Let $\mathfrak{H}_\Theta^+ \subset \refchangenew{19}{\changednew{(\mathfrak{a}^*)^+}}$ {\color{blue} denote the sub-cone consisting of positive linear combinations of roots \refchange{18}{in $\Theta$.}} For any $\phi = \sum_{\alpha \in \Theta} c_\alpha \alpha \in \mathfrak{H}_\Theta^+$ \refchange{18}{{(\color{blue}note that $c_{\alpha}>0$ for any $\alpha \in \Theta$)}}, define

\[
a(\phi) := \sum_{\alpha \in \Theta} c_\alpha, \quad \hat{\phi} := \frac{\phi}{a(\phi)}.
\]
\changed{For each $\phi\in (\mathfrak{a}^*)^+$}, we define the \emph{$\phi$-Poincaré series} of the representation $\rho$ as
\[
Q^{\phi}_{\rho(\Gamma)}(s) := \sum_{\gamma \in \Gamma} e^{-s \phi(\kappa(\rho(\gamma)))},
\]
and the associated \emph{$\phi$-critical exponent} \refchange{19}{as}
\[
\delta^{\phi}_{\rho}(\Gamma) := \inf \left\{ {\color{blue}s \in \mathbb{R}^{>0} } \mid Q^{\phi}_{\rho(\Gamma)}(s) < \infty \right\}.
\]

\refchange{20}{{\color{blue}It has to be noted that for a general representation $\rho$, $\delta^{\phi}_{\rho}(\Gamma)$ might not be a finite number.}} However, if $\Gamma$ is geometrically finite and $\rho$ is relatively $\Theta$-Anosov, then  \refchange{21}{{\color{blue} Canary–Zhang–Zimmer \cite[Theorem~1.1]{CaZhZi2} (see also Proposition~\ref{anosovgap}) showed that}} for any fixed base point $b_0 \in \mathbb{D}$, there exist constants $a > 1$ and $A > 0$ such that for all $\gamma \in \Gamma$ and $\alpha \in \Theta$,
\begin{equation}\label{eqanosovgap}
    \frac{1}{a} d_{\mathbb{D}}(b_0, \gamma b_0) - A \le \alpha(\kappa(\rho(\gamma))) \le a d_{\mathbb{D}}(b_0, \gamma b_0) + A.
\end{equation}

{\color{blue}As the critical exponent of the series $d_{\mathbb{D}}(b_0, \gamma b_0),\gamma\in \Gamma$ is finite (for example, see \refchange{21}{Beardon \cite{beardonconvergence} and Patterson \cite{Patterson1976}})} , $\delta^{\phi}_\rho(\Gamma) < \infty$ for all $\phi \in \mathfrak{H}_\Theta^+$.

Our main rigidity result for critical exponents is as follows:

\begin{theorem}\label{thmA}
Let $\Gamma \subset \mathsf{PSL}(2,\mathbb{R})$ be a \refchange{22}{{\color{blue}non-elementary discrete subgroup}}, and let $\rho: \Gamma \to G$ be a $\Theta$-positive representation. Then for all $\alpha \in \Theta$:
\begin{enumerate}
    \item $\delta^{\alpha}_{\rho}(\Gamma) \le 1$.
    \item If $\Gamma$ is a lattice, then $\delta^{\alpha}_{\rho}(\Gamma) = 1$.
    \item If $\Gamma$ is geometrically finite but not a lattice, then $\delta^{\alpha}_{\rho}(\Gamma) < 1$.
\end{enumerate}
\end{theorem}

For closed surface groups, Potrie–Sambarino \cite[Theorem B]{potrie2017eigenvaluesentropyhitchinrepresentation} first established that Hitchin representations satisfy \(\delta^\alpha_\rho(\Gamma) = 1\) for every simple root \(\alpha\). This result was later extended by Pozzetti–Sambarino–Wienhard \cite{Liplim} to a broader class of Anosov representations under Lipschitz regularity assumptions on the limit set, establishing critical exponent rigidity for both maximal and \(\Theta\)-positive representations into \(\mathsf{SO}(p,q)\).

For geometrically finite \refchange{23}{{\color{blue}$\Gamma\subset \mathsf{PSL}(2,\mathbb{R})$}}, Canary–Zhang–Zimmer \refchange{24}{{\color{blue}\cite[Theorem 8.1 and Proposition 11.2]{CaZhZi}}} adapted techniques from Pozzetti-Sambarino-Wienhard~\cite{pozzetti2020conformalityrobustclassnonconformal} to show that, for cusped Hitchin representations, the critical exponent coincides with the Hausdorff dimension of the limit set. In particular, this dimension is bounded from above by one, with equality if and only if \(\Gamma\) is a lattice.

Our approach to proving Theorem~\ref{thmA} generalizes the strategy of Canary–Zhang–Zimmer~\cite{CaZhZi}, who utilized inner and outer radius estimates for shadows from Pozzetti–Sambarino–Wienhard~\cite{pozzetti2020conformalityrobustclassnonconformal} in the context of \((1,1,2)\)-hypertransverse representations from projectively visible subgroups. In the setting of \(\Theta\)-positive representations, however, the limit map for the boundary root {\color{blue} (see the discussion between Theorem \ref{thm7} and Theorem \ref{thm8})} fails to be \((1,1,2)\)-hypertransverse in general. \refchangenew{25}{\changednew{For instance, the second exterior power of the maximal representation in Example~\ref{ergodicitycounterexample} is not $2$-Anosov. Indeed, let $\alpha'$ denote the unique simple root of $\mathfrak{sl}(2,\mathbb{R})$. The gap between the second and third logarithmic singular values of the exterior square representation is given by
\[
|\alpha'(\kappa(\rho_1(\gamma)))-\alpha'(\kappa(\rho_2(\gamma)))|.
\]
Thus, if the exterior square representation were $2$-Anosov, then $\rho$ in Example~\ref{ergodicitycounterexample} would be Anosov with respect to all simple roots. By a result of Davalo~\cite[Theorem 1.1]{Colin}, this implies that $\rho$ is Hitchin, which in turn implies that any deformation of $\rho$ is Hitchin. However, $\rho$ deforms to $\rho_1\oplus \rho_1$, which is not Hitchin, yielding a contradiction.
}}To address this issue, we establish the requisite inner and outer radius estimates by exploiting the monotonicity property of the \refchange{26}{{\color{blue}limit map}} (See Theorem~\ref{thm12}).

\vspace{3mm}

The convexity of $e^{-s}$ implies that for any $0 \le t \le 1$,
\[
t Q^{\phi_1}_{\rho(\Gamma)}(s) + (1-t) Q^{\phi_2}_{\rho(\Gamma)}(s) \ge Q^{t\phi_1 + (1-t)\phi_2}_{\rho(\Gamma)}(s).
\]
Moreover, for any $c > 0$, we have $\delta^{c\phi}_\rho(\Gamma) = \frac{1}{c} \delta^\phi_\rho(\Gamma)$. Applying these observations with Theorem \ref{thmA}, we obtain:

\begin{corollary}\label{corG}
Let $\Gamma \subset \mathsf{PSL}(2,\mathbb{R})$ be a  \refchange{22}{{\color{blue} non-elementary discrete subgroup }}. If $\rho: \Gamma \to G$ is a $\Theta$-positive representation, then for any $\phi \in \mathfrak{H}_\Theta^+$,
\[
\delta^{\phi}_{\rho}(\Gamma) \le \frac{1}{a(\phi)}.
\]
If $\Gamma$ is geometrically finite and the equality is attained, then $\Gamma$ is a lattice.
\end{corollary}

\refchange{27}{{\color{blue}When $G=\mathsf{PSL}(d,\mathbb{R})$ and $\Theta = \Delta$, it was shown by Potrie–Sambarino~{\color{blue}\cite{potrie2017eigenvaluesentropyhitchinrepresentation}} (for $\Gamma$ a closed surface group) and Canary–Zhang–Zimmer~{\color{blue}\cite{CaZhZi}}}} (for $\Gamma$ geometrically finite) that equality in the critical exponent bound characterizes \refchange{28}{{\color{blue}representations whose image lies in some irreducible copy of $\mathsf{PSL}(2,\mathbb{R})$. That is, if $\delta^{\phi}_{\rho}(\Gamma) = \frac{1}{a(\phi)}$ for some $\phi = \sum_{\alpha \in \Delta} c_\alpha \alpha$ with $c_\alpha > 0$ for all $\alpha \in \Delta$, then $\rho(\Gamma)$ lies in an irreducible copy of $\mathsf{PSL}(2,\mathbb{R})$ in $\mathsf{PSL}(d,\mathbb{R})$. 
}}It is an interesting question whether one can show a similar strong rigidity result for $\Theta$-positive representations.

\subsection{Ergodicity}

Suppose $\Gamma \subset \mathsf{PSL}(2,\mathbb{R})$ is a discrete subgroup with $\Lambda(\Gamma) = \partial \mathbb{D} = S^1$, and let $\rho: \Gamma \to G$ be a $\Theta$-positive representation with limit map $\xi^{\Theta}$. For each $\alpha \in \Theta$, let $\xi^{\alpha}: \Lambda(\Gamma) \to \mathcal{F}_{\alpha}$ be the corresponding sub-limit map. Fix a Riemannian metric $d_{\alpha}$ on $\mathcal{F}_{\alpha}$ \refchange{30}{{\color{blue}(note that $\mathcal{F}_{\alpha}$ is compact, so any two Riemannian metrics on $\mathcal{F}_{\alpha}$ are bi-Lipschitz equivalent; hence, the choice of $d_{\alpha}$ does not affect the remaining arguments).}}
Proposition \ref{prop3} shows that for each $\alpha \in \Theta$, the image $\xi^{\alpha}(\Lambda(\Gamma))$ is a rectifiable curve, so we can define a Lebesgue measure $m_{\alpha}$ on $\xi^{\alpha}(\Lambda(\Gamma))$, which is quasi-invariant under the $\rho(\Gamma)$ action.

\begin{theorem}\label{ergodictheorem}
If $\Gamma \subset \mathsf{PSL}(2,\mathbb{R})$ is a lattice, and $\rho: \Gamma \to G$ is a $\Theta$-positive representation, then for any $\alpha \in \Theta$, {\color{red}the dynamical system $(\rho(\Gamma), m_{\alpha})$ has at most $D(\mathfrak{g},\alpha)$ ergodic components, where $D(\mathfrak{g},\alpha)$ is a constant depending only on the Lie algebra $\mathfrak{g}$ and the root $\alpha$.}
\end{theorem}

{\color{red}
\begin{example}\label{ergodicitycounterexample}
Assume $\Gamma$ is torsion-free and let $\rho_1, \rho_2 : \Gamma \to \mathsf{SL}(2,\mathbb{R})$ be two non-conjugate Fuchsian representations, written as
\[
\rho_i(\gamma) =
\begin{pmatrix}
a_i(\gamma) & b_i(\gamma) \\
c_i(\gamma) & d_i(\gamma)
\end{pmatrix}, 
\quad i = 1,2.
\]
Consider the product representation
\[
\rho = \rho_1 \times \rho_2 :
\Gamma \longrightarrow 
\mathsf{SL}(2,\mathbb{R}) \times \mathsf{SL}(2,\mathbb{R}) 
\longrightarrow \mathsf{Sp}(4,\mathbb{R}),
\]
given explicitly by
\[
\rho(\gamma) =
\begin{pmatrix}
a_1(\gamma) & & b_1(\gamma) & \\
& a_2(\gamma) & & b_2(\gamma) \\
c_1(\gamma) & & d_1(\gamma) & \\
& c_2(\gamma) & & d_2(\gamma)
\end{pmatrix},
\]
which is maximal \refchangenew{31}{\changednew{(see Burger, Iozzi, Labourie and Wienhard~\cite[Example 3.9]{maximal})}}, and thus $\Theta$-positive, where $\Theta$ is the long root in $\mathfrak{sp}(4,\mathbb{R})$.

Let $m'_1, m'_2, m'$ denote the pullbacks to $\Lambda(\Gamma)$ of the Lebesgue measures on the respective limit curves associated to $\rho_1, \rho_2$, and $\rho$. Then we can check that (see Appendix~\ref{append3}) $m'_1$ and $m'_2$ are mutually singular ergodic measures, while $m'$ and $m'_1 + m'_2$ are mutually absolutely continuous. Consequently, $m'$ is not ergodic, and decomposes into two distinct ergodic components corresponding to $m'_1$ and $m'_2$. 

Moreover, we observe that $D(\mathfrak{sp}(4,\mathbb{R}),\alpha) = 2$ (see also Appendix~\ref{append3}). Hence, this example demonstrates that our estimate is optimal.

\end{example}

}

{\color{blue}
\begin{remark}\label{psdiff}
\refchange{31}{Assume $\Gamma\subset \mathsf{PSL}(2,\mathbb{R})$ is a lattice and $\rho:\Gamma \to \mathsf{PGL}(d,\mathbb{R})$ is a projective relatively Anosov representation whose $1$-limit set is a Lipschitz manifold. Let $m$ be the Lebesgue measure on the $1$-limit set. Then there exists an $m$-measurable and $\rho$-equivariant map from the $1$-limit set to the $(1,2)$-flag manifold, sending almost every point to the flag consisting of the point and the tangent line through it.

Pozzetti–Sambarino–Wienhard~\cite[Section~6.1]{Liplim} showed that $m$ can be chosen (within the same measure class) so that the transformation cocycle of $m$ under the $\rho(\Gamma)$ action is the partial Iwasawa cocycle on the $(1,2)$-flag manifold composed with the above measurable embedding of the $1$-limit set. However, one should note that the possible discontinuity of the tangent lines may cause the transformation cocycle to be discontinuous with respect to the limit curve. This prevents us from applying the classical Patterson–Sullivan theory (for example, Dey–Kapovich~\cite{Dey_2022}, Kim-Oh-Wang~\cite{kim_properly_2025} and Canary–Zhang–Zimmer~\cite{CaZhZi3}) whose cocycle is continuous with respect to the limit set. As a result, there is no contradiction in the fact that $m_{\alpha}$ might not be non-ergodic, whereas the classical Patterson–Sullivan measure is ergodic.}
\end{remark}

\begin{remark}
There are two kinds of roots in $\Theta$: boundary roots and non-boundary roots (see the discussion between Theorem~\ref{thm7} and Theorem~\ref{thm8}). One can further check for which roots $\alpha \in \Theta$ the measure $m_{\alpha}$ may be non-ergodic.

If $\alpha \in \Theta$ is a non-boundary root, then $m_{\alpha}$ is indeed ergodic. There are two ways to show this. The first method is to note that for non-boundary roots, the $\alpha$-limit set is a $C^{1}$ curve. This fact is well known for Hitchin representations: Labourie~\cite{labourie2005anosovflowssurfacegroups} showed that the limit curve of a Hitchin representation is a Frenet curve. For other types of $\Theta$-positive structures, see for example Pozzetti–Sambarino–Wienhard~\cite[Corollary~10.4]{Liplim} for $G = \mathsf{SO}(p,q)$. Hence the partial Iwasawa cocycle is continuous, and the ergodicity of $m_{\alpha}$ then follows directly from the classical Patterson–Sullivan theory. The second method uses Definition~\ref{rankofcone} and Theorem~\ref{ergodictheoreminpaper}:  By the definition therein, we have $D(\mathfrak{g},\alpha)=1$ for any non-boundary root. \refchange{31}{So the non-ergodic phenomenon only happens for boundary roots.}
\end{remark}

}

\subsection{Dimension of the Limit Set}\label{ssec11}

We continue to use \( \xi^{\Theta} \) and \( \xi^{\alpha} \) to denote the limit maps \( \Lambda(\Gamma) \to \mathcal{F}_{\Theta}, \mathcal{F}_{\alpha} \), respectively. For any \( \Theta \subset \Delta \), let \( L_{\Theta} \) be the corresponding Levi subgroup and \( S_{\Theta} \) its semisimple part. The Hausdorff dimension is considered with respect to fixed Riemannian metrics \( d_{\alpha} \) on \( \mathcal{F}_{\alpha} \). Denote by \( \Lambda_c(\Gamma) \) the conical limit set of \( \Gamma \). \refchange{32}{{\color{blue} As mentioned above, we can classify the elements of $\Theta$ into boundary roots and non-boundary roots.}}

\begin{theorem}\label{thmC}
Let \( \Gamma \subset \mathsf{PSL}(2,\mathbb{R}) \) be a \changed{non-elementary} discrete subgroup, and let \( \rho: \Gamma \to G \) be a \( \Theta \)-positive representation where $G$ is simple.  
\begin{enumerate}
    \item For any non-boundary root \( \alpha \in \Theta\) in the Dynkin diagram , we have
    \[
    \delta^{\alpha}_{\rho}(\Gamma) = \dim\big( \xi^{\alpha}(\Lambda_c(\Gamma)) \big).
    \]
    
    \item If $\Theta\not = \Delta$. For the boundary root \( \alpha_{\Theta} \in \Theta\) in the Dynkin diagram, the following inequalities hold:
    \[
    \frac{r}{2} \delta^{\omega'_{\alpha_{\Theta}}}_{\rho}(\Gamma) \le \dim\big( \xi^{\alpha_{\Theta}}(\Lambda_c(\Gamma)) \big) \le \delta^{\alpha_{\Theta}}_{\rho}(\Gamma) \le 1,
    \]
    where \( \omega'_{\alpha_{\Theta}} \) is the fundamental weight corresponding to \( \alpha_{\Theta} \), considered as a root of \( S_{\Theta - \alpha_{\Theta}} \), and \( r>1\) is the real rank of \( S_{\Theta - \alpha_{\Theta}} \) \refchangenew{35}{\changednew{(see Theorem~\ref{thm8})}}.
    
    \item Moreover, if the representation is transverse not only with respect to \( \Theta \) but also with respect to the adjacent roots in the Dynkin diagram, then for any \( \alpha \in \Theta \),
    \[
    \dim\big( \xi^{\alpha}(\Lambda_c(\Gamma)) \big) = \delta^{\alpha}_{\rho}(\Gamma).
    \]
\end{enumerate}
\end{theorem}

\begin{remark}
It remains an interesting question whether the equality 
\[
\dim\big( \xi^{\alpha}(\Lambda_c(\Gamma)) \big) = \delta^{\alpha}_{\rho}(\Gamma)
\]
holds in general. If not, how might one construct a counterexample? Furthermore, is there a sharper estimate available in the general case?
\end{remark}

\vspace{3mm}

\textbf{Comparison with Previous Works:}

Our paper partially generalizes the argument of Canary–Zhang–Zimmer~\cite{CaZhZi} to all \( \Theta \)-positive representations arising from arbitrary discrete subgroups \( \Gamma \subset \mathsf{PSL}(2,\mathbb{R}) \) to arbitrary simple Lie groups $G$ \refchangenew{34}{\changednew{admitting $\Theta$-positive structure}}. \refchange{34}{{\color{blue}It is worth noting that $G$ need not be $\mathsf{PSL}(d,\mathbb{R})$, and that $\Theta$ may be a proper subset of the set of simple roots.}} \changed{And although our discussion focuses on simple Lie groups $G$, groups admitting $\Theta$-positive structures may also be semisimple. However, all the results in this paper extend directly to part of semisimple cases, for example, provided that $G$ is a direct product of simple Lie groups that admit $\Theta$-positive structures.}

Compared to the work of Pozzetti–Sambarino–Wienhard~\cite{Liplim} containing the critical exponent rigidity of maximal representations and \( \Theta \)-positive representations into \( \mathsf{SO}(p,q) \), a minor advantage of our approach—specifically regarding Theorem \ref{thmA}—is that it does not require any assumption on the algebraicity of \( G \), nor does it involve analyzing the irreducibility or Zariski closure of \( \rho(\Gamma) \). Furthermore, the regular distortion property introduced in Section \ref{sec4} may have broader applications in the study of dynamics of \( \Theta \)-positive representations.

\refchange{35}{{\color{blue}There are also many other works on the Hausdorff dimension of limit sets of Anosov subgroups (see, for instance, Glorieux–Monclair–Tholozan~\cite{glorieux2023hausdorffdimensionlimitsets}, Dey–Kapovich~\cite{Dey_2022}, Dey–Kim–Oh~\cite{dey2024ahlforsregularitypattersonsullivanmeasures}, Pozzetti–Sambarino–Wienhard~\cite{pozzetti2020conformalityrobustclassnonconformal}, and Kim–Minsky–Oh~\cite{kim2023hausdorffdimensiondirectionallimit}). Similar estimates to those in our Theorem~\ref{thmC} have appeared in the first three papers. Glorieux–Monclair–Tholozan~\cite{glorieux2023hausdorffdimensionlimitsets} proved that for a projective Anosov subgroup, the Hausdorff dimension of the symmetric limit set with respect to the Riemannian metric is bounded above by the simple root critical exponent and bounded below by the translation length critical exponent. Dey–Kim–Oh~\cite{dey2024ahlforsregularitypattersonsullivanmeasures} proved that the Hausdorff dimension of the limit set with respect to the symmetrized visual pre-metric equals the critical exponent of the symmetrized functional determining the visual pre-metric (see also Dey–Kapovich~\cite{Dey_2022}). As an application, they also obtained similar estimates for the Hausdorff dimension with respect to the Riemannian metric as in Glorieux–Monclair–Tholozan~\cite{glorieux2023hausdorffdimensionlimitsets}: it is bounded above by the critical exponent of the root and bounded below by the critical exponent of the symmetrized fundamental weight. It should be noted that our estimate in Theorem~\ref{thmC} applies to positive representations from general discrete subgroups to general Lie groups with $\Theta$-positive structure, which are not necessarily Anosov or relatively Anosov, and the lower bound we obtain is stronger than in previous works, being $\frac{r}{2}$ times the fundamental weight critical exponent.}}

\vspace{3mm}

\textbf{Outline of the Paper:}

\refchange{36}{{\color{blue}
Section \ref{sec2} is devoted to preliminaries. We review the basic theory of Lie groups and, following Canary–Zhang–Zimmer~\cite{CaZhZi3} and Guichard–Wienhard~\cite{GW1}, introduce the basic theory of transverse representations.

In Section \ref{sec3}, following Guichard–Wienhard~\cite{GW2}, we introduce the basic theory of $\Theta$-positivity. We also prove that $\Theta$-positive representations are $\Theta$-transverse and that the limit curve is rectifiable when $\Lambda(\Gamma) = S^1$.
}}

In Section \ref{sec4}, we establish that the limit map of a \( \Theta \)-positive representation satisfies a \emph{regular distortion property} \refchange{37}{{\color{blue}(see Definition \ref{def13})}}, which plays a central role in the proof of the rigidity of critical exponents. The key technique involves analyzing the action of \( \rho(\Gamma) \) on positive tuples in \( \xi^{\Theta}(\Lambda(\Gamma)) \), using tangent cone approximations of the positive semigroup. As a corollary of the regular distortion property, we obtain several shadow lemmas for \( \Theta \)-positive representations. Theorem \ref{ergodictheorem} is also proved in this Section.

In Section \ref{sec5} we follow strategy in Pozzetti–Sambarino–Winhard~\cite{Liplim} to prove \changed{Theorem \ref{thmA}~(1) and (2).}

Section \ref{sec6} demonstrates the existence of the double of a \( \Theta \)-positive representation \changed{from a geometrically finite $\Gamma$}, thereby completing the proof of the remaining Theorem \ref{thmA}~(3).

Finally, in Section \ref{append2}, we outline the proof of Theorem \ref{thmC}, closely following the arguments in Canary–Zhang–Zimmer~\cite[Section 8]{CaZhZi}.

\vspace{3mm}
\textbf{Acknowledgement:}
The author wishes to express his sincere gratitude to his advisor \maybe{Prof. Tengren Zhang} for introducing him to the subject of this work and for his continuous support throughout its development. His advisor's consistent guidance, as well as his patient and insightful feedback, have been of fundamental importance to the author’s academic growth.

The author also wishes to express sincere thanks to the referee of this paper, who provided a very detailed review and many useful suggestions.

\section{Preliminaries}\label{sec2}

\subsection{Review of Semisimple Lie Groups}\label{ssec1}
We first review some basic theory about semisimple Lie groups, parabolic subgroups and flag manifolds.

\subsubsection{Basic Structure Theory}\label{sssec1}
In this section we always assume that $G$ is a real \changed{non-compact} semisimple Lie group with finitely many components. Denote its Lie algebra by $\mathfrak{g}$ and identity component by $G^{\circ}$. We further assume that $G^{\circ}$ has finite center $Z(G^{\circ})$. Fix a \emph{Cartan decomposition} $\mathfrak{g} = \mathfrak{k}\oplus \mathfrak{p}$ where $\mathfrak{k}$ is the Lie algebra of a maximal compact subgroup $K\subset G$. It can be shown that if $G$ is connected, then $K$ is connected (see Helgason~\cite[Chapter VI: Theorem 1.1]{helgason}). Let $\mathfrak{a}\subset \mathfrak{p}$ be a fixed maximal abelian subalgebra. Let $\Delta\subset \mathfrak{a}^*$ be a choice of positive simple roots \changed{and for each $\alpha\in \Delta$, let $\omega_{\alpha}$ denote its fundamental weight.} Let $\Sigma,\Sigma^{+},\Sigma^{-}\subset \mathfrak{a}^*$ be the set of non-zero roots, positive roots and negative roots \refchange{38}{{\color{blue}respectively}},  let $\mathfrak{k}_0$ be the centralizer of $\mathfrak{a}$ in $\mathfrak{k}$, let $W$ be the restricted Weyl group, and let $\mathfrak{g}_{\alpha}$ be the root subspace for root $\alpha\in\Sigma$. Set $\mathfrak{g}_0 = \mathfrak{a}\oplus \mathfrak{k}_0$, then we have the restricted root space decomposition: 
$$
\mathfrak{g} = \mathfrak{a}\oplus \refchangenew{39}{\changednew{\mathfrak{k}_0}} \refchange{39}{{\color{blue} \oplus}} \bigoplus_{\alpha\in\ \Sigma} \mathfrak{g}_{\alpha} = \bigoplus_{\alpha\in\ \Sigma\cup\{0\}} \mathfrak{g}_{\alpha} .
$$

\vspace{3mm}

Let $\mathfrak{a}^{+}$ denote the open positive Weyl Chamber, i.e. \[\mathfrak{a}^{+}: = \{v\in\mathfrak{a}|\forall \alpha\in \Delta, \alpha(v)>0\}.\] Any $g\in G$ has a (not unique) \emph{Cartan decomposition in $G$} such that $g=k_1 \exp (\kappa(g)) k_2$ where $k_1, k_2\in K$ and $\kappa(g)\in \overline{\mathfrak{a}^+}$. The vector $\kappa(g)$ is determined uniquely and continuously by $g$ and it is called the \emph{Cartan projection}. 

When $G = \mathsf{SL}(d,\mathbb{R})$, we may choose $K$ to be $\mathsf{SO}(d)$ with orthonormal basis $\{e_1,e_2,...,e_d\}$, $\mathfrak{p}$ to be the space of  \refchange{40}{{\color{blue}$\mathbb{R}$-valued $d\times d$}} symmetric matrices, and $\mathfrak{a}$ to be the space of trace-zero diagonal matrices $\{\mathrm{diag}(\lambda_1,...\refchange{41}{{\color{blue},}}\lambda_d)|\sum_{i = 1}^{d}\lambda_i = 0\}$. We may also choose the simple roots to be 
$$
\alpha_i (\mathrm{diag}(\lambda_1,...\refchange{41}{{\color{blue},}}\lambda_d)) = \lambda_i-\lambda_{i+1},1\le i\le d-1
$$
and the positive Weyl chamber is
$$
\mathfrak{a}^+ = \{\mathrm{diag}(\lambda_1,...\refchange{41}{{\color{blue},}}\lambda_d)\in \mathfrak{a}|\forall 1\le i\le d-1\  \lambda_i> \lambda_{i+1}\}.
$$
Moreover, in this case
$$
\kappa(g) = \mathrm{diag}(\log \sigma_1(g),...\refchange{41}{{\color{blue},}} \log \sigma_d(g))
$$
where $\sigma_1(g)\ge \sigma_2(g)\refchange[\baselineskip]{42}{{\color{blue}\ge}}...\ge \sigma_d(g)$ are the singular values of $g\in \mathsf{SL}(d,\mathbb{R})$.

\subsubsection{Parabolic Subgroups}\label{sssec2}
For any non-empty $\Theta\subset \Delta$, we define the corresponding \emph{Levi subalgebra} $\mathfrak{l}_{\Theta}$, \emph{nilpotent subalgebra} $\mathfrak{u}_{\Theta}$ and \emph{parabolic subalgebra} $\mathfrak{p}_{\Theta}$ as follows:
$$
\Sigma_{\Theta}: = \Sigma-\mathrm{Span}(\Delta-\Theta),\quad \Sigma_{\Theta}^{+}: =\Sigma_{\Theta}\cap \Sigma^{+},\quad \Sigma_{\Theta}^{-}: =\Sigma_{\Theta}\cap \Sigma^{-},   $$
$$
\mathfrak{l}_{\Theta} : = \mathfrak{g}_0\oplus \sum_{\alpha\in \mathrm{Span}(\Delta-\Theta)\cap \Sigma^{+}} [\mathfrak{g}_{\alpha}\oplus \mathfrak{g}_{-\alpha}],\quad
\mathfrak{u}_{\Theta} : = \sum_{\alpha\in\Sigma_{\Theta}^{+} } \mathfrak{g}_{\alpha},\quad \mathfrak{p}_{\Theta}: = \mathfrak{l}_{\Theta}\oplus \mathfrak{u}_{\Theta}.
$$
Moreover, we can define the opposite parabolic subalgebra and the opposite nilpotent subalgebra by reversing the sign:
$$
\mathfrak{u}_{\Theta}^{opp} : = \sum_{\alpha\in\Sigma_{\Theta}^{-} } \mathfrak{g}_{\alpha},\quad \mathfrak{p}_{\Theta}^{opp}: = \mathfrak{l}_{\Theta}\oplus \mathfrak{u}_{\Theta}^{opp}.
$$
\vspace{3mm}

\changed{Throughout this paper, for any $g\in G$, we use $Ad_{g}(\cdot)$ to denote the adjoint action of $g$ on the Lie algebra $\mathfrak{g}$, and $\mathsf{c}_g(\cdot)$ to denote the conjugation action by $g$ on the group $G$.}

Define the \emph{parabolic subgroups} $P_{\Theta}: = N_{G}(\mathfrak{p}_{\Theta}) = \{g\in G \refchange{43}{{\color{blue}\mid}} Ad_{g} \mathfrak{p}_{\Theta} = \mathfrak{p}_{\Theta}\}$ and $P_{\Theta}^{opp}: = N_{G}(\mathfrak{p}_{\Theta}^{opp})$. Next define the \emph{Levi subgroup} $L_{\Theta}: = P_{\Theta}\cap P_{\Theta}^{opp}$ with the identity component $L_{\Theta}^{\circ}$, and define $U_{\Theta}$ and $U_{\Theta}^{opp}$ to be the connected \emph{unipotent subgroups} of $G$ with Lie algebra $\mathfrak{u}_{\Theta},\mathfrak{u}_{\Theta}^{opp}$. $P_{\Theta}$, $P_{\Theta}^{opp}$ and $L_{\Theta}$ are all closed subgroups of $G$, and their Lie algebras are $\mathfrak{p}_{\Theta}$, $\mathfrak{p}_{\Theta}^{opp}$, and $\mathfrak{l}_{\Theta}$ respectively. $U_{\Theta}, U_{\Theta}^{opp}$ are also closed, and they are diffeomorphic to their Lie algebras via the exponential map.

\vspace{3mm}

Let $\mathfrak{z}_{\Theta}$ be the center of $\mathfrak{l}_{\Theta}$ and $\mathfrak{s}_{\Theta} = [\mathfrak{l}_{\Theta},\mathfrak{l}_{\Theta}]$. As $\mathfrak{l}_{\Theta}$ is reductive, we have $\mathfrak{l}_{\Theta} = \mathfrak{z}_{\Theta}\oplus \mathfrak{s}_{\Theta}$ and $\mathfrak{s}_{\Theta}$ is semisimple. Moreover $S_{\Theta}:=[L_{\Theta},L_{\Theta}]$ is a semisimple subgroup of $G$ with Lie algebra $\mathfrak{s}_{\Theta}$. Define $\mathfrak{t}_{\Theta} = \{\refchange{44}{{\color{blue}v\in \mathfrak{a}}}|,  \forall \alpha\in (\Delta-\Theta), \alpha(v) = 0\}$ and one can check  $\mathfrak{t}_{\Theta} = \mathfrak{z}_{\Theta}\cap \mathfrak{a}$. Since the Killing form $B:\mathfrak{g}\times \mathfrak{g}\to \mathbb{R}$ restricts to an inner product on $\mathfrak{a}$, we may set $\mathfrak{a}_{\Theta}$ to be the orthogonal complement of $\mathfrak{t}_{\Theta}$ in $\mathfrak{a}$ and $H_{\alpha}$ to be vectors such that $\forall H\in \mathfrak{a}, B(H_{\alpha},H) = \alpha(H)$. We can show that 
 $\mathfrak{a}_{\Theta}$ is a Cartan subspace of $\mathfrak{s}_{\Theta}$ and $\{H_{\alpha}|\alpha\in (\Delta-\Theta)\}$ is a basis of $\mathfrak{a}_{\Theta}$. Furthermore we may choose the positive Weyl chamber of $\mathfrak{a}_{\Theta}$ \refchange{45}{{\color{blue}to be $\{v\in \mathfrak{a}_{\Theta}|\forall \alpha \in (\Delta-\Theta), \alpha(v)>0\}$}}, in which case the set of positive simple roots of $S_{\Theta}$ is the restriction of the roots $(\Delta-\Theta)$ to $\mathfrak{a}_{\Theta}$, and the Dynkin diagram of $S_{\Theta}$ is the sub-diagram of the Dynkin diagram of $G$ induced by the vertices $(\Delta-\Theta)$.
\vspace{3mm}

\refchange{46}{{\color{blue}The adjoint action of $L_{\Theta}^{\circ}$ on $\mathfrak{u}_{\Theta}$ gives a representation from $L_{\Theta}^{\circ}$ to $\mathsf{SL}(\mathfrak{u}_{\Theta})$,}} and thus induces a weight space decomposition \refchange{46}{{\color{blue}with respect to $\mathfrak{t}_{\Theta}$}}:
$$
\mathfrak{u}_{\Theta}=\bigoplus_{\lambda\in P(\mathfrak{u}_{\Theta},\mathfrak{t}_{\Theta})} \mathfrak{u}_{\lambda}
$$
where $P(\mathfrak{u}_{\Theta},\mathfrak{t}_{\Theta})\subset \mathfrak{t}_{\Theta}^{*}$ is the set of weights and $\forall \lambda\in P(\mathfrak{u}_{\Theta},\mathfrak{t}_{\Theta})$, $\mathfrak{u}_{\lambda}$ is the associated weight space. One can show that $P(\mathfrak{u}_{\Theta},\mathfrak{t}_{\Theta}) = \{\beta|_{\mathfrak{t}_{\Theta}}:\beta\in \Sigma^{+}\}$. So we abuse the notation and write $\mathfrak{u}_{\beta}:=\mathfrak{u}_{\lambda}$ if \refchange{47}{{\color{blue}$\beta\in \Sigma^+$}} satisfies $\lambda = \beta|_{\mathfrak{t}_{\Theta}}$. For any $\beta\in \Sigma^{+}$, \refchange{48}{{\color{blue} we have}}
$$
\mathfrak{u}_{\beta} = \bigoplus_{\alpha\in \Sigma,\alpha-\beta \in \mathrm{Span}(\Delta -\Theta)} \mathfrak{g}_{\alpha},
$$
and $\forall \beta\in \Theta$, $\mathfrak{u}_\beta \not = 0$. Similar discussion holds for $\mathfrak{u}_{\Theta}^{opp}$, and for each $\beta\in \Sigma^{-}$ we define
$$
\mathfrak{u}_{\beta}^{opp} = \bigoplus_{\alpha\in \Sigma,\alpha-\beta \in \mathrm{Span}(\Delta -\Theta)} \mathfrak{g}_{-\alpha}.
$$

To avoid confusion, we use $\mathfrak{u}_{\{\alpha\}}$ (respectively $\mathfrak{u}_{\{\alpha\}}^{opp}$) to denote the nilpotent subalgebra for $P_{\alpha}$ (respectively, $P_{\alpha}^{opp}$) and distinguish them with $\mathfrak{u}_{\beta},\mathfrak{u}_{\beta}^{opp}$ above.

The following theorem summarizes basic facts about the weight subspaces (all the below facts \refchange{50}{{\color{blue}carry}} analogously to $\mathfrak{u}_{\Theta}^{opp}$\refchange{51}{{\color{blue}).}}

\begin{theorem}\label{thm6}(Guichard-Wienhard~\cite[Theorem 2.2]{GW2} and  \refchange{52}{{\color{blue} Kostant~\cite{Kostant2010}}})
     \begin{enumerate}
         \item For every $\lambda \in P(\mathfrak{u}_{\Theta},\mathfrak{t}_{\Theta})$, $\mathfrak{u}_{\lambda}$ is an irreducible representation of $L_{\Theta}^{\circ}$.

         \item For every $\lambda,\lambda'\in P(\mathfrak{u}_{\Theta},\mathfrak{t}_{\Theta})$, $[\mathfrak{u}_{\lambda},\mathfrak{u}_{\lambda'}] = \mathfrak{u}_{\lambda+\lambda'}$.

         \item The Lie algebra $\mathfrak{u}_{\Theta}$ is generated by $\mathfrak{u}_{\beta},\beta\in \Theta$. 
     \end{enumerate}
\end{theorem}

\vspace{3mm}

\subsubsection{Adapted Representation}\label{sssec4}

We introduce the \emph{adapted representation}, \changed{which will later be used to linearize transverse representations.} The idea first appeared in Guichard--Wienhard~\cite[Remark 4.12]{GW1} \changed{and we follow the construction of Gueritaud--Guichard--Kassel--Wienhard~\cite{Gu_ritaud_2017}}. We call such representations ``adapted representations,'' following Canary--Zhang--Zimmer~\cite[Appendix B]{CaZhZi2}.

\refchange{53,54}{
{\color{blue}{
\begin{lemma}{\cite[Lemma 3.2 and Lemma 3.3]{Gu_ritaud_2017}}\label{lem4}
Fix $\Theta\subset \Delta$. For any positive combination of $\Theta$-fundamental weights $\chi = \sum_{\alpha\in \Theta} n_{\alpha}\omega_{\alpha}, n_{\alpha}\in \mathbb{Z}^{> 0}$, 
there exists an irreducible representation 
\[
\psi_{\Theta,\chi}: G\to \mathsf{SL}(V_{\Theta,\chi})
\] 
with the following properties:
\begin{enumerate}

    \item $\psi_{\Theta,\chi}$ is $1$-proximal, and its highest weight is $N\chi$ for some $N\in \mathbb{Z}_{>0}$. 
More precisely, $\psi_{\Theta,\chi}(\exp(\mathfrak{a}))$ consists of diagonal matrices, and if $H\in \mathfrak{a}^+$, then the entries of $\psi_{\Theta,\chi}(\exp(H))$ are non-decreasing, with the top entry being the unique maximum, namely $\exp\!\big(N\chi(H)\big)$.

    \item For all $g\in G$, the first simple root of $\psi_{\Theta,\chi}(g)$ in $\mathsf{SL}(V_{\Theta,\chi})$ is 
    \[
    \min_{\alpha \in \Theta} \alpha(\kappa(g)).
    \]

    \item Let $\eta_{\Theta,\chi}\subset V_{\Theta,\chi}$ be the weight space of the highest weight, and let $\eta_{\Theta,\chi}^*$ be the sum of all other weight spaces. Then:  
    \begin{enumerate}
        \item $\dim \eta_{\Theta,\chi} = 1$.  
        \item $\mathsf{Stab}_{G}(\eta_{\Theta,\chi}) = P_{\Theta},\qquad 
        \mathsf{Stab}_{G}(\eta_{\Theta,\chi}^*) = P_{\Theta}^{opp}$.
    \end{enumerate}
\end{enumerate}
\end{lemma}

For each $\Theta\subset \Delta$, set $\chi_{\Theta} = \sum_{\alpha \in \Theta} \omega_{\alpha}$,
we adopt the shorthand notation:
\[
\psi_{\Theta,\chi_\Theta} = \psi_{\Theta}, \quad 
V_{\Theta,\chi_\Theta} = V_{\Theta}, \quad 
\eta_{\Theta,\chi_\Theta} = \eta_{\Theta}, \quad 
\eta_{\Theta,\chi_\Theta}^* = \eta_{\Theta}^*.
\]
}}}

\vspace{3mm}

From the min--max description, we have the following properties of singular values:

\begin{lemma} \label{lem5}
For any $A,B\in \mathsf{SL}(d,\mathbb{R})$ and any $1\le k\le d$,
\[
\sigma_k(A)\,\sigma_1(B)\ \ge\ \sigma_k(AB)\ \ge\ \sigma_k(A)\,\sigma_d(B),
\]
\[
\sigma_k(A)\,\sigma_1(B)\ \ge\ \sigma_k(BA)\ \ge\ \sigma_k(A)\,\sigma_d(B).
\]
\end{lemma}

For a proof, see for example Canary~\cite[Lemma~29.1]{AnosovNotes}.

Lemma~\ref{lem4} shows that one can realize the simple roots as singular value gaps, so we can generalize Lemma~\ref{lem5} to obtain a uniform continuity statement for the Cartan projection. 

\begin{corollary}{\refchange{55}{{\color{blue}(Benoist~\cite[Section~4.6]{Benoistasym} and \cite[Section~5.1]{benoist2}).}}}\label{cor1}
If $E\subset G$ is a compact subset, then there exists a constant $C>0$ depending only on $E$ such that for any $\alpha\in\Delta$, $g\in G$, and $h\in E$, we have
\[
\big|\alpha(\kappa(\rho(hg))) - \alpha(\kappa(\rho(g)))\big| < C,
\qquad
\big|\alpha(\kappa(\rho(gh))) - \alpha(\kappa(\rho(g)))\big| < C.
\]
\end{corollary}

\subsubsection{Flag Manifold}\label{subsecflag}

For a general Lie group $G$, let $Inn(\mathfrak{g})\subset Aut(\mathfrak{g})$ be the subgroup of inner automorphisms, i.e., the identity component of $Aut(\mathfrak{g})$. We define the $\Theta$-flag manifold $\mathcal{F}_{\Theta}$ to be the $Inn(\mathfrak{g})$ conjugacy classes of $\mathfrak{p}_{\Theta}$ in $\mathfrak{g}$ and define the opposite $\Theta$-flag manifold $\mathcal{F}_{\Theta}^{opp}$ to be the $Inn(\mathfrak{g})$ conjugacy classes of $\mathfrak{p}_{\Theta}^{opp}$ in $\mathfrak{g}$.

We recall two decompositions of parabolic subgroups with the additional assumption: If $\psi:G\to Aut(\mathfrak{g})$ is the adjoint representation, then we further assume that $\psi(G) = Inn(\mathfrak{g})$ (the reason is that we want to make $G$ reductive in the sense of Knapp~\cite[Chapter VII: Section 2]{knapp}). Note that if $G$ is connected, then the assumption holds.

Let $M_{\Theta} = Z_{G}^{0}(\mathfrak{t}_{\Theta})\cap K$ where $Z_{G}^{}(\mathfrak{t}_{\Theta})$ is the $G-$centralizer of $\mathfrak{t}_{\Theta}$ and $Z_{G}^{0}(\mathfrak{t}_{\Theta})$ is the co-split component of $Z_{G}^{}(\mathfrak{t}_{\Theta})$ (see Knapp~\cite[Proposition 7.27]{knapp}).

\begin{theorem}\label{thm1} 

We have the decomposition $ P_{\Theta}: = M_{\Theta} \exp{\mathfrak{a}} U_{\Delta}$. Moreover, it is direct, which means:
$$
M_{\Theta} \times \refchange{56}{{\color{blue}(\exp{\mathfrak{a}})}} \times U_{\Delta}\to P_{\Theta}\quad (m,a,u)\mapsto mau
$$
is a diffeomorphism.
\end{theorem}
For the proof of the equality, see Knapp~\cite[\refchange{57}{{\color{blue}Propositions}} 7.82, 7.83]{knapp}, and the directness comes from the Iwasawa decomposition (see Helgason~\cite[Ch.VI, Theorem 5.1]{helgason}).

Under the assumption that $\psi(G) = Inn(\mathfrak{g})$, we have $\mathcal{F}_{\Theta} = G/P_{\Theta}$ and $\mathcal{F}_{\Theta}^{opp} = G/P_{\Theta}^{opp}$, \refchange{58}{{\color{blue} and we now state two standard results}}.

\begin{corollary} \label{lem1}
    If $\Theta_{1}\subset \Theta_2$, then there is a natural $G$ equivariant projection $\pi^{\Theta_2}_{\Theta_1}:\mathcal{F}_{\Theta_2}\to \mathcal{F}_{\Theta_1}$. 
\end{corollary}

\begin{proof}
We show that $P_{\Theta_2}\subset P_{\Theta_1}$ so we can well define $\pi^{\Theta_2}_{\Theta_1}(gP_{\Theta_2}):=gP_{\Theta_1}$. By definition $\Theta_1\subset \Theta_2$ implies that $\mathfrak{t}_{\Theta_1}\subset \mathfrak{t}_{\Theta_2}$. And $M_{\Theta_i} = Z^{0}_G(\mathfrak{t}_{\Theta_i})\cap K$ shows that $M_{\Theta_2}\subset M_{\Theta_1}$ (see Knapp~\cite[Proposition 7.27]{knapp}), so Theorem \ref{thm1} implies that $P_{\Theta_2}\subset P_{\Theta_1}$.
\end{proof}

\begin{corollary} \label{lem2}
For any $\Theta\subset \Delta$, $\mathcal{F}_{\Theta}$ is compact.
\end{corollary}

\begin{proof}
The Iwasawa decomposition (Helgason~\cite[Ch.VI, Theorem 5.1]{helgason}) shows that we can write $G = K \exp{\mathfrak{a}} U_{\Delta}$, which means that for any $g\in G$, there exists a unique triple $(k(g),a(g),n(g))\in K \times \exp{\mathfrak{a}}\times U_{\Delta} $ such that $g = k(g) a(g) n(g)$. As a result
$$
\mathcal{F}_{\Theta} = K\exp{\mathfrak{a}} U_{\Delta}/M_{\Theta} \exp{\mathfrak{a}} U_{\Delta} = K \exp{\mathfrak{a}} U_{\Delta}/U_{\Delta} \exp{\mathfrak{a}} M_{\Theta}=   K/M_{\Theta}. 
$$
(The second equality is by $P_{\Theta} = P_{\Theta}^{-1}$).

$M_{\Theta} = Z_{G}^{0}(\mathfrak{t}_{\Theta})\cap K$ shows that it is a closed subgroup of the compact group $K$(see \cite[Proposition 7.27]{knapp}), so $K/M_{\Theta}$ is a compact manifold.
\end{proof}

Second we recall the global \emph{Levi decomposition} of $P_{\Theta}$(which is also true for $P_{\Theta}^{opp}$). \changed{Although Warner \cite{warner_harmonic_1972} states the result for connected semisimple Lie groups, the assumption $\psi(G) = Inn(\mathfrak{g})$ ensures that the decomposition remains valid.}

\begin{proposition}{\cite[Proposition 1.2.4.14]{warner_harmonic_1972}}\label{Levi}
For any parabolic subgroup $P_{\Theta}$, we have a semidirect product decomposition
\[
P_{\Theta} = L_{\Theta} U_{\Theta}.
\]
\end{proposition}

\begin{remark}\label{General Flag}
\refchange{60}{{\color{blue}Recall that $\psi :G\to Aut(\mathfrak{g})$ denotes the adjoint representation and we still assume that $\psi(G) =  Inn(\mathfrak{g})$.}} Under the notations in Section \ref{sssec4}, let $[\eta_{\Theta}]\in \mathbb{P}(V_{\Theta})$ be the projectivization of $\eta_{\Theta}$. \changed{From the definition of parabolic subgroups and Lemma \ref{lem4}~(3)}, we see that $\psi_{\Theta}(g) [\eta_{\Theta}] = [\eta_{\Theta}]$ if and only if $Ad_g(\mathfrak{p}_{\Theta}) = \mathfrak{p}_{\Theta}$. So the flag manifold $\mathcal{F}_{\Theta}$ can be identified as the $\psi_{\Theta}(G)$ orbit of $[\eta_{\Theta}]\in \mathbb{P}(V_{\Theta})$. Corollary \ref{lem2} shows that it is compact, thus closed. So $\mathcal{F}_{\Theta}$ is a closed, smooth (indeed real analytic) and homogeneous projective manifold. Moreover the assumption $\psi(G) = Inn(\mathfrak{g})$ shows that $\mathcal{F}_{\Theta}$ is connected. Similar discussions hold for $\mathcal{F}_{\Theta}^{opp}$.
\end{remark}

   We further discuss the situation when $\psi(G) \not =  Inn(\mathfrak{g})$. As $Aut(\mathfrak{g})$ is an algebraic group, it has finitely many components as cosets of $Inn(\mathfrak{g})$. And there is a canonical action of $Aut(\mathfrak{g})$ on the Dynkin diagram of the restricted root system $\Delta$ (Guichard-Wienhard~\cite[Section 5.3]{GW2}). Let $Aut_1(\mathfrak{g})$ be the subgroup which acts trivially. Clearly $Inn(\mathfrak{g})\subset  Aut_1(\mathfrak{g})$.

   We say $(x,y)\in \mathcal{F}_{\Theta}\times \mathcal{F}_{\Theta}^{opp}$ is \emph{transverse} if there exists $g\in Inn(\mathfrak{g})$ such that $(x,y) = g(\mathfrak{p}_{\Theta},\mathfrak{p}_{\Theta}^{opp})$. Note that the choice of $g$ is arbitrary up to a right multiplication of $\psi(L_{\Theta})\cap Inn(\mathfrak{g})$.

\begin{lemma}\label{l1}
$Aut_1(\mathfrak{g})$ acts on $\mathcal{F}_{\Theta}$ and $\mathcal{F}_{\Theta}^{opp}$. Moreover, for any $g\in Aut_1(\mathfrak{g})$, $g(\mathfrak{p}_{\Theta},\mathfrak{p}_{\Theta}^{opp})$
is still a transverse pair in $\mathcal{F}_{\Theta}\times \mathcal{F}_{\Theta}^{opp}$.   
\end{lemma}

\begin{proof} 

 Recall that any two Cartan decompositions of $\mathfrak{g}$ are conjugate to each other under $Inn(\mathfrak{g})$. For a fixed Cartan decomposition $\mathfrak{g} = \mathfrak{k}\oplus \mathfrak{p}$, any two Cartan subalgebras in $\mathfrak{p}$ are conjugate to each other under the maximal compact subgroup $K_0\subset Inn(\mathfrak{g})$ whose Lie algebra is $\mathfrak{k}$ (Helgason~\cite[Chapter III: Theorem~7.2 and Chapter V: Lemma~6.3]{helgason}). Furthermore, for a fixed Cartan decomposition $\mathfrak{g} = \mathfrak{k}\oplus \mathfrak{p}$ and a fixed Cartan subalgebra $\mathfrak{a}\subset \mathfrak{p}$, the Weyl group $W$ acts transitively on Weyl chambers. 
 
The $Inn(\mathfrak{g})$-conjugation sending $g(\mathfrak{k},\mathfrak{p},\mathfrak{a}^+)$ to $(\mathfrak{k},\mathfrak{p},\mathfrak{a}^+)$ implies that $g(\mathfrak{p}_{\Theta},\mathfrak{p}_{\Theta}^{opp})$ is $Inn(\mathfrak{g})$-conjugate to \refchange{61}{{\color{blue}$(\mathfrak{p}_{\phi_g(\Theta)},\mathfrak{p}_{\phi_g(\Theta)}^{opp})$}}, where $\phi_g$ is the induced Dynkin diagram automorphism. Since $g \in Aut_1(\mathfrak{g})$, we have $\phi_g(\Theta)=\Theta$, which concludes the proof.

\end{proof}

\subsubsection{A further assumption}\label{assumption} 

By Lemma~\ref{l1}, if $\psi(G)\subset Aut_1(\mathfrak{g})$, then $G$ still acts on $\mathcal{F}_{\Theta}$ \changed{(this action is clearly transitive)}, so we can continue to identify $\mathcal{F}_{\Theta}$ with $G/P_{\Theta}$ \refchange{62}{{\color{blue}(recall that all discussions prior to Remark~\ref{General Flag} are made under the assumption that $\psi(G)=Inn(\mathfrak{g})$).}} Furthermore, Corollary~\ref{lem1} remains valid: there is always an $Inn(\mathfrak{g})$-equivariant projection $\pi^{\Theta_2}_{\Theta_1}$, and since $\psi(G)\subset Aut_1(\mathfrak{g})$, the proof of Lemma~\ref{l1} shows that $\pi^{\Theta_2}_{\Theta_1}$ is in fact $G$-equivariant. And note that Corollary~\ref{lem2} remains valid as well, since the definition of $\mathcal{F}_{\Theta}$ depends only on $Inn(\mathfrak{g})$.

So in most parts of this paper, we will further assume without loss of generality that $\psi(G)\subset Aut_1(\mathfrak{g})$. This assumption is harmless, as it holds for common examples of semisimple Lie groups.

\changed{And for brevity, given any $g \in G$ and $\Theta \subset \Delta$, we write $g \mathfrak{p}_{\Theta}$ for $\psi(g)P_{\Theta} \in \mathcal{F}_{\Theta}$ whenever $\psi(G) \subset Aut_1(\mathfrak{g})$.}

\subsection{Fuchsian Groups}\label{ssec2}

Let $(\mathbb{D},d_{\mathbb{D}})$ be the Klein model of the hyperbolic disc with Hilbert metric. Let $\Gamma\subset \mathsf{PSL}(2,\mathbb{R})\cong \mathsf{PSO}^+(2,1)=\mathrm{Isom}^+(\mathbb{D},d_{\mathbb{D}})$ be a discrete subgroup \refchange{63}{{\color{blue}(where $\mathrm{Isom}^+(\mathbb{D},d_{\mathbb{D}})$ is the group of orientation preserving isometries)}}. From now on we fix a base point $b_0$ inside $\mathbb{D}$, take the boundary of $\mathbb{D}$ in $\mathbb{P}(\mathbb{R}^{2,1})$ and identify $\partial \mathbb{D} = S^1$. We define the \emph{limit set} of $\Gamma$ to be $\Lambda(\Gamma): =\overline{\Gamma b_0}\cap \partial{\mathbb{D}}$. Note that $\Lambda(\Gamma)$ is compact and it does not depend on the choice of $b_0$. We say $\Gamma$ is \emph{non-elementary} if $|\Lambda(\Gamma)|>2$, and in this case $\Lambda(\Gamma)$ is furthermore an uncountable perfect set.

Every $\gamma\in \Gamma$ with \refchange{64}{{\color{blue} infinite order}} is  \refchange{65}{{\color{blue}either hyperbolic, meaning that $\gamma$ has a pair of attracting and repelling fixed points (denoted by $\gamma^+$ and $\gamma^-$, respectively) on $\Lambda(\Gamma)$, or parabolic, meaning that $\gamma$ has a unique fixed point on $\Lambda(\Gamma)$. In the parabolic case, we abuse notation by letting both $\gamma^+$ and $\gamma^-$ denote this unique fixed point.}} 

\refchange{66}{{\color{blue}Assume $\gamma_n \in \Gamma$ is a sequence, we have the following convergence properities (the proofs can be found on Bowditch~\cite{Bowditch+1999+23+54} and Canary-Zhang-Zimmer~\cite{CaZhZi}):}}

\begin{lemma}{\refchange{67}{{\color{blue}\cite[Lemma 2.5]{Bowditch+1999+23+54} and \cite[Proposition 3.5]{CaZhZi}.}}}\label{lem3}
    
    Let $\Gamma\subset \mathsf{PSL}(2,\mathbb{R})$ be a non-elementary discrete subgroup. Suppose $\gamma_n(b_0)\to x\in \Lambda(\Gamma)$ and $\gamma_n^{-1}(b_0)\to y\in \Lambda(\Gamma)$, then for any $z\in \Lambda(\Gamma)-\{y\}$, we have $\gamma_n(z)\to x$. The convergence is uniform on compact sets of $\Lambda(\Gamma)-\{y\}$.  Similarly, $\gamma_n^{-1}(z)\to y$ uniformly on compact sets of $\Lambda(\Gamma)-\{x\}$.

    Moreover, if each $\gamma_n$ has infinite order (i.e., it is either parabolic or hyperbolic), then $\gamma_n^{+} \to x$ and $\gamma_n^{-} \to y$. 

    Furthermore, if $x \neq y$, then the sequence $\gamma_n$ is eventually hyperbolic (even without assuming that for $n\gg 1$, $\gamma_n$  has infinite order), with $\gamma_n^{+} \to x$ and $\gamma_n^{-} \to y$.
\end{lemma}

{\color{blue}It should be noted that the ``Moreover'' part is not stated explicitly in the references, so we provide an explanation. Since the case $x \neq y$ has already been addressed in the ``Furthermore'' part, we may assume that $x = y$.

Suppose, for contradiction, that $\gamma_n^{+} \not\to x$. Then there exists a subsequence with $\gamma_n^{+} \to x' \in \Lambda(\Gamma)$, where $x' \neq x$. Nevertheless, we still have $\gamma_n(b_0) \to x$ and $\gamma_n^{-1}(b_0) \to y$. As $x \neq x'$, the sequence $\gamma_n^{+}$ is eventually contained in a compact subset of $\Lambda(\Gamma) - \{x\}$. By the first paragraph of Lemma~\ref{lem3}, this implies that 
\[
\gamma_n^{-1}(\gamma_n^{+}) \to y = x.
\]
On the other hand, $\gamma_n^{-1}(\gamma_n^{+}) = \gamma_n^{+}$, so we obtain $x' = y = x$, a contradiction. Hence $\gamma_n^{+} \to x$.
The argument for $\gamma_n^{-} \to y$ is analogous.}

\vspace{3mm}

{\color{blue}

\refchangenew{76}{\changednew{
A discrete subgroup $\Gamma \subset \mathsf{PSL}(2,\mathbb{R})$ is called  \textbf{\emph{geometrically finite}} if it is finitely generated. And if $\Gamma$ is non-elementary, it is called
\begin{enumerate}
    \item \textbf{a \emph{lattice}} if it admits a fundamental domain of finite area.
    \item \textbf{a \emph{uniform lattice}} if this fundamental domain is compact. 
    \item \textbf{\emph{convex cocompact}} if its convex core
\[
C(\Gamma) := \operatorname{Hull}_{\mathbb{D}}(\Lambda(\Gamma))/\Gamma
\]
is compact, where $\operatorname{Hull}_{\mathbb{D}}$ denotes the geodesic convex hull in $\mathbb{D}$.
\end{enumerate}
It is a standard result that every lattice (whether uniform or not) and every convex cocompact subgroup is geometrically finite. Furthermore, $\Gamma$ is a lattice if and only if it is geometrically finite and its limit set satisfies \(\Lambda(\Gamma)=\partial\mathbb{D}\). And note that every uniform lattice is a convex cocompact subgroup, satisfying $\operatorname{Hull}_{\mathbb{D}}(\Lambda(\Gamma)) = \mathbb{D}$.
}}

If $\Gamma$ is torsion-free, then $\Gamma$ acts freely on $\mathbb{D}$ and $\mathbb{D}/\Gamma$ is a hyperbolic surface. Moreover, $\Gamma$ is a lattice if and only if $\mathbb{D}/\Gamma$ has finite area, $\Gamma$ is a uniform lattice if and only if $\mathbb{D}/\Gamma$ is a closed hyperbolic surface, \refchangenew{76}{\changednew{and $\Gamma$ is convex cocompact if and only if $\mathbb{D}/\Gamma$ is a surface not homeomorphic to a cylinder or disk, and with finitely many genus and funnels.}} \refchange{68}{{\color{blue}
Finally, note that by Selberg's lemma, any geometrically finite subgroup $\Gamma$ admits a finite-index torsion-free subgroup.
}}

\subsection{Transverse and relatively Anosov Representations}\label{ssec3}
We refer to Canary-Zhang-Zimmer~\cite{CaZhZi2} and \cite{CaZhZi3} for a quick introduction to transverse representations from general discrete subgroups $\Gamma\subset \mathsf{PSL}(2,\mathbb{R})$.  We still use the notations and assumptions in \refchange{70}{{\color{blue}Sections}} \ref{sssec1}, \ref{ssec2} and \ref{assumption}. And we \refchange{71}{{\color{blue}denote by }}$\rho:\Gamma\to G$ a representation.

\begin{definition}\label{def2}
We say a pair of maps 
$$
(\xi^{\Theta},\xi^{\Theta,opp}): \Lambda(\Gamma)\to \mathcal{F}_{\Theta}\times \mathcal{F}_{\Theta}^{opp},\quad x \refchange{72}{{\color{blue}\, \mapsto \,}}  (\xi^{\Theta}(x), \xi^{\Theta,opp}(x))
$$
is 
\begin{enumerate}
    \item \textbf{$\rho-$equivariant} \refchange{73}{{\color{blue}if for any $\gamma\in\Gamma$ and $x\in \Lambda(\Gamma)$, we have $\xi^{\Theta}(\gamma x) = \rho(\gamma) \xi^{\Theta}(x)$ and $\xi^{\Theta,opp}(\gamma x) = \rho(\gamma) \xi^{\Theta,opp}(x)$.}}

    \item \textbf{transverse} \refchange{74}{{\color{blue}if for any distinct $x,y\in\Lambda(\Gamma)$, we have $\xi^{\Theta}(x)$ is transverse to $\xi^{\Theta,opp}(y)$.}}

    \item \textbf{strongly dynamics preserving} \refchange{75}{{\color{blue}if $\gamma_n\in \Gamma$ is a sequence and $\gamma_n(b_0)\to x$ and $\gamma_n^{-1} (b_0)\to y$, then 
    \[
\rho(\gamma_n)F \to \xi^{\Theta}(x) \quad \text{for any } F \in \mathcal{F}_{\Theta} \text{ transverse to } \xi^{\Theta,opp}(y),
\]
and
\[
\rho(\gamma_n)^{-1}F^{opp} \to \xi^{\Theta,opp}(y) \quad \text{for any } F^{opp} \in \mathcal{F}_{\Theta}^{opp} \text{ transverse to } \xi^{\Theta}(x).
\]}}
    
\end{enumerate}
\end{definition}

\vspace{3mm}
Let $\Theta\subset \Delta$ and $\Gamma \subset \mathsf{PSL}(2,\mathbb{R})$ be a discrete subgroup.

\begin{definition}\label{Anosov}
     A representation $\rho:\Gamma\to G$ is called \emph{$\Theta$-transverse} if there exists a $\rho-$equivariant, transverse, continuous, and strongly dynamics preserving map $(\xi^{\Theta},\xi^{\Theta,opp}):\Lambda(\Gamma)\to \mathcal{F}_{\Theta}\times \mathcal{F}_{\Theta}^{opp}$. 
\end{definition}

\changed{
Corollary~\ref{lem1} shows that for any $\Theta' \subset \Theta$, there exists a $G$-equivariant projection
\[
\mathcal{F}_{\Theta} \times \mathcal{F}_{\Theta}^{opp} \to 
\mathcal{F}_{\Theta'} \times \mathcal{F}_{\Theta'}^{opp}
\]
which preserves transverse pairs. Consequently, if $\rho$ is $\Theta$-transverse, then it is also $\Theta'$-transverse.
}

\refchange{76}{
{\color{blue}
\begin{remark}\label{convexcocompact}
If $\Gamma$ is convex cocompact, then any $\Theta$-transverse representation  
$\rho: \Gamma \to G$ is precisely a \emph{$\Theta$-Anosov} representation.  
Moreover, if $\Gamma$ is non-elementary and geometrically finite, then $\Theta$-transverse representations coincide with  
\emph{relatively $\Theta$-Anosov} representations. For example, see Kapovich-Leeb-Porti~\cite{kapovich2017anosov} and Canary-Zhang-Zimmer~\cite{CaZhZi2}.
\end{remark}
}}

\changed{We list several properties of $\Theta$-transverse representations. \refchange{79}{{\color{blue}Till the end of this Section, we still assume that $G$ is a non-compact real semisimple Lie group with finitely many components whose identity component has finite centre, $\Gamma\subset \mathsf{PSL}(2,\mathbb{R})$ is a discrete subgroup, $\rho:\Gamma\to G$ is a transverse representation, and we use the assumption and notations in Section \ref{assumption}.}} Although many of the statements in the references are formulated under the assumption that $G$ is connected, if $G$ has only finitely many components, then the subgroup 
\[
\Gamma_0 = \{\gamma \in \Gamma \mid \rho(\gamma) \in G^{\circ}\}
\]
has finite index in $\Gamma$, and thus $\Lambda(\Gamma_0) = \Lambda(\Gamma)$. The restriction $\rho|_{\Gamma_0} : \Gamma_0 \to G^{\circ}$ is still a $\Theta$-transverse representation, and the corresponding properties carry over directly from the finite-index pair $(\Gamma_0,\rho|_{\Gamma_0})$ to the original pair $(\Gamma,\rho)$.}

\vspace{3mm}

$\Theta$-transverse representations have the following \emph{Cartan property}, which implies that $(\xi^{\Theta},\xi^{\Theta,opp})$ is unique \refchange{172}{\changed{and gives a local criterion for transverse representations (see the below item (3))}}. From now on, we refer to it as the \emph{limit map}.

\begin{proposition}{\refchange{80}{{\color{blue}\cite{kapovich2017anosov} and \cite[Proposition~2.3, Proposition~2.6 and Observation~6.1]{CaZhZi3}}}}\label{cartanproperty}

Fix $b_0\in \mathbb{D}$. Let $\Gamma\subset \mathsf{PSL}(2,\mathbb{R})$ be a discrete subgroup and $\rho:\Gamma\to G$ be a $\Theta$-transverse representation with limit map $(\xi^{\Theta},\xi^{\Theta,opp})$. \refchange{77}{{\color{blue}Let $\gamma_n\in \Gamma$ be a sequence.}} If $\gamma_n(b_0)\to x\in \Lambda(\Gamma)$, $\gamma_n^{-1}(b_0)\to y\in \Lambda(\Gamma)$ and $\rho(\gamma_n) = \refchange{78}{{\color{blue}\mu_n \exp(\kappa(\rho(\gamma_n))) \nu _n}}$ is the Cartan decomposition of $\rho(\gamma_n)$, then 
\begin{enumerate}
    \item \refchange[\baselineskip]{79}{{\color{blue}$\mu_n \mathfrak{p}_{\Theta}\to \xi^{\Theta}(x)$}} in $\mathcal{F}_{\Theta}$ \refchange[2\baselineskip]{204}{\changed{and $\nu_n^{-1}\mathfrak{p}_{\Theta}^{opp}\to  \xi^{\Theta,opp}(y)$ in $\mathcal{F}_{\Theta}^{opp}$.}}
    \item $\lim_{n\to \infty} \min_{\alpha\in \Theta} \alpha(\kappa(\rho(\gamma_n)))=\infty.$

\refchange{172}{\color{blue}{
    \item $\rho$ is $\Theta$-transverse if and only if there exists a $\rho$-equivariant, transverse, and continuous map 
\[
(\xi^{\Theta}, \xi^{\Theta,opp}) : \Lambda(\Gamma) \to \mathcal{F}_{\Theta} \times \mathcal{F}_{\Theta}^{opp}
\]
with the following local strongly dynamics-preserving property: There exist open subsets 
$U \subset \mathcal{F}_{\Theta}$ and $U^{opp} \subset \mathcal{F}_{\Theta}^{opp}$ such that for any sequence $\gamma_n\in \Gamma$ as assumed above,
\[
\rho(\gamma_n)F \to \xi^{\Theta}(x) \quad \text{for any } F \in U \text{ transverse to } \xi^{\Theta,opp}(y),
\]
and
\[
\rho(\gamma_n)^{-1}F^{opp} \to \xi^{\Theta,opp}(y) \quad \text{for any } F^{opp} \in U^{opp} \text{ transverse to } \xi^{\Theta}(x).
\]}}
\end{enumerate}
\end{proposition}

\vspace{1mm}
{\color{blue}{
Canary–Zhang–Zimmer \cite{CaZhZi3} showed that the adapted representation can be used to linearize $\Theta$-transverse representations.

For any $\Theta$ and $\chi$ as in Lemma~\ref{lem4}, part~(3) of that lemma yields equivariant embeddings which preserve transverse pairs:
\[
\xi_{\Theta,\chi} : \mathcal{F}_{\Theta} \to \mathbb{P}(V_{\Theta,\chi}), 
\quad g\mathfrak{p}_{\Theta} \mapsto g[\eta_{\Theta,\chi}],
\]
\[
\xi_{\Theta,\chi}^{opp} : \mathcal{F}_{\Theta}^{opp} \to \mathbb{P}(V_{\Theta,\chi})^{*}, 
\quad g\mathfrak{p}_{\Theta}^{opp} \mapsto g[\eta_{\Theta,\chi}^*],
\]
where $\mathbb{P}(V_{\Theta,\chi})$ is identified with the flag manifold associated to the first simple root (say it is $\alpha_1$) of $\mathsf{SL}(V_{\Theta,\chi})$, and $\mathbb{P}(V_{\Theta,\chi})^{*}$ is identified with the corresponding opposite flag manifold, which is equivalently the projective dual space, or, equivalently, the space of hyperplanes in $V_{\Theta,\chi}$.

\begin{proposition}{\cite[Proposition B.1]{CaZhZi3}}\label{linearize}
Let $\Gamma\subset \mathsf{PSL}(2,\mathbb{R})$ be a discrete subgroup and $\rho : \Gamma \to G$ be a $\Theta$-transverse representation.  
For any adapted representation $\psi_{\Theta,\chi}$ as in Lemma~\ref{lem4}, the composed representation 
\[
\psi_{\Theta,\chi} \circ \rho : \Gamma \to \mathsf{SL}(V_{\Theta,\chi})
\]
is $\alpha_1$-transverse with limit maps
\[
\xi_{\Theta,\chi} \circ \xi^{\Theta} 
\quad \text{and} \quad
\xi_{\Theta,\chi}^{opp} \circ \xi^{\Theta,opp}.
\]
\end{proposition}

}}

And when $\Gamma$ is non-elementary and geometrically finite, $\Theta$-transverse representations are relatively $\Theta$-Anosov representations and have the following growth property: We say $g\in \mathsf{SL}(d,\mathbb{R})$ is \emph{weakly unipotent} if the \refchange{81}{{\color{blue}Jordan-Chevally}} decomposition of $g$ has elliptic semisimple part and non-trivial unipotent part. \refchange[\baselineskip]{82}{{\color{blue}For any $1\le k \le d-1$}}, we say $g\in \mathsf{SL}(d,\mathbb{R})$ is \emph{$k-$proximal} if $\lambda_k(g)>\lambda_{k+1}(g)$, where $\lambda_1(g)\ge\lambda_2(g)\refchange[\baselineskip]{83}{{\color{blue}\ge}}...\ge\lambda_d(g)$ are the modulus of $g$'s eigenvalues. \changed{For each $\alpha \in \Theta$, as noted in the discussions following Definition~\ref{Anosov}, relatively $\Theta$-Anosov representations are also relatively $\{\alpha\}$-Anosov. By applying the $\psi_{\{\alpha\},\omega_{\alpha}}$ in Lemma \ref{lem4} and combining Lemma~\ref{lem4}~(2), Proposition~\ref{linearize} and the cited theorem below, we obtain:}

\begin{proposition}{\refchange{84}{{\color{blue}\cite[Theorem 1.1]{CaZhZi2}}}}\label{anosovgap}

Fix $b_0\in \mathbb{D}$. Let $\Gamma\subset \mathsf{PSL}(2,\mathbb{R})$ be a non-elementary and geometrically finite subgroup and $\rho:\Gamma\to G$ be a relatively $\Theta$-Anosov representation.

\begin{enumerate}

    \item  If $\alpha\in \Theta$, then there exists $a>1,A>0$, such that for any $\gamma\in \Gamma$
    $$
\frac{1}{a} d_{\mathbb{D}}(b_0,\gamma(b_0))-A\le \alpha (\kappa(\rho(\gamma)))\le a d_{\mathbb{D}}(b_0,\gamma(b_0))+A.
    $$
    
    \item  $\rho$ is type preserving in the following sense. Let $\psi_{\Theta,\chi}$ be an adapted representation. If $\gamma$ is parabolic, then $\psi_{\Theta,\chi}\circ \rho(\gamma)$ is weakly unipotent. If $\gamma$ is hyperbolic, then $\psi_{\Theta,\chi}\circ \rho(\gamma)$ is $1-$proximal.

\end{enumerate}

\end{proposition}

\vspace{1mm}

{\color{blue}{
Finally, we state a partial coarse additive estimate for simple roots. The following cited lemma in Canary--Zhang--Zimmer~\cite{CaZhZi} is formulated for transverse representations into $\mathsf{PGL}(d,\mathbb{R})$. But for any $\alpha\in \Theta$, by applying the adapted representation $\psi_{\{\alpha\},\omega_{\alpha}}$ in Lemma \ref{lem4} and using Lemma~\ref{lem4}~(2) together with Proposition~\ref{linearize}, one can extend the statement to $\Theta$-transverse representations into general Lie groups in our setting.

For any $x,y\in \mathbb{D}$, we use $[x,y]$ to denote the geodesic segment connecting them.
\begin{lemma}[{\cite[Lemma~6.6]{CaZhZi}}]\label{partialcoarseadditive}
Let $\Gamma \subset \mathsf{PSL}(2,\mathbb{R})$ be a discrete subgroup. Suppose $\rho:\Gamma \to G$ is an $\alpha$-transverse representation. Then, for any $b_0 \in \mathbb{D}$ and $r>0$, there exists $C>0$ such that if $\gamma,\eta \in \Gamma$ and
\[
d_{\mathbb{D}}\big(\gamma(b_0),[b_0,\eta(b_0)]\big) \le r,
\]
then
\[
\alpha\!\big(\kappa(\rho(\eta))\big)
\;\ge\;
\alpha\!\big(\kappa(\rho(\gamma))\big)
\;+\;
\alpha\!\big(\kappa(\rho(\gamma^{-1}\eta))\big)
\;-\; C.
\]
\end{lemma}
}}

\vspace{3mm}

\section{$\Theta$-positivity}\label{sec3}

We recall some basic notation and theorems on $\Theta$-positivity, following Guichard and Wienhard~\cite{GW2}. For full details and illustrative examples, we refer readers to their original paper. Throughout this section, we make the same assumptions on the real \refchangenew{108,119}{\changednew{simple}} Lie group $G$ as at the beginning of Section~\ref{ssec1}: namely, $G$ has finitely many components, its identity component $G^{\circ}$ has finite center, and if $\mathfrak{g}$ is the Lie algebra of $G$ with adjoint representation $\psi:G \to Aut(\mathfrak{g})$, then $\psi(G) \subset Aut_1(\mathfrak{g})$ (see Section~\ref{assumption}).

\changed{Although our discussion focuses on simple Lie groups $G$, according to the original definition by Guichard and Wienhard~\cite{GW2}, groups admitting $\Theta$-positive structures may also be semisimple. However, all the results in this paper extend directly to part of the semisimple case, for example, provided that $G$ is a direct product of simple Lie groups that admit $\Theta$-positive structures.}

\subsection{Definitions and first properties}\label{ssec4}
\subsubsection{Convex cones}
Recall that a \emph{cone} in \refchange{85}{{\color{blue}real}} vector space $V$ is a set such that if $v$ is in it, then $\mathbb{R}^{> 0}v$ is also in it. A \emph{convex cone} $\Omega$ is a cone which satisfies $v,w\in \Omega\implies  v+w\in \Omega$. \refchange{86}{\changed{We call $\Omega$ \emph{acute} if $\overline{\Omega}$ does not contain any non-trivial \refchange{88}{vector subspace of $V$.}}}

We prove some lemmas about convex cones for future use. Let $(V,\|\|)$ be a finite dimensional normed real vector space and $\Omega\subset V$ be an open acute convex cone. Denote $Aut(\Omega)\subset \mathsf{GL}(V)$ the automorphism group preserving $\Omega$. Let $\pi:V-\{0\}\to \mathbb{P}(V)$ be the projection. \refchange{89}{\changed{Let $K'\subset \Omega$ be a sub-cone and $\tilde{K}\subset \Omega$ be a subset}}, we say $K'$ is a \emph{projectively compact sub-cone} if $\pi(K')$ is compact. We say $\tilde{K}\subset \Omega$ is \emph{bounded} if $\tilde{K}$ is bounded in $V$, and \emph{properly bounded} if $\overline{\tilde{K}}$ is compact in $\Omega$ (so it needs to be away from \refchange{90}{\changed{$\partial \Omega$ which includes $\{0\}$}}).

\refchange{103}{\changed{The first lemma's proof is adapted from Burger-Iozzi-Labourie-Wienhard~\cite[Lemma~8.10]{maximal}.}}

{\color{blue}{
\begin{lemma}\label{lem15}
    Let $\Omega\subset V$ be an open acute convex cone, and \refchange{104}{\changed{let $v_i\in \Omega, i = 1,\dots,k$.}}
    
    \begin{enumerate}
        \item There exists a constant $C_0>1$ such that
        \[
        \frac{\sum_{i = 1}^{k} \|v_i\|}{C_0}\le \left\| \sum_{i = 1}^{k} v_i\right\| \le C_0 \sum_{i = 1}^{k} \|v_i\|.
        \]
        
        \item Fix a compact subset $\tilde{K}\subset \Omega$. Then there exists a constant $C>1$ depending on $\Omega$ and $\tilde{K}$ with the following property: if $\sum_{i = 1}^{k} v_i \in \tilde{K}$ then
        \[
        \frac{1}{C}\le \sum_{i = 1}^{k} \|v_i\| \le C.
        \]
    \end{enumerate}
\end{lemma}

\begin{proof}
Let $S\subset V$ be the unit norm sphere. Since the cone is acute, the closed convex hull of $\overline{\Omega}\cap S$ does not contain the origin. By the Hahn-Banach theorem, we can pick a linear functional $l$ such that $\overline{l(\overline{\Omega}\cap S)}\subset (0,\infty)$, so there exists a constant $C'>1$ such that for any $v\in \changednew{\overline\Omega}$, we have:
\[
\frac{1}{C'} \|v\| \le l(v)\le C' \|v\|.
\]

(1) Since all $v_i \in \Omega$, their sum $V_S := \sum_{i=1}^k v_i$ is also in $\Omega$. Applying the above inequality to $V_S$ and to each $v_i$ gives:
\[
\|V_S\| \le C' l(V_S) = C' \sum_{i = 1}^{k} l(v_i) \le C' \sum_{i = 1}^{k} (C' \|v_i\|) = (C')^2 \sum_{i = 1}^{k} \|v_i\|.
\]
In the other direction,
\[
\sum_{i = 1}^{k} \|v_i\| \le \sum_{i = 1}^{k} C' l(v_i) = C' l(V_S) \le C' (C' \|V_S\|) = (C')^2 \|V_S\|.
\]
Setting $C_0 = (C')^2$ proves the first item.

(2) For the second item, by linearity and the first inequality, we have
\[
\frac{1}{C'}l\left(\sum v_i\right) = \frac{1}{C'} \sum l(v_i) \le \sum \|v_i\| \le C' \sum l(v_i) = C' l\left(\sum v_i\right).
\]
As $\sum v_i \in \tilde{K}$ and $\tilde{K}\subset \Omega$ is compact, $l(\sum v_i)$ is bounded from above and below by positive constants. By the inequality above, $\sum \|v_i\|$ is also bounded, which completes the proof.
\end{proof}
}}

\changed{
The application of the Hahn--Banach theorem in the proof of Lemma~\ref{lem15} shows that if 
$\Omega$ is an open acute convex cone, then its projectivization is contained in the affine 
chart $l^{-1}(\mathbb{R}^{>0})$. \refchange{87}{Moreover, the inclusion 
$\overline{l(\Omega \cap S)} \subset (0,\infty)$ implies that the projectivization in this 
affine chart is a bounded domain.}
}

\begin{lemma} \label{lem7}
     Let $\Omega\subset V$ be an open acute convex cone. Fix a compact subset $Q\subseteq Aut(\Omega)$, a projectively compact sub-cone $K'\subset \Omega$ and \refchange{91}{\changed{a basis}} $\{e_i\}$ for $V$. There exists a constant $C>1$ depending only on $Q$, $K'$ and $\{e_i\}$ with the following property: if $g\in Aut(\Omega)$ and $g = k_1 D k_2$ \refchange{92}{\changed{where $D\in Aut(\Omega)$ is a diagonal matrix}} in the basis $\{e_i\}$ and $k_1,k_2\in Q$, then for any $v\in K'$
     $$
     \frac{1}{C} \lambda_{D} \|v\| \le \|g v \|\le C \lambda_{D} \|v\|
     $$
     where $\lambda_{D}$ is the maximal modulus of entries in $D$.
\end{lemma}
\begin{proof}

\refchangenew{94}{\changednew{Let $\|\cdot\|_{O}$ denote the operator norm. Since $g = k_1 D k_2$, where $k_1$ and $k_2$ vary within a compact set, there exists a constant $C_0 > 0$ such that 
\[
\|g\|_O \le C_0 \|D\|_O = C_0 \lambda_D.
\]
Consequently, the inequality $\|gv\| \le C_0 \lambda_D \|v\|$ holds for any $v \in K'$.}} Thus, if the Lemma is false, then there exist \refchange{93}{\changed{sequences $g_n\in Aut(\Omega)$, $k_{1,n},k_{2,n}\in Q$, $D_n\in Aut(\Omega)$ diagonal}} such that $g_n=k_{1,n} D_nk_{2,n}$ and a sequence $v_n\in K'$ such that
\begin{equation}\label{collapseine}
\|g_n v_n\|<\frac{1}{n} \lambda_{D_n} \|v_n\|.
\end{equation}

Note that as $D_n$ preserves the open acute convex cone, all diagonal entries in $D_n$ are positive numbers. So eigenvalues of modulus $\lambda_{D_n}$ are just the eigenvalues $\lambda_{D_n}$.

Assume without loss of generality that $\forall n\in \mathbb{Z}^{>0}, \|v_n\| = 1$. By taking a subsequence, we may assume that $v_n\to v\in K'$,  $k_{1,n}\to k_1\in Q$ and $k_{2,n}\to k_2\in Q$.
\refchangenew{94}
{
\changednew{Let $d = \dim_{\mathbb{R}}(V)$. For every $n \in \mathbb{Z}^{>0}$ and $1 \le j \le d$, let $D_{n,j}$ denote the $j$-th eigenvalue of $D_n$, ordered by non-increasing modulus. After passing to a subsequence, we can assume that for every $1 \le j < d$, the limit
\[
\lim_{n \to \infty} \frac{D_{n,j}}{D_{n,j+1}}
\]
exists (which may be $+\infty$). Due to Inequality~\ref{collapseine}, for at least one $1 \le j < d$, we have
\[
\lim_{n \to \infty} \frac{D_{n,j}}{D_{n,j+1}} = \infty.
\]
Define $j' = \inf\{j \mid \lim_{n \to \infty} \frac{D_{n,j}}{D_{n,j+1}} = \infty\}$. Let $V_{1,n}$ be the span of the first $j'$ eigenvectors of $D_n$, and $V_{2,n}$ be the span of the remaining $d-j'$ eigenvectors of $D_n$. Since $D_n$ is diagonal, $V_{1,n}$ and $V_{2,n}$ vary within a finite set, so after passing to a further subsequence, we can assume that $V_{1,n} = V_1$ and $V_{2,n} = V_2$ are constant linear subspaces.}
}

{\color{blue}\changednew{Note that the above construction implies that $\lim_{n \to \infty} \lambda_{D_n}^{-1} D_n|_{V_1}$ converges to an invertible linear transformation, while $\lim_{n \to \infty} \lambda_{D_n}^{-1} D_n|_{V_2}$ converges to $0$.} So we have for any $v$ transverse to $V_2$ and any fixed $n$, the projectivizations of $(D_n)^{k} v$ will eventually converge to some point in $\mathbb{P}(V_1)\subset \mathbb{P}(V)$ as $k\to \infty$. Similarly, for any $v$ transverse to $V_1$ and any fixed $n$, the projectivizations of $(D_n)^{-k} v$ will eventually \refchange{96}{ be contained in arbitrarily small neighbourhoods of $\mathbb{P}(V_2)\subset \mathbb{P}(V)$ as $k\to \infty$.}

\refchange{95}{The paragraph after the proof of Lemma~\ref{lem15} shows that after projectivization, $\Omega$ is a bounded domain in some affine chart, so the projectivization of $\partial{\Omega}$ is a non-empty compact set.} Moreover, for any $n$, we have $D_n(\Omega) = \Omega$ and hence $D_n(\partial \Omega) = \partial \Omega$.

We show that $V_1 \cap \Omega = \emptyset$, i.e., $\mathbb{P}(V_1)\cap \pi(\Omega) = \emptyset$. Suppose not,  the attracting property above implies that for any fixed $n$ and any $v \in \partial \Omega - V_2$, the projectivization of $(D_n)^k v$ converges to some point in $\mathbb{P}(V_1)$ as $k \to \infty$. Moreover, since $\Omega$ is open and acute, there exists some $v \in \partial \Omega - V_2$ such that for any fixed $n$, $\lim_{k\to \infty} \pi \big((D_n)^k(v)\big)\in \mathbb{P}(V_1)\cap \pi(\Omega) $, which contradicts $D_n(\partial \Omega) = \partial \Omega$. \refchange{97}{Hence we conclude that $\Omega \cap V_1 = \emptyset$.} On the other hand, the attracting property also shows that $(\partial \Omega-\{0\}) \cap V_1 \neq \emptyset$.

Next we show that $\Omega \cap V_2 = \emptyset$. Otherwise, pick $w_1 \in \partial \Omega \cap V_1$ and $w_2 \in \Omega \cap V_2$. Let $\overrightarrow{w_1w_2}^{\circ}$ be the open ray (excluding $w_1$) starting from $w_1$ and passing through $w_2$, and let $w_3 = \overrightarrow{w_1w_2}^{\circ} \cap \partial \Omega$ which is transverse to $V_2$. For each fixed $n$, the projectivization of $(D_n)^k w_3$ converges to a point in $\mathbb{P}(V_1)\subset \mathbb{P}(V)$ as $k \to \infty$, so the sequence of projective segments $(D_n)^k[w_1,w_3]$ converges to a projective segment in $\mathbb{P}(V_1)$. However, the projectivization of $(D_n)^k w_2$ lies in $\mathbb{P}(V_2)$. Since $V_1 \oplus V_2$ is a direct sum and $(D_n)^k w_2$ lies in the interior of $(D_n)^k[w_1,w_3]$, this yields a contradiction. \refchange{97}{Thus $\Omega \cap V_2 = \emptyset$.}
}

\refchange{99}{\changed{Write $k_{2,n} v_n$ as direct decompositions $k_{2,n} v_n = v_{1,n}+v_{2,n}$ where $v_{i,n}\in V_i, i = 1,2$.}} \refchange[\baselineskip]{98}{\changed{Recall that we have assumed $k_{2,n}\to k_2\in Q\subset Aut(\Omega)$ and $v_n\to v\in K'\subset \Omega$}}, \refchange[\baselineskip]{100}{\changed{so $v_{1,n},v_{2,n}$ also converge in $V_1$ and $V_2$.}} As $V_1\cap \Omega =V_2\cap \Omega = \emptyset$, the limits of $v_{1,n}$ and $v_{2,n}$ are non-zero.

$\|g_n v_n\|<\frac{1}{n} \lambda_{D_n} \|v_n\|$ implies that
\begin{equation}\label{collapseine2}
\| k_{1,n} \refchangenew{101,102}{\changednew{(\lambda_{D_n}^{-1} D_n)}}(v_{1,n} +v_{2,n}) \|< \frac{1}{n}.
\end{equation}
\changednew{Note that $(\lambda_{D_n}^{-1} D_n) v_{1,n}$ converges to a non-zero vector $v_1'$ in $V_1$, while $(\lambda_{D_n}^{-1} D_n) v_{2,n}$ converges to zero. Since $k_{1,n}$ also converges, the left-hand side of Inequality~\ref{collapseine2} converges to a non-zero value, which is a contradiction.}
\end{proof}

\subsubsection{$\Theta$-positive structure}

\refchange{105}{\changed{Recall at the beginning of Section \ref{sec3}, we assume that $G$ is a \refchangenew{108,119}{\changednew{simple}} Lie group $G$ with finitely many components with identity component $G^{\circ}$, and the center of $G^{\circ}$ is finite. Moreover if $\mathfrak{g}$ is the Lie algebra and $\psi:G\to Aut(\mathfrak{g})$ is the adjoint representation, then we assume that $\psi(G)\subset Aut_1(\mathfrak{g})$ (see Section \ref{assumption}).}}

\begin{definition}\label{def4}
    Let $G$ be a real \refchangenew{108,119}{\changednew{simple}} Lie group as above and $\Theta\subset\Delta$ be a non-empty subset. We say $G$ admits $\Theta$-positive structure if for any $\beta\in\Theta$, there exists \refchange{106}{\changed{an}} $Ad_{L_{\Theta}^{\circ}}-$invariant non-empty \refchange[\baselineskip]{107}{\changed{open acute convex cone}} $\mathring{c}_{\beta}\subset \mathfrak{u}_{\beta}$, and denote its closure as $c_{\beta}$.
\end{definition}

\begin{remark}
    Equivalently, one can define $\Theta$-positive structure by requiring that for any $\beta\in\Theta$, there exists \refchange{106}{\changed{an}} $Ad_{L_{\Theta}^{\circ}}-$invariant non-empty open acute convex cone $\mathring{c}_{\beta}^{opp}\subset \mathfrak{u}_{\beta}^{opp}$, and denote its closure as $c_{\beta}^{opp}$.
    
     Although the definition requires us to make a choice between $c_{\beta}$ and $-c_{\beta}$ for each $\beta\in\Theta$, Guichard-Wienhard~\cite[Proposition 5.2]{GW2} shows that the choices are unique up to the action of $Aut_1(\mathfrak{g})$.
\end{remark}

Guichard-Wienhard~\cite{GW2} classified all the possible $(G,\Theta)$ that admit $\Theta$-positive structures.

\begin{theorem}[\protect{\cite[Theorem 3.4]{GW2}}]\label{thm7} 

Let $G$ be a simple real Lie group and $\Theta\subset \Delta$ be non-empty. $G$ admits $\Theta$-positive structure if and only if $(G,\Theta)$ falls into one of the four cases:
\begin{enumerate}
    \item $G$ is a split real form and $\Theta = \Delta$.
    \item $G$ is of Hermitian tube type of real rank $r$ and $\Theta = \{\alpha_{r}\}$, where $\alpha_{r}$ is the long simple root.
    \item $G$ is locally isomorphic to $\mathsf{SO}(p,q)$, \refchange{109}{\changed{$2<p<q$}}, and $\Theta = \{\alpha_1,...\refchange{110}{\changed{,}}\alpha_{p-1}\}$, where \refchange[\baselineskip]{111}{\refchangenew{111}{\changednew{$\alpha_i$'s}}} are the long simple roots.
    \item $G$ is locally isomorphic to the real forms of $\mathsf{F}_{4}, \mathsf{E}_6,\mathsf{E}_7$ \refchange[\baselineskip]{112}{\changed{or}} $\mathsf{E}_8$ whose restricted root system is of type $\mathsf{F}_{4}$, and $\Theta =\{\alpha_1,\alpha_2\}$ where $\alpha_1,\alpha_2$ are the two long simple roots.
\end{enumerate}
\end{theorem}

If we are in the first case, that is $\Delta = \Theta$, then $\forall \alpha \in \Delta, \mathfrak{u}_{\alpha} = \mathfrak{g}_{\alpha}$ which is one dimensional, and $c_{\alpha}$ is a copy of $\mathbb{R}^{>0}$ if identify $\mathfrak{g}_{\alpha}$ with $\mathbb{R}$. If $\Delta \not = \Theta$, then Theorem \ref{thm7} shows that $\Theta\cup \refchange{113}{\changed{(\Delta-\Theta)}}$ divides the Dynkin diagram of $G$ into two connected components. Let $\alpha_{\Theta}\in\Theta$ be the unique root connected to $(\Delta-\Theta)$ and we call it the \emph{boundary root}. \refchange{114}{\changed{To be brief we let $\Theta -\alpha_{\Theta}$ denote $\Theta -\{\alpha_{\Theta}\}$}}, and for any $\beta \in \Theta -\alpha_{\Theta}$, $\mathfrak{u}_{\beta}$ is still one dimensional so we call these non-boundary roots \emph{real type}, and the boundary root $\alpha_{\Theta}$ is of \emph{Hermitian tube type} in the following sense:

\begin{theorem}[\protect{\cite[Proposition 3.8 and \changednew{Section 5.1}]{GW2}}]\label{thm8} 
Let $G$ be a simple Lie group with $\Theta$-positive structure, and $\alpha_{\Theta}$ be the boundary root, then:
\begin{enumerate}
    \item $S_{\Theta-\alpha_{\Theta}}$ is of Hermitian tube type. \refchangenew{35}{\changednew{Moreover, if $G$ is not locallly isomorphic to $\mathsf{SL}(2,\mathbb{R})$, then the real rank $r$ of $S_{\Theta-\alpha_{\Theta}}$ is greater than $1$.}}
    \item The maximal parabolic subgroup of $S_{\Theta-\alpha_{\Theta}}$ associated to $\alpha_{\Theta}$ is $P_{\Theta}\cap S_{\Theta-\alpha_{\Theta}}$, and the Levi subgroup is $L_{\Theta}\cap S_{\Theta-\alpha_{\Theta}}$.
    \item $\mathfrak{u}_{\alpha_{\Theta}}$ is the nilpotent subalgebra of $P_{\Theta}\cap S_{\Theta-\alpha_{\Theta}}$.
\end{enumerate} 
\end{theorem}

\subsection{The positive semigroup}\label{ssec5}
We recall \refchange{115}{\changed{a}} parameterization of the $\Theta$-positive semigroup following Guichard-Wienhard~\refchange[\baselineskip]{116}{\changed{\cite[Section 4.3]{GW3} and \cite[Sections 8,9,10]{GW2}.}} 

\subsubsection{$\Theta$-Weyl Group $W(\Theta)$} \label{sssec13}
For every $\beta\in \Delta$, let $s_{\beta}$ be the reflection of $\mathfrak{a}$ about the kernel of the simple root $\beta$.  Let \refchange{117}{\changed{$W=\langle s_{\beta}|\beta\in \Delta\rangle$}} be the \emph{Weyl group} of $G$. \refchangenew{Newchange 1}{\changednew{For any nonempty $\Theta\subset \Delta$, let $W_{\Theta}$ denote the subgroup generated by $\{s_{\beta} \mid \beta\in \Theta\}$. Note that $W_{\Delta-\Theta}$ can be identified with the Weyl group of $S_{\Theta}$, where $S_{\Theta}$ is the semisimple part of $L^{\circ}_{\Theta}$.}}

 Suppose $(G,\Theta)$ has $\Theta$-positive structure. Define \changed{$\sigma_{\alpha_{\Theta}} = \omega_{\{\alpha_{\Theta}\}\cup (\Delta-\Theta)}\cdot \omega_{(\Delta-\Theta)}^{-1}$ where $\omega_{\{\alpha_{\Theta}\}\cup (\Delta-\Theta)}$ and  $\omega_{(\Delta-\Theta)}$} are the longest elements in \refchange{118}{\changed{$W_{\{\alpha_{\Theta}\}\cup (\Delta-\Theta)}$ and $W_{(\Delta-\Theta)}$}}. For the remaining $\beta\in \Theta - \alpha_{\Theta}$, define $\sigma_{\beta} = s_{\beta}$. Define $W(\Theta)$ to be the subgroup of $W$ generated by $\{\sigma_{\beta}|\beta\in \Theta\}$. Guichard-Wienhard~\cite[Corollary 8.15]{GW2} shows that we have the following list of $W(\Theta)$ corresponding to the classification of $(G,\Theta)$:

\begin{enumerate}
    \item If $(G,\Theta)$ is in split real case, then $W(\Theta) \cong  W$.
    \item If $(G,\Theta)$ is in Hermitian Tube type case, then $W(\Theta) \cong W_{\mathsf{A}_1}$.
    \item If $(G,\Theta)$ is in $\mathsf{SO}(p,q)$ case, then $W(\Theta) \cong  W_{\mathsf{B}_{p-1}}$.
    \item If $(G,\Theta)$ is \refchange{120}{\changed{locally isomorphic to the real forms of $\mathsf{F}_{4}, \mathsf{E}_6,\mathsf{E}_7$ or $\mathsf{E}_8$ whose restricted root system is of type $\mathsf{F}_{4}$ and $\Delta \not = \Theta$}}, then $W(\Theta) \cong W_{\mathsf{G}_2}$.
\end{enumerate}
Where $W_{\mathsf{A}_1}$, $W_{\mathsf{B}_{p-1}}$, $W_{\mathsf{G}_2}$ are the Weyl groups of the root systems $\mathsf{A}_1$,  $\mathsf{B}_{p-1}$ and $\mathsf{G}_{2}$ \refchange{121}{\changed{respectively}}. Under this identification, let \refchange{122}{\changed{$\omega_{\Theta}^{0}$}} be the longest element in $W(\Theta)$.

\subsubsection{$U_{\Theta}^{\ge 0}$, $U_{\Theta}^{>0}$ and Parameterization} \label{sssec14}
Now we define the non-negative semigroup $U_{\Theta}^{\ge 0}$ and positive semigroup $U_{\Theta}^{>0}$.

\begin{definition}\label{def7}
    Let $U_{\Theta}^{\ge 0}$ be the semigroup generated by \refchange{123}{\changed{$\{\exp v\mid v\in c_{\beta},~\beta\in \Theta\}$}}, and $U_{\Theta}^{>0}$ be the interior of $U_{\Theta}^{\ge 0}$ (with respect to the topology of $U_{\Theta}$).

    \refchange{124}{\changed{Similarly, let $U_{\Theta}^{opp,\ge 0}$ be the semigroup generated by $\{\exp v\mid v\in c_{\beta}^{opp},~\beta\in \Theta\}$, and $U_{\Theta}^{opp,>0}$ be the interior of $U_{\Theta}^{opp,\ge 0}$ (with respect to the topology of $U_{\Theta}^{opp}$).}}
\end{definition}
\changed{It should be noted that, before defining $U_{\Theta}^{\ge 0}$ and $U_{\Theta}^{opp,\ge 0}$, there is an implicit choice between $c_{\beta}$ and $-c_{\beta}$ (since both are preserved by $L_{\Theta}^{\circ}$). Different choices lead to different identifications of the nonnegative semigroups, but all are conjugate under $Aut_1(\mathfrak{g})$ (see Guichard-Wienhard~\cite[Corollary~5.3]{GW2}).} We can check that $U_{\Theta}^{>0},U_{\Theta}^{opp,>0}$ are \refchange{123}{\changed{semigroups without identity}} too. \changed{They also depend on the choice between $c_{\beta}$ and $-c_{\beta}$. }

Label the simple roots in $\Theta$ by $\beta_i$ for $i = 1,\dots,|\Theta|$. Let $l$ be the smallest number such that there exists a reduced expression $\gamma = (i_1,\dots\refchange{125}{\changed{,}} i_l)$ for $\omega_{\Theta}^{0}$ satisfying
$$
\omega_{\Theta}^{0} = \sigma_{\beta_{i_1}}\refchange[\baselineskip]{126}{\changed{\cdots}}\sigma_{\beta_{i_l}}.
$$
Define 
$$
F_{\gamma}: \refchange{127}{\changed{\mathring{c}_{\beta_{i_1}}\times \cdots\times \mathring{c}_{\beta_{i_l}}}} \to U_{\Theta}^{\ge 0}
$$

$$
(v_{i_1},\cdots, v_{i_l}) \mapsto \exp v_{i_1}\cdots\exp v_{i_l}
$$
and similarly define
$$
F_{\gamma}^{opp}: \mathring{c}_{\beta_{i_1}}^{opp}\times \cdots\times \mathring{c}_{\beta_{i_l}}^{opp} \to U_{\Theta}^{opp,\ge 0}
$$
$$
(v_{i_1},..., v_{i_l}) \mapsto \exp v_{i_1}\cdots\exp v_{i_l}.
$$

We have:
\begin{theorem}{\refchange{128}{\changed{\cite[Theorem 10.1]{GW2}.}}}\label{thm9}

For every reduced expression $\gamma$, $F_{\gamma}$ is a diffeomorphism from $\mathring{c}_{\beta_{i_1}}\times \cdots \times \mathring{c}_{\beta_{i_l}}$ to $U_{\Theta}^{>0}$. Similarly $F_{\gamma}^{opp}$ is a diffeomorphism from $\mathring{c}_{\beta_{i_1}}^{opp}\times \cdots\times \mathring{c}_{\beta_{i_l}}^{opp}$ to $U_{\Theta}^{opp,> 0}$.   
\end{theorem}

Fix a reduced expression $\gamma$ for $\omega_{\Theta}^{0}$ and define $\mathring{c}_{\gamma} := \mathring{c}_{\beta_{i_1}}\times \cdots \times \mathring{c}_{\beta_{i_l}}$. For any $x\in U_{\Theta}^{>0}$, \refchange{129}{\changed{denote by}} $v^{\gamma} _{x} = F_{\gamma}^{-1}(x)\in \mathring{c}_{\gamma}$ the \emph{cone coordinate} of $x$. In the remaining discussion, we will always assume that a reduced expression has been fixed implicitly, unless stated otherwise.
When there is no ambiguity, we omit the $\gamma$ and just call $v_x\in \mathring{c}_{\gamma}$ the cone coordinate of $x$. For any $v\in \mathring{c}_{\gamma}$, $i_k\in \gamma$, denote by $v_{i_k}$ the $\mathring{c}_{\beta_{i_k}}$ factor. Similarly define $\mathring{c}_{\gamma}^{opp} := \mathring{c}_{\beta_{i_1}}^{opp}\times \cdots\times \mathring{c}_{\beta_{i_l}}^{opp}$. For any $x\in U_{\Theta}^{opp,>0}$, denote $v_x^{opp} = {F_{\gamma}^{opp}}^{-1}(x)$ and denote $v_{i_k}^{opp}$ the $\mathring{c}_{\beta_{i_k}}^{opp}$ coordinate.

Note that $\mathring{c}_{\gamma}$ and $\mathring{c}_{\gamma}^{opp}$ are both $Ad_{L_{\Theta}^{\circ}}$-invariant, and we can verify that for any $g\in L_{\Theta}^{\circ}$, $x\in U_{\Theta}^{>0}$, $x'\in U_{\Theta}^{>0, opp}$, \refchange{130}{\changed{we have}} $v_{gxg^{-1}} = Ad_g (v_x)$ and $v_{gx'g^{-1}}^{opp} = Ad_g(v_{x'}^{opp})$.

\subsubsection{Tangent Cone} \label{sssec15}

Define the \emph{open tangent cone} of $U_{\Theta}^{>0}$ by $T\mathring{c} := \refchange{131}{\changed{\bigoplus}}_{\beta \in \Theta} \mathring{c}_{\beta}\subset \mathfrak{u}_{\Theta}$ and let $Tc$ be its closure in $\mathfrak{u}_{\Theta}$. Similarly define $T\mathring{c}^{opp},Tc^{opp}\subset \mathfrak{u}_{\Theta}^{opp}$. Although $U_{\Theta}^{>0}$ is not a smooth sub-manifold of $G$, we show $T\mathring{c}$ still serves as a linear approximation to $U_{\Theta}^{>0}$ near $id$.

 For any $\beta\in \Theta$, define $\pi_{\beta}:\mathring{c}_{\gamma}\to \mathring{c}_{\beta}\subset \mathfrak{u}_{\beta}$ by
    $$
    v  \refchange{133}{ \changed{\mapsto} } \sum_{k, \beta_{i_k} = \beta} v_{i_k}
    $$
\refchange{135}{\changed{where $v_{i_k}$ is regarded as an element of $\mathring{c}_{\beta}\subset \mathfrak{u}_{\beta}$, and $\sum$ denotes addition in $\mathfrak{u}_{\beta}$.}}

\refchange{132}{\changed{One can show that for any reduced expression $\gamma$ and any $\beta \in \Theta$, the root $\beta$ occurs in $\gamma$. Otherwise, if some $\beta \in \Theta$ does not occur in $\gamma$, then Theorem~\ref{thm9} implies that $U_{\Theta}^{>0}\,\mathfrak{p}_{\beta}^{opp} = \mathfrak{p}_{\beta}^{opp}$. However, $U_{\Theta}^{>0}\,\mathfrak{p}_{\beta}^{opp}$ must be an open set in \refchangenew{132}{\changednew{$\mathcal{F}_{\beta}^{opp}$}} (as $U_{\Theta}^{>0}\,\mathfrak{p}_{\Theta}^{opp}$ is an open set in $\mathcal{F}_{\Theta}^{opp}$}, see Guichard-Wienhard~\cite[Proposition~13.1]{GW2}), which is a contradiction. Hence $\pi_{\beta}$ is always a nonzero map.}}

Also define \refchange{134}{\changed{$\pi_T: \mathring{c}_{\gamma}\to T\mathring{c}$}} by
    $$
     v\mapsto \sum_{\beta\in\Theta} \pi_{\beta}(v)
    $$
\refchange{136}{\changed{where $\pi_{\beta}(v) \in \mathring{c}_{\beta}\subset \mathfrak{u}_{\beta}\subset \mathfrak{u}_{\Theta}$, and $\sum$ denotes addition in $\mathfrak{u}_{\Theta}$.}} Similarly define $\pi_{\beta}^{opp}: \mathring{c}_{\gamma}^{opp}\to \mathring{c}_{\beta}^{opp}\subset \mathfrak{u}_{\beta}^{opp}$ and $\pi_T^{opp}: \mathring{c}_{\gamma}^{opp}\to T\mathring{c}^{opp}$.

Recall that $\mathfrak{u}_{\Theta}$ is a graded nilpotent Lie algebra with the grade-one factor $\sum_{\beta\in \Theta} \mathfrak{u}_{\beta}$. Since $\exp:\mathfrak{u}_{\Theta}\to U_{\Theta}$ is a diffeomorphism, we can define its inverse $\log: U_{\Theta}\to \mathfrak{u}_{\Theta}$ globally. 

\refchange{137}{\changed{Here are some properties of the tangent cone.}}

\begin{theorem}{\refchange[\baselineskip]{137}{\changed{\cite[Proposition 5.1, \refchange{139}{Section~10.8}: Lemma 10.21,  Proposition 10.22, Theorem 10.24]{GW2}.}}}\label{thm10}
    
    Fix norms on $\mathfrak{u}_{\Theta}$ and $\mathring{c}_{\gamma}$, then
    \begin{enumerate}
        \item The diagonal action of $L_{\Theta}^{\circ}$ acts transitively and properly on $T\mathring{c}$, and the stabilizer of any point in $Tc - T\mathring{c}$ is non-compact.
        
        \item For each compact subset $S\subset G$, there is a constant $C$  depending only on $S$, such that for any $x\in S\cap U_{\Theta}^{>0} $ and any $\beta \in \Theta$, we have
        \begin{enumerate}
            \item \changed{$\|\tilde{\pi}_{\beta}(\log x) - \pi_\beta(v_x) \|\le C \|v_x\|\cdot \| \pi_{\beta}(v_x) \|$ where $\tilde{\pi}_{\beta}:\mathfrak{u}_{\Theta}\to \mathfrak{u}_{\{\beta\}}$ is the linear projection.}
            \item $\|\log x - \pi_T(v_x) \|\le C \| v_x \|^2$.
        \end{enumerate}

        \item If $\gamma: [0,1]\to U_{\Theta}^{\ge 0}\subset G$ is a $C^1$ map and $\gamma(0) = id$, then $\gamma'(0)\in Tc$.

        \item For any $X\in T\mathring{c}$, $\exp(X)\in U_{\Theta}^{>0}$.

    \end{enumerate}
    Similar results hold for $U_{\Theta}^{opp,>0}$, $T\mathring{c}^{opp}$ and $Tc^{opp}$.
\end{theorem}

{\color{blue}{Although only item~$(b)$ was stated in the reference for Theorem~\ref{thm10}~(2), the degree argument in the reference indeed implies the more refined item~$(a)$. We explain it briefly here:

For any $x\in U_{\Theta}^{>0}$ and any $\beta\in \Theta$, the Baker–Campbell–Hausdorff formula shows that
\begin{align*}
\log x &= \sum_{1\le k\le l} (v_x)_{i_k}+\sum \text{iterated Lie brackets of } (v_{x})_{i_k} \\
&= \sum_{\substack{1\le k\le l \\ \beta_{i_k} = \beta}} (v_x)_{i_k}
   + \sum_{\substack{1\le k\le l \\ \beta_{i_k} \not = \beta}} (v_x)_{i_k} \\
&\quad + \sum \text{all the iterated Lie brackets of } (v_{x})_{i_k} \text{ involving some $(v_{x})_{i_k}, \beta_{i_k} = \beta$} \\
&\quad + \sum \text{all the iterated Lie brackets of } (v_{x})_{i_k} \text{ involving none of $(v_{x})_{i_k}, \beta_{i_k} = \beta$}.
\end{align*}

From Theorem~\ref{thm6}, $\mathfrak{u}_{\Theta}$ is a graded Lie algebra whose grade-one components are $\mathfrak{u}_{\beta},\beta\in \Theta$. We can check that
\[
\pi_{\beta}(v_x) = \sum_{\substack{1\le k\le l \\ \beta_{i_k} = \beta}} (v_x)_{i_k}
\]
and
\[
\tilde{\pi}_{\beta}(\log x) - \pi_\beta(v_x) = \sum \text{all the iterated Lie brackets of } (v_{x})_{i_k} \text{ involving some $(v_{x})_{i_k}, \beta_{i_k} = \beta$}.
\]

Let $C(x) = \max_{1\le k \le l} \|v_{i_k}\|$. Note that the Lie bracket is bilinear, which means there exists a constant $C'$ such that $\|[X,Y]\| \le C' \|X\| \|Y\|$ for any $X,Y\in \mathfrak{u}_{\Theta}$. Also, note that since $\mathfrak{u}_{\Theta}$ is nilpotent, terms involving sufficiently many Lie brackets vanish. Then, as long as $x$ varies in a compact subset $S$, all the components of $v_x$ are uniformly bounded; and also note that all the Lie brackets involve some $(v_{x})_{i_k}, \beta_{i_k} = \beta$, so there exists a constant $C_0$ (depending on $S$) such that
\[
\| \tilde{\pi}_{\beta}(\log x) - \pi_\beta(v_x)\|\le C_0 C(x)(\sum_{1\le k\le l,\beta_{i_k} =\beta}\left\| v_{i_k}\right\|).
\]
Note that there exists a constant $C_1$ such that $C(x) \le C_1 \| v_x\|$ always holds. And since in the above summation, $v_{i_k}\in \mathring{c}_{\beta}$ for all relevant $k$, Lemma~\ref{lem15} shows that there exists $C_2$ such that
\[
\sum_{1\le k\le l,\beta_{i_k} =\beta}\left\| v_{i_k}\right\|\le C_2 \left\| \sum_{1\le k\le l,\beta_{i_k} = \beta} v_{i_k}\right\| = C_2 \| \pi_{\beta}(v_x)\|.
\]
We have thus arrived at the desired inequality.
}}

\vspace{3mm}

\changednew{We continue to view $\mathfrak{u}_{\Theta}$ as a graded Lie algebra, as in Theorem~\ref{thm6}, and prove the following lemma:}

\begin{lemma}[Formal Group Law]\label{lem8} 
If $x,y \in U_{\Theta}^{>0}$, then we have $\pi_{T}(v_{xy}) = \pi_{T}(v_x) + \pi_{T}(v_y)$. 

Moreover, for all $\beta \in \Theta$, we have
\[
\pi_{\beta}(v_{xy}) = \pi_{\beta}(v_x) + \pi_{\beta}(v_y).
\]

Analogous statements hold for $\pi_T^{opp}$ and $\pi_{\beta}^{opp}$.
\end{lemma}
\begin{proof}
Still using the Baker–Campbell–Hausdorff formula, we can see that for any $X,Y\in \mathfrak{u}_{\Theta}$, $\refchange{140}{\changed{\log (\exp X \exp Y)}} = (X+Y) +e(X,Y)$, where $e(X,Y)$ consists of higher order Lie brackets. So if $X_1,Y_1$ are the grade one factors of $X,Y$ in $\mathfrak{u}_{\Theta}$, then $X_1+Y_1$ is the grade one factor of $\refchange{140}{\changed{\log (\exp X \exp Y)}} $ in $\mathfrak{u}_{\Theta}$ (which is denoted by $(\log \exp X \exp Y)_1$). 

\refchange[\baselineskip]{141}{\changed{Applying $\big(\log (\exp X \exp Y)\big)_1 = X_1+Y_1$ iteratively, we see that for any $x \in U_{\Theta}^{>0}$, $\pi_T(v_x)$ is the grade-one component of $\log x$, and we have proved that $\pi_{T}(v_{xy}) = \pi_{T}(v_x) + \pi_{T}(v_y)$.}}

\refchangenew{142}{\changednew{The ``moreover'' statement follows from the observations that $\pi_T(v_x) \in T\mathring{c} := \bigoplus_{\beta \in \Theta} \mathring{c}_{\beta}$ and that $\pi_{\beta}(v_x)$ is the $\mathring{c}_{\beta}$-component of $\pi_{T}(v_x)$.}}

\end{proof}

\subsubsection{Positive maps} \label{sssec16}
The classification of $\Theta$-positive structures shows that $\Theta\subset \Delta$ is symmetric under any possible involutions of the Dynkin diagram. As a result, if $\dot{\omega}_0\in Inn(\mathfrak{g})$ is the lifting of the longest element $\omega_0$ in the Weyl group $W$, then $\mathfrak{p}_{\Theta}$ is conjugated to $\mathfrak{p}_{\Theta}^{opp}$ under $\dot{\omega}_0$. So we can identify $\mathcal{F}_{\Theta}$ with $\mathcal{F}_{\Theta}^{opp}$.

\refchange{143}{\changed{
The original definition of positive tuples in Guichard–Wienhard~\cite[Definition~13.23]{GW2} is stated using the language of diamonds. We adopt an equivalent definition based on their lemma.}}

\changed{
\begin{definition}{\refchange{143}{\cite[Lemma 13.25]{GW2}}}\label{def9}
 \refchange{144, 149}{An $n$-tuple} $(f_1,f_2,...,f_n)$ in $\mathcal{F}_{\Theta}$ is \emph{positive} if there exists $g\in Aut_1(\mathfrak{g}), u_1,u_2,...u_{n-2}\in U_{\Theta}^{>0}$, such that: 
$$
(f_1,...,f_n) = g (\mathfrak{p}_{\Theta},u_1u_2\cdots u_{n-2}\mathfrak{p}_{\Theta}^{opp}\refchange{145}{,...,}u_1u_2\mathfrak{p}_{\Theta}^{opp}, u_1\mathfrak{p}_{\Theta}^{opp}, \mathfrak{p}_{\Theta}^{opp}).
$$
\end{definition}
}

\begin{remark} \label{rem12}
\refchange{146}{\changed{When $G = \mathsf{PSL}(d,\mathbb{R})$ and $\Theta = \Delta$, the semigroup $U_{\Delta}^{>0}$ coincides with the set of totally positive upper triangular matrices, and there has been a substantial body of work on this topic since the early 20th century.}}

    More specially, when $G = \mathsf{PSL}(2,\mathbb{R})$, $\mathcal{F}_{\Delta}$ is $S^1$, and \refchange{144}{\changed{an $n-$tuple}} in $\mathcal{F}_{\Delta}$ is positive if and only if it is cyclically ordered along $S^1$. As a result, any pairwise distinct triple in $S^1$ is positive.
\end{remark}

Guichard-Wienhard~\cite{GW2} proved some useful properties about positive tuples, we list some of them here for future use.

\begin{theorem}{\refchange{147}{\changed{\cite[Proposition 13.1, Proposition 13.15, Proposition 13.19, Lemma 13.24]{GW2}.}}}\label{thm11}
    Suppose $G$ has $\Theta$-positive structure, then

    \begin{enumerate}
        \item Positive $n-$tuples in $\mathcal{F}_{\Theta}$ are invariant under action of $Aut_1(\mathfrak{g})$.
        \item Positive $n-$tuples are invariant under the action of $D_{2n}$, i.e., invariant under cyclic permutations and under the involution $i\leftrightarrow n-i+1$. As a result, positive triples in $\mathcal{F}_{\Theta}$ are invariant under permutations of $S_3$, and positive quadruples in $\mathcal{F}_{\Theta}$ are invariant under permutations of \refchange{148}{\changed{$D_4=\langle(1,2,3,4),(1,4)(2,3)\rangle\subset S_4$.}}
        \item The group $G$ (so is $Aut_1(\mathfrak{g})$) acts properly on the space of positive triples in $\mathcal{F}_{\Theta}$. 
        \item $(f_1,f_2,f_3)$ is positive in $\mathcal{F}_{\Theta}$ iff it is in the $Aut_1(\mathfrak{g})$ orbit of $(\mathfrak{p}_{\Theta},U_{\Theta}^{>0} \mathfrak{p}_{\Theta}^{opp},\mathfrak{p}_{\Theta}^{opp})$, iff it is in the $Aut_1(\mathfrak{g})$ orbit of $(\mathfrak{p}_{\Theta},(U_{\Theta}^{>0})^{-1} \mathfrak{p}_{\Theta}^{opp},\mathfrak{p}_{\Theta}^{opp})$. 
        \item $(f_1,f_2,f_3,f_4)$ is positive in $\mathcal{F}_{\Theta}$ iff there exists $g\in Aut_1(\mathfrak{g}), u,v\in U_{\Theta}^{>0}$, such that $(f_1,f_2,f_3, f_4) = g(\mathfrak{p}_{\Theta},u\mathfrak{p}_{\Theta}^{opp},\mathfrak{p}_{\Theta}^{opp},v^{-1}\mathfrak{p}_{\Theta}^{opp})$, iff there exists $g\in Aut_1(\mathfrak{g}), u,v\in U_{\Theta}^{opp,>0}$, such that \refchange{149}{\changed{$(f_1,f_2,f_3, f_4) = g(\mathfrak{p}_{\Theta},u\mathfrak{p}_{\Theta},uv\mathfrak{p}_{\Theta},\mathfrak{p}_{\Theta}^{opp})$.}}
        \item The set of positive triples in $\mathcal{F}_{\Theta}$ is a union of connected components in the space of pairwise transverse triples, and thus open in $(\mathcal{F}_{\Theta})^3$.
        \item Let $\mathcal{TF}\subset \mathcal{F}_{\Theta}$ be the subset of flags simultaneously transverse to $\mathfrak{p}_{\Theta}$ and $\mathfrak{p}_{\Theta}^{opp}$. The map $U_{\Theta}^{>0}\to \mathcal{F}_{\Theta}$, \refchange{150}{\changed{$u\mapsto u\mathfrak{p}_{\Theta}^{opp}$}} is a diffeomorphism onto a connected component of $\mathcal{TF}$. Similarly $U_{\Theta}^{opp,>0}\to \mathcal{F}_{\Theta}$, \refchange[\baselineskip]{150}{\changed{$u\mapsto u\mathfrak{p}_{\Theta}$}} is also a diffeomorphism onto a connected component of $\mathcal{TF}$ and $U_{\Theta}^{>0}\mathfrak{p}_{\Theta}^{opp} = U_{\Theta}^{opp,>0}\mathfrak{p}_{\Theta}$.
    \end{enumerate}
\end{theorem}

\refchange{147}{\changed{More precisely, item~(1) follows directly from our Definition~\ref{def9}. Item~(2) is precisely \cite[Lemma~13.24, Proposition~13.15~(3) and Proposition~13.19~(8)]{GW2}. Item~(3) is precisely \cite[Proposition~13.15~(7)]{GW2}. Item~(4) is precisely \cite[Proposition~13.15~(4),(5)]{GW2}. Item~(5) is precisely \cite[Proposition~13.19~(2),(5)]{GW2}. Item~(6) is precisely \cite[Proposition~13.15~(6)]{GW2}. Item~(7) is precisely \cite[Proposition~13.1]{GW2}.
}}

\vspace{3mm}
Now let $\Lambda\subset S^1$ be a subset of $S^1$, we say a map from $\Lambda$ to $\mathcal{F}_{\Theta}$ is \emph{positive}, if it sends every positive $n-$tuple in $\Lambda$ to a positive $n-$tuple in $\mathcal{F}_{\Theta}$.  If $|\Lambda|\ge 4$, then a map from $\Lambda$ to $\mathcal{F}_{\Theta}$ is positive if and only if it sends positive quadruples to positive quadruples (see \cite[Lemma 13.30]{GW2}).

We provide a convergence lemma on parameterizing positive maps:

\begin{lemma}\label{posconvergence}
\refchange{151}{\changed{Let $\Lambda\subset S^1$ be a closed subset. Let $a_n,b_n\in \Lambda$ be converging sequences such that $a_n\to a,b_n\to b, a\not = b$. Pick a closed arc in $S^1$ joining $a$ and $b$ and denote it by $[a,b]_{arc}$ (with interior $(a,b)_{arc}$). \refchange{152}{\changed{Let $E\subset (a,b)_{arc}$ be a closed subset and let $\xi:\Lambda\to \mathcal{F}_{\Theta}$ be a continuous positive map.}}}}

There exists \refchange{153}{\changed{a sequence $g_n\in Aut_1(\mathfrak{g})$ which converges to $g\in Aut_1(\mathfrak{g})$}} and a sequence of continuous maps $x_n:E\to U_{\Theta}^{>0}$ such that for $n\gg1$ and any $p\in E$
$$
\xi(a_n,p,b_n) = g_n(\mathfrak{p}_{\Theta}^{opp}, x_n(p) \mathfrak{p}_{\Theta}^{opp},\mathfrak{p}_{\Theta}).
$$
Moreover, $x_n$ has the following properties: There exists a continuous map $x_{\infty}:E\to U_{\Theta}^{>0}$ such that $x_n\to x_{\infty}$ uniformly on $E$, and for any $p,q\in E$ such that $(a,p,q,b)$ is positive and any $k\in \refchange{154}{\changed{\mathbb{Z}^{>0}}}\cup \{\infty\}$, we have $x_k(p)^{-1}x_k(q)\in U_{\Theta}^{>0}$.
\end{lemma}

\begin{proof}
  Pick $p_0\in E$. For $n\refchange{155}{\changed{\gg}}1$, $(a_n,p_0,b_n)$ is a distinct triple. Note that $\xi(a_n),\xi(b_n)$ converges to the transverse pair $\xi(a),\xi(b)\in (\mathcal{F}_{\Theta})^2$, so there exists a a sequence $g_n\in Aut_1(\mathfrak{g})$ which converges to $g\in Aut_1(\mathfrak{g})$ such that $\xi(a_n,b_n) =g_n(\mathfrak{p}_{\Theta}^{opp},\mathfrak{p}_{\Theta})$ and $\xi(p_0)\in g_n U_{\Theta}^{>0}\mathfrak{p}_{\Theta}^{opp}$. \refchange{156}{\changed{It should be noted that the sequence $g_n \in Aut_1(\mathfrak{g})$ and the limit $g \in Aut_1(\mathfrak{g})$ may not be unique, and we are making a choice.}}
\changed{As $\xi(p_0)$ lies in the connected component $g_n U_{\Theta}^{>0}\mathfrak{p}_{\Theta}^{opp}$ of $g_n \mathcal{TF}$, as mentioned in Theorem~\ref{thm11}\,(7), it follows from \cite[Definition~13.2\,(4) and Definition~13.23]{GW2} that $\xi^{\alpha}(E)$ also lies in the same connected component}
, so for each $n$, there exists a unique continuous $x_n: E\to U_{\Theta}^{>0}$ such that for any $p\in E$,
    $$
    \xi(a_n,p,b_n) = g_n(\mathfrak{p}_{\Theta}^{opp},x_n(p)\mathfrak{p}_{\Theta}^{opp}, \mathfrak{p}_{\Theta}).
    $$
    Namely for any $p\in E$,
    $$
   g_n^{-1} \xi(p) =x_n(p)\mathfrak{p}_{\Theta}^{opp}.
    $$
As $a_n \to a$, $b_n \to b$, $g_n \to g$, and $E \subset (a,b)_{\mathrm{arc}}$ is closed, we have that for $n \gg 1$, the set $g_n^{-1}\xi(E)$ is contained in a compact subset of $\mathcal{TF}$. Hence $x_n \to x_{\infty}$ uniformly on $E$, where $x_{\infty} : E \to U_{\Theta}^{>0}$ is defined by the property that for any $p \in E$,
\[
g^{-1}\xi(p) = x_{\infty}(p)\,\mathfrak{p}_{\Theta}^{opp}.
\]
Moreover, since $\xi$ is positive, for any $p,q \in E$ such that $(a,p,q,b)$ is positive and any $k \in \mathbb{Z}^{>0} \cup \{\infty\}$, we have
\[
x_k(p)^{-1} x_k(q) \in U_{\Theta}^{>0}.
\]

\end{proof}

By Theorem~\ref{thm11}~(2) \refchange{157}{\changed{and~(7), we obtain the following equivalent statements of Lemma~\ref{posconvergence}, whose proofs are obtained by directly repeating the proof of Lemma~\ref{posconvergence}:}}

\begin{corollary}\label{posconv2}
Under the same assumption of Lemma \ref{posconvergence}:
\begin{enumerate}
    \item There exists a sequence $g_n\in Aut_1(\mathfrak{g})$ which converges to $g\in Aut_1(\mathfrak{g})$ and a sequence of continuous maps $x_n:E\to U_{\Theta}^{>0}$ such that for any $p\in E$
$$
\xi(a_n,p,b_n) = g_n(\mathfrak{p}_{\Theta}, x_n(p) \mathfrak{p}_{\Theta}^{opp},\mathfrak{p}_{\Theta}^{opp}).
$$
Moreover, $x_n$ has the following properties: There exists a continuous map $x_{\infty}:E\to U_{\Theta}^{>0}$ such that $x_n\to x_{\infty}$ uniformly on $E$, and for any $p,q\in E$ such that $(a,p,q,b)$ is positive and any $k\in \mathbb{Z}^{>0}\cup \{\infty\}$, we have $x_k(q)^{-1}x_k(p)\in U_{\Theta}^{>0}$. 
 
    \item There exists a sequence $g_n\in Aut_1(\mathfrak{g})$ which converges to $g\in Aut_1(\mathfrak{g})$ and a sequence of continuous maps $x_n:E\to U_{\Theta}^{opp,>0}$ such that for any $p\in E$
$$
\xi(a_n,p,b_n) = g_n(\mathfrak{p}_{\Theta}^{opp}, x_n(p) \mathfrak{p}_{\Theta},\mathfrak{p}_{\Theta}).
$$
Moreover, $x_n$ has the following properties: There exists a continuous map $x_{\infty}:E\to U_{\Theta}^{opp,>0}$ such that $x_n\to x_{\infty}$ uniformly on $E$, and for any $p,q\in E$ such that $(a,p,q,b)$ is positive and any $k\in \mathbb{Z}^{>0}\cup \{\infty\}$, we have $x_k(q)^{-1}x_k(p)\in U_{\Theta}^{opp, >0}$.

    \item There exists a sequence $g_n\in Aut_1(\mathfrak{g})$ which converges to $g\in Aut_1(\mathfrak{g})$ and a sequence of continuous maps $x_n:E\to U_{\Theta}^{opp,>0}$ such that for any $p\in E$
$$
\xi(a_n,p,b_n) = g_n(\mathfrak{p}_{\Theta}, x_n(p) \mathfrak{p}_{\Theta},\mathfrak{p}_{\Theta}^{opp}).
$$
Moreover, $x_n$ has the following properties: There exists a continuous map $x_{\infty}:E\to U_{\Theta}^{opp,>0}$ such that $x_n\to x_{\infty}$ uniformly on $E$, and for any $p,q\in E$ such that $(a,p,q,b)$ is positive and any $k\in \mathbb{Z}^{>0}\cup \{\infty\}$, we have $x_k(p)^{-1}x_k(q)\in U_{\Theta}^{opp, >0}$.

\end{enumerate}

\end{corollary}

Next we prove the rectifiability of positive maps:

Suppose $\Lambda\subset S^1$ is a closed subset and $\xi:\Lambda\to (X,d)$ is a map from $\Lambda$ to a metric space $(X,d)$. We say $\xi$ is \emph{rectifiable} if there exists a constant $C$ such that for any $n\in \mathbb{Z}^{>0}$ and any positive tuple \refchange{158}{\changed{$(x_1,...,x_n)$}} in $\Lambda$,
$$
\sum_{i = 1}^{n-1}d(\xi(x_i),\xi(x_{i+1}))<C.
$$

\changed{When $\Lambda = S^1$ and $\xi$ is rectifiable, $\xi(S^1)$ is called a \emph{rectifiable curve} (see Stein-Shakarchi~\cite[Chapter 3: Section 3.1]{steinreal}), and it has non-zero and finite $1$-dimensional Hausdorff measure with respect to $d$ (see \cite[Chapter 7: Theorem 2.4]{steinreal}).}

\begin{proposition}\label{prop3}
Suppose $\Lambda\subset S^1$ is a closed subset and $\xi^{\Theta}:\Lambda \to \mathcal{F}_{\Theta}$ is a continuous positive map. Then $\xi^{\Theta}$ is rectifiable with respect to any Riemannian metric $d_{\Theta}$ on $\mathcal{F}_{\Theta}$.
\end{proposition}
\begin{proof}
\refchange{159}{\changed{The proof is a generalization of Burger-Iozzi-Labourie-Wienhard~\cite[Lemma~8.10]{maximal}, which established the rectifiability of positive maps for maximal representations. Here we extend the argument to the general $\Theta$-positive setting.}} One can assume $\Lambda$ to be infinite. 

\refchange{160, 161}{\changed{Since $\Lambda$ is closed, it is compact. Therefore, it suffices to prove rectifiability on small subintervals. In other words, it is enough to show:}}

Let $I\subset S^1$ be a \refchange{163}{\changed{closed sub-interval}} such that \refchange[2\baselineskip]{161}{\changed{$\Lambda\cap I\not = \emptyset$}} and $\Lambda - I$ has at least three points. Pick three distinct points $a,b,c\in \Lambda-I$ such that $b$ and $I$ do not lie in the same component of $S^1-\{a,c\}$. Let \refchange[\baselineskip]{162}{\changed{$\{x_1,x_2,...,x_n\}\subset \Lambda \cap I$ such that $(a,x_1,x_2,...,x_n,c,b)$ is positive}}. For fixed $I$, there exists a constant $C$ independent of $\{x_i\}$, such that
$$
\sum_{i = 1}^{n-1}d_{\Theta}(\xi^{\Theta}(x_i),\xi^{\Theta}(x_{i+1}))<C.
$$

According to Lemma \ref{posconvergence}, there exists $g\in Aut_1(\mathfrak{g})$ and a map $u: (I\cap \Lambda)\cup \{c\}\to  U_{\Theta}^{>0}$ such that for any $x\in (I\cap \Lambda)\cup\{c\}$
$$
\xi^{\Theta}(a,x,b) = g(\mathfrak{p}_{\Theta}^{opp}, u(x) \mathfrak{p}_{\Theta}^{opp}, \mathfrak{p}_{\Theta}).
$$
And for any $x,y\in (I\cap \Lambda)\cup\{c\}$ such that $(a,x,y,b)$ is positive, we have $u(x)^{-1}u(y)\in U_{\Theta}^{>0}$. \refchange{164}{\changed{Recall that $v_{u(x)}$ denotes the cone coordinate (with respect to an implicitly chosen reduced expression) defined after Theorem~\ref{thm9}, and that $\pi_{T}$ is the projection to $T\mathring{c}$ defined before Theorem~\ref{thm10}.}} So for any $x\in \Lambda\cap I$, we have $\pi_T(v_{u(c)}) = \pi_T(v_{u(x)})+\pi_T(v_{u^{-1}(x)u(c)})$. As \refchange{165}{\changed{$T\mathring{c}$}} is acute, for fixed $I,a,b,c$, $\pi_T(v_{u(x)})$ is uniformly bounded independently of $x$.

\refchange{167,168}{\changed{As $(I \cap \Lambda)\cup \{c\}$ is compact, its image $u((I \cap \Lambda)\cup \{c\}) \subset U_{\Theta}^{>0}$ is also compact. Hence, for any $x \in (I \cap \Lambda)\cup \{c\}$, the element $u(x)$ is uniformly bounded in the Lie group $G$.}}

Denote $x_{n+1} = c$.  For \refchange{166}{\changed{$i = 1,...,n$}}, let $d_i = u(x_{i})^{-1}u(x_{i+1})\in U_{\Theta}^{>0}$ which are also uniformly bounded elements in $G$. As the metric $d_{\Theta}$ on $\mathcal{F}_{\Theta}$ is Riemannian, for any norm on $\mathfrak{u}_{\Theta}$, there exists constants $C_1,C_2$, such that
$$
\sum_{i = 1}^{n-1}d_{\Theta}(\xi^{\Theta}(x_i),\xi^{\Theta}(x_{i+1}))<\sum_{i = 1}^{n}d_{\Theta}(\xi^{\Theta}(x_i),\xi^{\Theta}(x_{i+1}))
$$
$$
\le C_1 \sum_{i = 1}^{n}d_{\Theta}(\mathfrak{p}_{\Theta}^{opp},d_i\mathfrak{p}_{\Theta}^{opp}) \le C_2 \sum_{i = 1}^{n} \|\log(d_i) \|.
$$
(The $G$ action on $\mathcal{F}_{\Theta}$ is smooth thus bi-Lipsthitz. As $u(x_i)$ are uniformly bounded elements in $G$ independent of $x_i$, their actions are uniformly bi-Lipschitz so we can get the inequality of $C_1$. The inequality of $C_2$ is from the definition of Riemannian metric.)

From Theorem \ref{thm10} (2), there exists a constant $C_3$

$$
\|\log(d_i)-\pi_T(v_{d_i}) \|\le  C_3\| v_{d_i}\|^2.
$$

\changednew{Lemma~\ref{lem15}~$(1)$ implies that there exists a constant $C'$ such that for any $v\in \mathring{c}_{\gamma}$, we have $\|v\|\le C'\|\pi_T(v)\|$.} Since $d_i$ are all bounded, $v_{d_i}$ are bounded too, so there exists a constant $C_4$ such that $\|v_{d_i}\|^2\le  C_4\|\pi_T(v_{d_i})\|$. Thus there exists a constant $C_5$ such that
$$
\|\log(d_i)\|\le  C_5\|\pi_T( v_{d_i})\|.
$$

By the formal group law (Lemma~\ref{lem8}), \refchange{169}{\changed{$\sum_{i = 1}^{n} \pi_T(v_{d_i}) = \pi_T(v_{u(x_1)^{-1}u(c)})$. Since $\Lambda \cap I$ is closed, let $\tilde{x}$ be the endpoint of $\Lambda \cap I$ such that for any $\hat{x}\ne \tilde{x} \in \Lambda \cap I$, $(a,\hat{x},\tilde{x},c)$ is positive. Again, by the formal group law (Lemma~\ref{lem8}), $\pi_T(v_{u(x_1)^{-1}u(c)}) = \pi_{T}(v_{u(x_1)^{-1}u(\tilde{x})})+\pi_{T}(v_{u(\tilde{x})^{-1}u(c)})$, which is uniformly properly bounded (properly bounded was defined in the discussion preceding Lemma~\ref{lem15}).}}
From Lemma \ref{lem15} we see that \refchange{170}{\changed{$\sum_{i = 1}^{n} \|\pi_T(v_{d_i})\|$}} is uniformly bounded, so the map is rectifiable.
\end{proof}

\begin{remark}\label{rem16}
Note that there is a smooth projection from $\mathcal{F}_{\Theta}$ to $\mathcal{F}_{\alpha}$ for any $\alpha\in \Theta$. Composing the positive map $\xi^{\Theta}$ by the projection we get $\xi^{\alpha}$, which is still rectifiable.
\end{remark}

\subsection{$\Theta$-positive Representations}

Now we define $\Theta$-positive representations. Let $\Gamma\subset \mathsf{PSL}(2,\mathbb{R})$ be a non-elementary discrete subgroup and $G$ be a real simple Lie group with $\Theta$-positive structure.

\begin{definition}\label{def11}
    A representation $\rho: \Gamma\to G$ is called \emph{$\Theta$-positive} if there is a continuous, positive and $\rho-$equivariant map $\xi^{\Theta}:\Lambda(\Gamma)\to \mathcal{F}_{\Theta}$.
\end{definition}

\begin{proposition}\label{thm17}
     $\Theta$-positive representations are $\Theta$-transverse. As a result, when $\Gamma$ is geometrically finite, $\Theta$-positive representations are relatively $\Theta$-Anosov. 
     
\refchange{171}{\changed{Furthermore, the positive map coincides with the limit map for $\Theta$-transverse representations, and is therefore unique (see the paragraph before Proposition~\ref{cartanproperty}).}}

\end{proposition}

\begin{proof}
Recall that we fixed a Riemannian metric $d_{\Theta}$ on $\mathcal{F}_{\Theta}$. And recall that the classification of $\Theta$-positive structures shows that $\Theta\subset \Delta$ is symmetric under any possible involutions of the Dynkin diagram. If $\dot{\omega}_0\in Inn(\mathfrak{g})$ is the lifting of the longest element $\omega_0$ in the Weyl group $W$, then $\dot{\omega}_0 \mathfrak{p}_{\Theta} = \mathfrak{p}_{\Theta}^{opp}$, so $\mathcal{F}_{\Theta}$ can be identified with $\mathcal{F}_{\Theta}^{opp}$, and we can tautologically define $\xi^{\Theta,opp}:\Lambda(\Gamma)\to \mathcal{F}_{\Theta}^{opp}$.

Positivity implies that the maps $\xi^{\Theta}$ and $\xi_{\Theta}^{opp}$ are transverse. From Definition \ref{Anosov}, we only need to prove $\xi^{\Theta}$ is strongly dynamics preserving. \refchange{172}{\changed{From Proposition \ref{cartanproperty}~(3) (see also Canary-Zhang-Zimmer \cite[Proposition 2.6]{CaZhZi3})}} it suffices to verify the strongly dynamics preserving property for some open subset, that is we only need to prove:

\emph{
\refchange{173}{\changed{Let $b_0\in \mathbb{D}$ be a fixed base point and $\gamma_n\in \Gamma$ be a sequence.}} If $\gamma_n(b_0)\to  x,\gamma_n^{-1}(b_0)\to y$ for $x,y\in\Lambda(\Gamma)$, then there exists a \refchange{174}{\changed{non-empty}} open subset $\mathcal{O}\subset \mathcal{F}_{\Theta}$ whose elements are all transverse to $\xi^{\Theta,opp}(y)$, and for any element $F\in \mathcal{O}$, we have $\rho(\gamma_n) F\to \xi^{\Theta}(x)$.
}

First suppose $x\not =y$, \refchange{175}{\changed{then from Lemma \ref{lem3}}} $\gamma_n$ are eventually hyperbolic and $\gamma_n^{+}\to x,\gamma_n^{-}\to y$. As $\Lambda(\Gamma)$ is a perfect set, we can pick $a,\refchangenew{177}{\changednew{z}},b\in \Lambda(\Gamma)$, such that $(y,b,\changednew{z},a,x)$ is positive. Thus, for sufficiently large $n$, the tuples $(\gamma_n^{-}, b,\changednew{z},a,\gamma_n^{+})$ and $(\gamma_n^{-},\gamma_n(b),\changednew{\gamma_n(z)},\gamma_n(a),\gamma_n^{+})$ are positive, and we have $\gamma_n(a),\gamma_n(b), \changednew{\gamma_n(z)}\to x$.

Let 
$$
\mathcal{O}_n := \{F\in \mathcal{F}_{\Theta}|(\xi^{\Theta}(\gamma_n^{-}),\xi^{\Theta}(b),F,\xi^{\Theta}(a),\xi^{\Theta}(\gamma_n^{+})) \text{ is positive}\}.
$$ 
If $n\gg 1$, then $\mathcal{O}_n$ is an open subset containing $\xi^{\Theta}(z)$.

\refchangenew{178}{\changednew{
For any transverse flags $F_1,F_2\in \mathcal{F}_{\Theta}$, let $\mathcal{TF}_{F_1,F_2}\subset \mathcal{F}_{\Theta}$ denote the space of flags simultaneously transverse to $F_1$ and $F_2$. By Theorem~\ref{thm11}~(7), we can verify that $\mathcal{TF}_{F_1,F_2}$ is an open subset with finitely many connected components, and each connected component can be written as $gU_{\Theta}^{>0} \mathfrak{p}_{\Theta}^{opp}$ for some $g\in Aut_1(\mathfrak{g})$ (in Guichard-Wienhard~\cite[Definition~13.2 and Corollary~13.6]{GW2}, each connected component is called a \emph{diamond} extremities at $F_1$ and $F_2$). Furthermore, if $(F_1, f_1,\dots,f_n, F_2)$ is a positive $(n+2)$-tuple, then $f_1,\dots,f_n$ are contained in the same connected component of $\mathcal{TF}_{F_1,F_2}$.

By Guichard-Wienhard~\cite[Proposition~13.26]{GW2},
\[
(\xi^{\Theta}(\gamma_n^{-}),\xi^{\Theta}(b),F,\xi^{\Theta}(a),\xi^{\Theta}(\gamma_n^{+}))
\]
is positive if and only if both of
\[
(\xi^{\Theta}(\gamma_n^{-}),\xi^{\Theta}(b),F,\xi^{\Theta}(\gamma_n^{+})) \quad \text{and} \quad (\xi^{\Theta}(\gamma_n^{-}),F,\xi^{\Theta}(a),\xi^{\Theta}(\gamma_n^{+}))
\]
are positive. Using \cite[Definition~13.2 and Corollary~13.6]{GW2}, for $n\gg1$, let $\mathcal{A}_n$, $\mathcal{B}_{n}$, and $\mathcal{C}_n$ denote the unique connected components of $\mathcal{TF}_{\xi^{\Theta}(\gamma_n^-),\xi^{\Theta}(a)}$, $\mathcal{TF}_{\xi^{\Theta}(b), \xi^{\Theta}(\gamma_n^+)}$, and $\mathcal{TF}_{\xi^{\Theta}(\gamma_n^-), \xi^{\Theta}(\gamma_n^+)}$ respectively, determined by:
\begin{enumerate}
    \item $\xi^{\Theta}(a),\xi^{\Theta}(b)\in \mathcal{C}_n$.
    \item $\mathcal{A}_n\subset \mathcal{C}_n$ and $\mathcal{B}_n\subset \mathcal{C}_n$.
\end{enumerate}
It can be further verified that $\xi^{\Theta}(a)\in \mathcal{B}_n$ and $\xi^{\Theta}(b)\in \mathcal{A}_n$.

By \cite[Proposition~13.19~(4)]{GW2}, we see that
\[
F\in \mathcal{A}_n \Leftrightarrow (\xi^{\Theta}(\gamma_n^{-}),F,\xi^{\Theta}(a),\xi^{\Theta}(\gamma_n^{+})) \text{ is positive}
\]
and similarly
\[
F\in \mathcal{B}_n \Leftrightarrow (\xi^{\Theta}(\gamma_n^{-}),\xi^{\Theta}(b),F,\xi^{\Theta}(\gamma_n^{+})) \text{ is positive}.
\]
Consequently, $\mathcal{O}_n = \mathcal{A}_n\cap \mathcal{B}_n$, which is non-empty for $n\gg 1$. On the other hand, applying the ``only if'' direction of \cite[Proposition~13.19~(4)]{GW2} to the quadruples
\[
(\xi^{\Theta}(b),F,\xi^{\Theta}(a),\xi^{\Theta}(\gamma_n^{+})) \quad \text{and} \quad (\xi^{\Theta}(\gamma_n^{-}),\xi^{\Theta}(b),F,\xi^{\Theta}(a)),
\]
we see that they are both positive if and only if $\mathcal{C}^+_n = \mathcal{C}^-_n$ and $F$ lies in this set, where $\mathcal{C}^+_n$ and $\mathcal{C}^-_n$ denote the unique connected components of $\mathcal{TF}_{\xi^{\Theta}(a),\xi^{\Theta}(b)}$ contained in $\mathcal{B}_n$ and $\mathcal{A}_n$, respectively. Note that for $n\gg 1$, $\xi^{\Theta}(z)$ lies in both $\mathcal{C}^+_n$ and $\mathcal{C}^-_n$. Since the connected components $\mathcal{C}^+_n$ and $\mathcal{C}^-_n$ intersect, they must coincide; specifically, they are equal to the connected component of $\mathcal{TF}_{\xi^{\Theta}(a),\xi^{\Theta}(b)}$ containing $\xi^{\Theta}(z)$. Denoting this component by $\mathcal{O}$, we see that $\mathcal{O}$ is contained in both $\mathcal{A}_n$ and $\mathcal{B}_n$; therefore, we conclude that for $n\gg 1$, $\mathcal{O}\subset \mathcal{A}_n\cap \mathcal{B}_n = \mathcal{O}_n$.
}}

Let 
$$
\mathcal{O}'_n :=\rho(\gamma_n)\mathcal{O}_n = \{F'\in \mathcal{F}_{\Theta}|(\xi^{\Theta}(\gamma_n^{-}),\xi^{\Theta}(\gamma_{n}b),F',\xi^{\Theta}(\gamma_na),\xi^{\Theta}(\gamma_n^{+})) \text{ is positive}\}
$$ 
and we can see that
$$
\changednew{\rho(\gamma_n)\mathcal{O} \subset  \mathcal{O}'_n.}
$$

It suffices to show the Hausdorff limit of $\mathcal{O}'_n$ \refchange{188}{\changed{is the single point $\xi^{\Theta}(x)$.}} If not, then there exists a sequence $F_n\in \mathcal{O}'_n$ such that
$$
(\xi^{\Theta}(\gamma_n^{-}),\xi^{\Theta}(\gamma_{n}b),F_n,\xi^{\Theta}(\gamma_na),\xi^{\Theta}(\gamma_n^{+}))
$$
is positive and for some $\epsilon>0$, 
\changed{
$$
d_{\Theta}(\xi^{\Theta}(x),F_n)>\epsilon.
$$
}
Let $c\in \Lambda(\Gamma)$ such that $(\gamma_n^{-},c,\gamma_n(b),\gamma_n(a),\gamma_n^{+})$ is positive for $n$ large enough, so by \cite[Proposition 13.26]{GW2} we have
$$
(\xi^{\Theta}(\gamma_n^{-}),\xi^{\Theta}(c),\xi^{\Theta}(\gamma_{n}b),F_n,\xi^{\Theta}(\gamma_na),\xi^{\Theta}(\gamma_n^{+}))
$$
is positive for $n$ large enough. Note that the positive triple \refchange{179}{\changed{$(\gamma_n^{-},c,\gamma_n^{+})$}} converges to the positive triple \refchange{180}{\changed{$(y,c,x)$}}. \refchange{181}{\changed{Repeating the proof of Lemma \ref{posconvergence}}}, we can find a sequence $g_n\in Aut_1(\mathfrak{g})$ converging in $Aut_1(\mathfrak{g})$,  \refchange{182}{\changed{sequences $x_{i,n}\in U_{\Theta}^{>0}, i = 1,2,3,4$}} converging in $U_{\Theta}^{\ge 0}$ such that 
$$
(\xi^{\Theta}(\gamma_n^{-}),\xi^{\Theta}(c),\xi^{\Theta}(\gamma_{n}b),F_n,\xi^{\Theta}(\gamma_na),\xi^{\Theta}(\gamma_n^{+}))
$$
$$
=g_n(\mathfrak{p}_{\Theta},x_{1,n}x_{2,n}x_{3,n}x_{4,n}\mathfrak{p}_{\Theta}^{opp},x_{1,n}x_{2,n}x_{3,n}\mathfrak{p}_{\Theta}^{opp},x_{1,n}x_{2,n}\mathfrak{p}_{\Theta}^{opp},x_{1,n}\mathfrak{p}_{\Theta}^{opp}, \mathfrak{p}_{\Theta}^{opp})
$$
where $x_{1,n}x_{2,n}x_{3,n}x_{4,n}$ converges in $U_{\Theta}^{>0}$. \refchange{183}{\changed{It should be noted that both $\gamma_n a$ and $\gamma_n b$ converge to $x$, so we cannot claim the convergence of $x_{1,n}, x_{2,n}, x_{3,n}$ in $U_{\Theta}^{>0}$, but only in $U_{\Theta}^{\ge 0}$.}}

\changed{As $\xi^{\Theta}(\gamma_n b)\to \xi^{\Theta}(x)$} and $g_n$ converges, we see that $x_{1,n}x_{2,n}x_{3,n}$ converges to $id$, so the cone coordinate $v_{x_{1,n}x_{2,n}x_{3,n}}\to 0$. From formal group law (Lemma \ref{lem8}), we have $ \pi_T(v_{x_{1,n}})+\pi_T(v_{x_{2,n}})+\pi_T(v_{x_{3,n}})=\pi_T(v_{x_{1,n}x_{2,n}x_{3,n}})\to 0$. \refchange{186}{\changed{Here $\pi_T$ takes values in the tangent cone $T\mathring{c}$ introduced in Section~\ref{sssec15}. Since $T\mathring{c}$ is acute, fix any norm on $\mathfrak{u}_{\Theta}$. Then there exist a functional $l$ and a constant $C>1$ such that for any $v \in T\mathring{c}$, $\tfrac{1}{C}\|v\|\le l(v)\le C\|v\|$ (see the proof of Lemma~\ref{lem15}). Hence, $\pi_T(v_{x_{1,n}x_{2,n}x_{3,n}})\to 0$ implies that each of $\pi_T(v_{x_{1,n}})$, $\pi_T(v_{x_{2,n}})$, and $\pi_T(v_{x_{3,n}})$ converges to $0$.}} 

\refchange{187}{\changed{Applying the above functional $l$ again, we have all of $v_{x_{1,n}}$, $v_{x_{2,n}}$, and $v_{x_{3,n}}$ converge to $0$, so all of $x_{1,n}$, $x_{2,n}$ and $x_{3,n}$ converge to $id$. So all of $\xi^{\Theta}(\gamma_n a)$, $\xi^{\Theta}(\gamma_n b)$ and $F_n$ converge to $\xi^{\Theta}(x)$,}}  contradicting the inequality \changed{$d_{\Theta}(\xi^{\Theta}(x),F_n)>\epsilon$.}

Now suppose $x = y$. Since $\Gamma$ is non-elementary, we may pick $\gamma\in \Gamma$ such that $z = \gamma^{-1}(x)\not = x$, then $\gamma_n\gamma(b_0)\to x$, $(\gamma_n\gamma)^{-1}(b_0)\to z$. By the first case, we see that $\rho(\gamma_n\gamma)(F)\to \xi^{\Theta}(x)$ for all $F$ transverse to $\xi^{\Theta,opp}(z)$. Furthermore, $F$ is transverse to $\xi^{\Theta,opp}(z)$ iff $\rho(\gamma) F$ is transverse to $\xi^{\Theta,opp}(x)$, so we conclude the proof.
\end{proof}

\section{Regular distortion property}\label{sec4}

In this \refchange{189}{\changed{section}} we show that the limit maps of $\Theta$-positive representations have the \emph{regular distortion property}. We let $\Gamma \subset \mathsf{PSL}(2,\mathbb{R})$ be a \refchange[\baselineskip]{190}{\changed{non-elementary}} discrete subgroup, and as before, \refchange[\baselineskip]{191}{\changed{we assume that $G$ is a real simple Lie group with finitely many components, whose identity component $G^{\circ}$ has finite center. Moreover, if $\mathfrak{g}$ denotes the Lie algebra of $G$ with adjoint representation $\psi:G \to Aut(\mathfrak{g})$, then as before, we also assume $\psi(G) \subset Aut_1(\mathfrak{g})$ (see Section~\ref{assumption}).}}

\refchange{192}{\changed{We fix $b_0 \in \mathbb{D}$. For any $\gamma \in \Gamma$, choose $\omega_{\gamma} \in \Lambda(\Gamma)$ such that the Euclidean distance \refchangenew{192}{\changednew{$d_{\mathbb{R}^2}(\gamma^{-1}(b_0), \omega_{\gamma})$}} between $\gamma^{-1}(b_0)$ and $\omega_{\gamma}$ equals the Euclidean distance \refchangenew{192}{\changednew{$d_{\mathbb{R}^2}(\gamma^{-1}(b_0), \Lambda(\Gamma))$}} from $\gamma^{-1}(b_0)$ to $\Lambda(\Gamma)$ (here “Euclidean distance” \refchangenew{192}{\changednew{$d_{\mathbb{R}^2}(-,-)$}} means that we view $\mathbb{D}$ as the unit disc in $\mathbb{R}^2$ \refchangenew{192}{\changednew{with $b_0$ identified as the origin of $\mathbb{R}^2$}}). Similarly, let $\alpha_{\gamma} \in \Lambda(\Gamma)$ be a point that realizes the Euclidean distance between $\gamma(b_0)$ and $\Lambda(\Gamma)$. We refer to $\omega_{\gamma}$ and $\alpha_{\gamma}$ as the \emph{coarse repeller} and the \emph{coarse attractor}, respectively.} \refchangenew{192}{\changednew{It is worth noting that they are not unique and we are making choices.}} }

Let $\Theta$ be a subset of positive simple roots. For each $\alpha \in \Theta$, fix a Riemannian metric $d_{\alpha}$ on $\mathcal{F}_{\alpha}$ and also fix a Riemannian metric $d_{\Theta}$ on $\mathcal{F}_{\Theta}$. Let $d_{S^1}$ be the angular metric from the origin on $S^1$, where the diameter of $S^1$ is $\pi$. We also fix the norms on $\mathfrak{u}_{\Theta}$ and $\mathring{c}_{\gamma}$ and denote them all by $\|\|$.

\begin{definition}\label{def13}
\begin{enumerate}
\item Let $(x,y,z)$ be a triple in $S^1$. Given any $0<d<\frac{2\pi}{3}$, we say $(x,y,z)$ is \emph{$d$-bounded} if $d_{S^1}(x,y)$, $ d_{S^1}(x,z)$, $ d_{S^1}(y,z)\refchange{213}{\changed{\ge}} d$.
\item A $\Theta$-transverse representation $\rho:\Gamma\to G$ has the \emph{regular distortion property} if for any $0<d<\frac{2\pi}{3}$, there is a constant $C>1$ depending only on $d$, such that for any $\alpha\in \Theta$ and any $\gamma\in\Gamma$, \refchange{193}{\changed{if $p,q\in \Lambda(\Gamma)$ and $(\omega_{\gamma},p,q)$ is $d-$bounded, then}}
$$
C e^{-\alpha(\kappa(\rho(\gamma)))}\ge d_{\alpha}(\rho(\gamma)\xi^{\alpha}(p),\rho(\gamma)\xi^{\alpha}(q))\ge \frac{1}{C} e^{-\alpha(\kappa(\rho(\gamma)))}
$$
where $\xi^{\alpha}$ the the limit map from $\Lambda(\Gamma)$ to $\mathcal{F}_{\alpha}$.
\end{enumerate}
\end{definition}

\begin{theorem}\label{thm12}
    Assume that $G$ has $\Theta$-positive structure. If $\rho$ is $\Theta$-positive, then $\rho$ has the regular distortion property.
\end{theorem}

\subsection{Proof of Theorem \ref{thm12} }\label{ssec7}
The proof is by contradiction and the idea is: if $x\in U_{\Theta}^{opp,>0}$ is near the identity, then the segment joining $\mathfrak{p}_{\Theta}, x \mathfrak{p}_{\Theta}$ on the limit set can be approximated by  $\pi_T^{opp}(v_x^{opp})$ in the tangent cone. Since the cone is preserved by \refchange{194}{\changed{$\psi( \exp \mathfrak{a})$}}, Lemma \ref{lem7} shows that we can use the maximal modulus of eigenvalues to estimate the distortion, which turns out to be the simple roots.

We recall some properties of $Aut_1(\mathfrak{g})$ first. Let $G_0,G_1,...,G_k$ be the connected components of $Aut_1(\mathfrak{g})$ where $G_0= Inn(\mathfrak{g})$ is the identity component. Let $\psi:G\to Aut(\mathfrak{g})$ be the adjoint representation, and recall that we have assumed $\psi(G)\subset Aut_1(\mathfrak{g})$.

Lemma \ref{l1} shows that $Aut_1(\mathfrak{g})$ acts on $\mathcal{F}_{\Theta}$ and $\mathcal{F}_{\Theta}^{opp}$ and preserves transverse pairs. Let $P_{\Theta}^{1}=\mathrm{Stab}_{Aut_1(\mathfrak{g})}(\mathfrak{p}_{\Theta})$, $P_{\Theta}^{1,opp}=\mathrm{Stab}_{Aut_1(\mathfrak{g})}(\mathfrak{p}_{\Theta}^{opp})$ and $L_{\Theta}^{1}=\mathrm{Stab}_{Aut_1(\mathfrak{g})}(\mathfrak{p}_{\Theta})\cap \mathrm{Stab}_{Aut_1(\mathfrak{g})}(\mathfrak{p}_{\Theta}^{opp})$.  For each connected component $G_i\subset Aut_1(\mathfrak{g})$ we choose $\tilde{g}_i \in G_i$. Let $G^{\circ}$ be the identity component of $G$. As $\psi(G^{\circ})=Inn(\mathfrak{g})$ acts transitively on transverse pairs, one can further pick $g_i\in G^{\circ}$ such that $ \tilde{g_i}(\mathfrak{p}_{\Theta},\mathfrak{p}_{\Theta}^{opp}) = \psi(g_i)(\mathfrak{p}_{\Theta},\mathfrak{p}_{\Theta}^{opp})$. Denote $a_i = \psi(g_i^{-1}) \tilde{g_i}\in L_{\Theta}^{1}$, then $G_i = a_i G_0$, and
\refchange{195}{\changed{
$$
L_{\Theta}^{1} = \bigcup_{0\le i\le k} a_i\psi( L_{\Theta}\cap G^{\circ}),\quad 
P_{\Theta}^{1} = \bigcup_{0\le i\le k} a_i \psi(P_{\Theta}\cap G^{\circ}),\quad
P_{\Theta}^{1,opp} = \bigcup_{0\le i\le k} a_i \psi(P_{\Theta}^{opp}\cap G^{\circ}).
$$
}}

We now prove a technical lemma about simplifying the parametrization expression.
\begin{lemma}\label{l2}
    Let $(\gamma_n)\in \Gamma$ be an unbounded sequence. Assume that there exists distinct $a,b\in \Lambda(\Gamma)$ such that \changed{$\omega_{\gamma_n}\to a$}. Choose a closed arc $[a,b]_{arc}\subset S^1$ joining $a$ and $b$,  and let $(a,b)_{arc}$ denote its interior. Let $E\subset (a,b)_{arc}\cap \Lambda(\Gamma)$ be a closed subset. 

\changed{
Let $(g_n)$ be a sequence in $Aut_1(\mathfrak{g})$ and let $(x_n)$ be a sequence of maps $x_n:E\to U_{\Theta}^{opp,>0}$, satisfying the following properties from \refchange{201}{\changed{Corollary \ref{posconv2}~(2)}}:
    \begin{itemize}
        \item $g_n$ converges to $g\in Aut_1(\mathfrak{g})$.
        \item For any $n\in \mathbb{Z}^{>0}$ and any $p\in E$,
        \[
        \xi^{\Theta}(\changed{\omega_{\gamma_n}},p,b) =g_n(\mathfrak{p}_{\Theta}^{opp},x_n(p)\mathfrak{p}_{\Theta},\mathfrak{p}_{\Theta}). 
        \]
        \item $x_n$ converges uniformly to a continuous map $x_{\infty}:E\to U_{\Theta}^{opp,>0}$. And for any $n\in \mathbb{Z}^{>0}\cup \{\infty\}$, and any $p,q\in E$ such that $(a,p,q,b)$ is positive, \refchange{202}{\changed{we have}}
        \[
        x_n(q)^{-1}x_n(p)\in U_{\Theta}^{opp,>0}.
        \]
    \end{itemize}
}

    If $\rho(\gamma_n) = \mu_n l_n\nu_n$ is a Cartan decomposition of $\rho(\gamma_n)$, then for $n\gg1$, there exist $u_n^{-}\in U_{\Theta}^{opp}$, $l'_n\in L_{\Theta}^{1}$ and continuous maps $x'_n: E\to U_{\Theta}^{opp,>0}$ such that $u_n^{-},l'_n$ \refchange{196}{\changed{are}} bounded, $x'_n$ uniformly converges to $x_{\infty}:E\to U_{\Theta}^{opp,>0}$, and the following holds
    \begin{enumerate}
        \item  For any $p\in E$
                $$
                \rho(\gamma_n)\xi^{\Theta}(p) = \mu_n l_n u_n^{-} l_n'x'_n(p)\mathfrak{p}_{\Theta}.
                $$

        \item  Moreover, if  $p,q\in E$ and $(a,p,q,b)$ is positive, then for any $n\in \mathbb{Z}^{>0}$, \refchange{199}{\changed{we have}} $x'_n(q)^{-1}x'_n(p)\in U_{\Theta}^{opp,>0}$.

        \item   Furthermore, if $\nu_n$ converges in $K$, then $u_n^{-}$ converges in $U_{\Theta}^{opp}$ \refchange{200}{\changed{and}} $l'_n$ converges in $L_{\Theta}^{1}$.
    \end{enumerate}

\end{lemma}

\begin{proof}

    \changed{As $\Gamma$ is discrete and $\gamma_n$ is unbounded, $\gamma_n^{-1}(b_0)$ is also unbounded in $\mathbb{D}$. And $\omega_{\gamma_n}\to a$ implies that $\gamma_n^{-1}(b_0)\to a$.}

    Recall that $\Theta-$positive representations are $\Theta$-transverse (Proposition~\ref{thm17}). The Cartan property \refchange{203}{\changed{(Proposition \ref{cartanproperty})}} implies that $\nu_n^{-1} \mathfrak{p}_{\Theta}^{opp}\to \xi^{\Theta}(a)$. As $\omega_{\gamma_n}\to a$, we also have $ g_n \mathfrak{p}_{\Theta}^{opp}\to \xi^{\Theta}(a) $. So $\big(\psi(\nu_n) g_n \big)\mathfrak{p}_{\Theta}^{opp}\to \mathfrak{p}_{\Theta}^{opp}$.  As $g_n$ converges and $\nu_n$ is bounded, we can see that there exists a compact subset $K'\subset P_{\Theta}^{1,opp}$, such that any converging subsequence of $\psi(\nu_{n_k}) g_{n_k}$ will converge to some $k'\in K'$. So for any open neighborhood $U'$ of $K'$ in $Aut_1(\mathfrak{g})$, there exists $N_{U'}\in \mathbb{Z}^{>0}$ such that if $n>N_{U'}$, then $\psi(\nu_n) g_n\in U'$.

Now we construct the sequences $u_n^-$ and $l'_n$. Note that $P_{\Theta}^{1,opp}P_{\Theta}^{1}$ is an open subset of $Aut_1(\mathfrak{g})$ (it is the union of the coset of the Bruhat cell, i.e., \refchange{205}{\changed{the union of  $a_{i}a_j\psi(P_{\Theta}^{opp}P_{\Theta}\cap G^{\circ}),0\le i,j\le k$}}). So for $n\gg1$, $\psi(\nu_n)g_n\in P_{\Theta}^{1,opp}P_{\Theta}^{1}$. Applying the usual Levi decomposition (Proposition \ref{Levi}) to these cosets, we can find a sequence $(u_n^{-},l'_n, u_n^{+})\in U_{\Theta}^{opp}\times L_{\Theta}^{1}\times U_{\Theta}$ such that $\psi(\nu_n)g_n = \psi(u^{-}_n) l'_n \psi(u^{+}_n)$, $u_n^{+}\to id$, and $u_n^-, l'_n$ are bounded sequences. Note that if $\nu_n$ converges, then clearly $u_n^{-},l'_n$ can be chosen to converge too. This proves item~(3) in the statement of the lemma.

Now we have for any $p\in E$,
 $$
 \rho(\gamma_n)\xi^{\Theta}(p)= \mu_n l_n u_n^{-} l_n' u_n^{+}x_n(p) \mathfrak{p}_{\Theta}.
 $$
\refchange{206}{\changed{We omit the notation $\psi$ here, since in Section~\ref{assumption} we noted that $\psi(g)\mathfrak{p}_{\Theta}$ may be abbreviated as $g\mathfrak{p}_{\Theta}$ whenever $\psi(G) \subset Aut_1(\mathfrak{g})$.}} \refchange{209}{\changed{It should be noted that although $u_n^+ \to id$, we can not directly claim that $u_n^+ x_n(p)\in U_{\Theta}^{opp,>0}$ because $u_n^+\not \in U_{\Theta}^{opp}$. To handle this, we reverse the coordinate so as to absorb $u_n^+$, and thus get the desired $x'_n$, as follows:}}

Theorem \ref{thm11} (7) implies that we can \refchange{208}{\changed{change the coordinates from $U_{\Theta}^{opp,>0}\mathfrak{p}_{\Theta}$ to $U_{\Theta}^{>0}\mathfrak{p}_{\Theta}^{opp}$ to parameterize the positive configurations.}} That is, we can find a sequence of continuous maps $z_n:E\to U_{\Theta}^{>0}$ which uniformly converges to a continuous map $z_{\infty}:E\to U_{\Theta}^{>0}$, such that for any $p\in E$ and any $k\in \mathbb{Z}^{>0}\cup \{\infty\}$, we have
$$
x_n(p) \mathfrak{p}_{\Theta} = z_n(p)\mathfrak{p}_{\Theta}^{opp}.
$$
Since $p\mapsto x_n(p) \mathfrak{p}_{\Theta}$ is a positive map on $E$, so is $p\mapsto z_n(p)\mathfrak{p}_{\Theta}^{opp}$. So for any $n\in \mathbb{Z}^{>0}\cup \{\infty\}$ and any $p,q\in E$ such that $(a,p,q,b)$ is positive, \refchange{207}{\changed{we have}}
$$
z_n(p)^{-1}z_n(q)\in U_{\Theta}^{>0}.
$$

Now for any $p\in E$, we have
$$
u_n^{+}x_n(p)\mathfrak{p}_{\Theta}=u_n^{+}z_n(p)\mathfrak{p}_{\Theta}^{opp}.
$$
Note that $E\subset (a,b)_{arc}$ is a closed (thus compact) subset, which together with $z_n\to z_{\infty}$ implies that $\cup_{n \in \mathbb{Z}} z_n(E)$ is contained in a compact subset of $U_{\Theta}^{>0}$. Also note that $U_{\Theta}^{>0}$ is an open subset of $U_{\Theta}$, so $u_n^{+}\to id$ implies that for $n$ large enough, $u_n^{+} z_n(p)\in U_{\Theta}^{>0}$ holds for any $p\in E$. \refchange{209}{\changed{Again apply the change of coordinates between $U_{\Theta}^{opp,>0}\mathfrak{p}_{\Theta}$ and $U_{\Theta}^{>0}\mathfrak{p}_{\Theta}^{opp}$ as above,}} we uniquely determine $x'_n:E\to U_{\Theta}^{opp,>0}$ by requiring 
$$
x'_n(-) \mathfrak{p}_{\Theta} =u_n^{+}z_n(-)\mathfrak{p}_{\Theta}^{opp}.
$$
As a result, we have for any $p\in E$,
$$
\rho(\gamma_n)\xi^{\Theta}(p) = \mu_n l_n u_n^{-} l_n'x'_n(p)\mathfrak{p}_{\Theta}
$$
as required in item~(1) of the statement.

\refchangenew{198}{\changednew{Since $u_n^+\to id$ and $z_n$ converges to $z_{\infty}$ uniformly, the map $p\mapsto x'_n(p)\mathfrak{p}_{\Theta} =  u_n^{+}z_n(p)\mathfrak{p}_{\Theta}^{opp},p\in E$ converges to the map $p\mapsto  z_{\infty}(p)\mathfrak{p}_{\Theta}^{opp}, p\in E$ uniformly, which is exactly \(p\mapsto x_{\infty}(p) \mathfrak{p}_{\Theta}, p \in E \), thus $x'_n$ converges to $x_{\infty}$ uniformly.}} On the other hand, the map $p\mapsto u_n^+z_n(p)\mathfrak{p}_{\Theta}^{opp}$ is a positive map on $E$, so is $p\mapsto x'_n(p)\mathfrak{p}_{\Theta}$, which implies item~(2) in the statement.
\end{proof}

\begin{proof}[Proof of Theorem \ref{thm12}]
    If the theorem is false, then there exists a $d>0$, $\alpha\in \Theta$ and sequences $p_n,q_n\in \Lambda(\Gamma),\gamma_n\in \Gamma$ such that the triple $(\changed{\omega_{\gamma_n}},p_n,q_n)$ is $d-$bounded and either
    $$
    d_{\alpha}(\rho(\gamma_n)\xi^{\alpha}(p_n),\rho(\gamma_n)\xi^{\alpha}(q_n))\le \frac{1}{n} e^{-\alpha(\kappa(\rho(\gamma_n)))}
    $$
    or
    $$
    d_{\alpha}(\rho(\gamma_n)\xi^{\alpha}(p_n),\rho(\gamma_n)\xi^{\alpha}(q_n))\ge n e^{-\alpha(\kappa(\refchange{212}{\changed{\rho(\gamma_n)}}))}
    $$
    holds for all $n$.

   Take a subsequence so we can assume that $(\changed{\omega_{\gamma_n}},p_n,q_n)$ converges to $(a,p,q)$ which is still $d-$bounded.  And after taking a subsequence, we assume $\mu_n\to \mu,\nu_n\to \nu$ where $\rho(\gamma_n) = \mu_nl_n \nu_n$ is a Cartan decomposition. 

   We first note that $\gamma_n$ is unbounded. Otherwise,  $d_{\alpha}(\rho(\gamma_n)\xi^{\alpha}(p_n),\rho(\gamma_n)\xi^{\alpha}(q_n))\ge n e^{-\alpha(\kappa(\refchange{215(a)}{\changed{\rho(\gamma_n)}}))}$ can not hold for infinitely many $n$ as left hand side is bounded and right hand side diverges.  Similarly $d_{\alpha}(\rho(\gamma_n)\xi^{\alpha}(p_n),\rho(\gamma_n)\xi^{\alpha}(q_n))\le \frac{1}{n} e^{-\alpha(\kappa(\rho(\gamma_n)))}$ can not hold for infinitely many $n$ as left hand side is bounded away from $0$ but right hand side converges to $0$. \changed{So $\gamma_n^{-1}(b_0)$ is also unbounded in $\mathbb{D}$. Now $\omega_{\gamma_n}\to a$ implies that $\gamma_n^{-1}(b_0)\to a$ too.}

   \refchange{215(b), 215(c)}{\changed{Pick $b\in \Lambda(\Gamma)$ such that $(a,p,q,b)$ is positive. By Lemma \ref{l2}, there exist}} converging sequences $u_n^{-}\in U_{\Theta}^{opp},l'_n\in L_{\Theta}^{1},x'_n,y'_n\in U_{\Theta}^{opp,>0}$ where $x_n'\to x'\in U_{\Theta}^{opp,>0}, y'_n\to y'\in U_{\Theta}^{opp,>0}$ such that
    $$
 \rho(\gamma_n)(\xi^{\Theta}(p_n), \xi^{\Theta}(q_n)) = \mu_n l_n u_n^{-} l_n'(x'_ny'_n\mathfrak{p}_{\Theta},x'_n \mathfrak{p}_{\Theta}).
    $$
We first suppose $l'_n=\psi(l''_n)$ for some \refchange{215(d)}{\changed{$l''_n\in L_{\Theta}^{\circ}$}} so that $l'_n$ preserves $U_{\Theta}^{opp,>0}$. 

As $l_n'$ converges (thus bounded), \refchange{215(e)}{\changed{define $x''_n = \mathsf{c}_{l''_n} (x'_n), y''_n = \mathsf{c}_{l''_n} (y'_n)$ (recall that for any $g\in G$, $\mathsf{c}_g(\cdot)$ denotes the $g$-conjugate action on $G$), so $x''_n$ and $y''_n$ still converge in $U_{\Theta}^{opp,>0}$.}} So we have 
$$
\rho(\gamma_n)(\xi^{\Theta}(p_n), \xi^{\Theta}(q_n)) = \mu_n l_n u_n^{-} (x''_n y''_n \mathfrak{p}_{\Theta},x''_n \mathfrak{p}_{\Theta})
$$
$$
=\mu_n \mathsf{c}_{l_n}(u_n^{-}) (\mathsf{c}_{l_n}(x''_n) \mathsf{c}_{l_n}(y''_n)\mathfrak{p}_{\Theta}, \mathsf{c}_{l_n}(x''_n)\mathfrak{p}_{\Theta}).
$$
By Proposition \ref{thm17} $\rho$ is $\Theta$-transverse. Since $\gamma_n$ is unbounded and $\gamma_n^{-1}(b_0)\to a$, by Proposition \ref{cartanproperty}~(2) we have \refchange{215(f)}{\changed{$\min_{\alpha\in \Theta}\alpha (\kappa(l_n)) = \min_{\alpha\in \Theta}\alpha (\kappa(\rho(\gamma_n)))  \to \infty$. So from the definition of $\mathfrak{u}_{\Theta}^{opp}$, we have for any $v\in \mathfrak{u}_{\Theta}^{opp}$, $Ad_{l_n} (v)\to 0$.}} As $U_{\Theta}^{opp} = \exp \mathfrak{u}_{\Theta}^{opp}$, the convergence of $u_n^{-}, x''_n, y''_n$ in $U_{\Theta}^{opp}$ imply that $\mathsf{c}_{l_n}(u_n^{-}), \mathsf{c}_{l_n}(x''_n), \mathsf{c}_{l_n}(y''_n)$ converge to id. As there is an equivariant projection $\pi^{\Theta}_{\alpha}$ from $\mathcal{F}_{\Theta}$ to $\mathcal{F}_{\alpha}$, we see that
$$
(\rho(\gamma_n)\xi^{\alpha}(p_n), \rho(\gamma_n)\xi^{\alpha}(q_n)) = \mu_n \mathsf{c}_{l_n}(u_n^{-}) (\mathsf{c}_{l_n}(x''_n) \mathsf{c}_{l_n}(y''_n)\mathfrak{p}_{\alpha}, \mathsf{c}_{l_n}(x''_n)\mathfrak{p}_{\alpha})
$$

Recall that $\mathfrak{u}_{\{\alpha\}}^{opp}$ is the nilpotent subalgebra of the parabolic subgroup $P_{\alpha}^{opp}$. The exponential map
$$
\exp:\mathfrak{u}_{\{\alpha\}}^{opp}\to \mathcal{F}_{\alpha}
$$
$$
X\mapsto (\exp X) \mathfrak{p}_{\alpha}
$$
is a diffeomorphism onto an open set $\tilde{U}$ containing \refchange{217}{\changed{$\mathfrak{p}_{\alpha}$}}. So the Riemannian norm $\|\cdot \|$ on $\mathfrak{u}_{\{\alpha\}}^{opp}\cong T_{\mathfrak{p}_{\alpha}} \mathcal{F}_{\alpha}$ defines a Riemannian metric $d'_{\alpha}$ on $\tilde{U}$. \refchange{218}{\changed{Pick a sufficiently small precompact neighborhood $U \subsetneq \tilde{U}$ of $\mathfrak{p}_{\alpha}$}}, we may assume $d'_{\alpha}$ and $d_{\alpha}$ are bi-Lipschitz equivalent on $U$. \changed{As a result of $(\rho(\gamma_n)\xi^{\alpha}(p_n), \rho(\gamma_n)\xi^{\alpha}(q_n)) = \mu_n \mathsf{c}_{l_n}(u_n^{-}) (\mathsf{c}_{l_n}(x''_n) \mathsf{c}_{l_n}(y''_n)\mathfrak{p}_{\alpha}, \mathsf{c}_{l_n}(x''_n)\mathfrak{p}_{\alpha})$, $\mu_n\to \mu$, and $\mathsf{c}_{l_n}(u_n^{-}), \mathsf{c}_{l_n}(x''_n), \mathsf{c}_{l_n}(y''_n)$ converging to id, we have}
$$
d_{\alpha}(\rho(\gamma_n) \xi^{\alpha}(p_n), \rho(\gamma_n) \xi^{\alpha}(q_n))
$$
$$
\approx d_\alpha(\mathsf{c}_{l_n}(x''_ny''_n)\mathfrak{p}_{\alpha},\mathsf{c}_{l_n}(x''_n)\mathfrak{p}_{\alpha})
$$
$$
\approx d'_{\alpha}(\mathsf{c}_{l_n}(x''_ny''_n)\mathfrak{p}_{\alpha}, \mathsf{c}_{l_n}(x''_n)\mathfrak{p}_{\alpha})
$$
$$
= \|{\tilde{\pi}^{opp}_{\alpha}}(\log (\mathsf{c}_{l_n}(x''_ny''_n))- \log (\mathsf{c}_{l_n}(x''_n)))  \|.
$$
\changed{Here $\tilde{\pi}_{\alpha}^{opp}$ denotes the linear projection from $\mathfrak{u}_{\Theta}^{opp}$ to $\mathfrak{u}_{\{\alpha\}}^{opp}$ and the notation $A_n\approx B_n$ means there exists a constant $C>1$ such that $\frac{1}{C}B_n\le A_n\le C B_n$.  The last equality holds because the differential of the natural projection $\mathcal{F}_{\Theta}\to \mathcal{F}_{\alpha}$, $g\mathfrak{p}_{\Theta}\mapsto g\mathfrak{p}_{\alpha}$, at the point $\mathfrak{p}_{\Theta}$ is exactly $\tilde{\pi}^{opp}_{\alpha}$.} 

\refchange{220}{\changed{As $\mathsf{c}_{l_n}(x''_n), \mathsf{c}_{l_n}(y_n'')\to \mathrm{id}$, both $\|v^{opp}_{\mathsf{c}_{l_n}(x''_ny''_n)}\|$ and $\|v^{opp}_{\mathsf{c}_{l_n}(x_n'')}\|$ converge to $0$. Applying Theorem~\ref{thm10}~(2)~(a), we have
\[
\lim_{n \to \infty} \frac{\|\tilde{\pi}^{opp}_{\alpha}(\log (\mathsf{c}_{l_n}(x''_ny''_n))\|}{\|\pi_{\alpha}^{opp}(v^{opp}_{\mathsf{c}_{l_n}(x''_ny''_n)}))\|} = \lim_{n \to \infty} \frac{\|\tilde{\pi}^{opp}_{\alpha}(\log (\mathsf{c}_{l_n}(x''_n)))\|}{\|\pi_{\alpha}^{opp}(v^{opp}_{\mathsf{c}_{l_n}(x''_n)})\|}=1,
\]
}}
so for $n\gg 1$,
$$
\|{\tilde{\pi}^{opp}_{\alpha}}(\log (\mathsf{c}_{l_n}(x''_ny''_n))- \log (\mathsf{c}_{l_n}(x''_n)))  \|\approx \|\pi_{\alpha}^{opp}(v^{opp}_{\mathsf{c}_{l_n}(x''_ny''_n)})-\pi_{\alpha}^{opp}(v^{opp}_{\mathsf{c}_{l_n}(x''_n)})\|,
$$
and formal group law (Lemma \ref{lem8}) shows that
$$
\|\pi_{\alpha}^{opp}(v^{opp}_{\mathsf{c}_{l_n}(x''_ny''_n)})-\pi_{\alpha}^{opp}(v^{opp}_{\mathsf{c}_{l_n}(x''_n)})\| = \| \pi_{\alpha}^{opp}(v^{opp}_{\refchange{221}{\changed{\mathsf{c}_{l_n}}}(y_n'')}) \| = \| Ad_{l_n} \pi^{opp}_{\alpha}(v^{opp}_{y_n''})\|.
$$

Note that $y_n''$ converges in $U_{\Theta}^{opp,>0}$, so $\pi_{\alpha}^{opp}(v_{y''_n}^{opp})$ converges in $\mathring{c}_{\alpha}^{opp}$. Also recall that by our assumption $l_n \in \psi (L_{\Theta}^{\circ})$, $\mathring{c}_{\alpha}^{opp}$ is an open cone preserved by $Ad_{l_n}$ in the weight subspace
$$
\mathfrak{u}_{\alpha}^{opp} = \sum_{\beta\in -\alpha - \mathrm{Span}_{\mathbb{Z}_{\ge 0}}(\Delta-\Theta)} \mathfrak{g}_{\beta}.
$$
So the unique highest weight with respect to the adjoint action of the Cartan subalgebra $\mathfrak{a}$ is $-\alpha\in \mathfrak{a}^*$. So we see that the maximal modulus of the eigenvalues of $Ad_{l_n}$ is $\lambda_n:= e^{-\alpha(\kappa(l_n))} = e^{-\alpha(\kappa(\rho(\gamma_n)))} $.

Applying Lemma \ref{lem7} to the cone $\mathring{c}_{\alpha}^{opp}$ and the sequence $Ad_{l_n}$, there exists a $C>1$ such that
$$
\frac{1}{C} e^{-\alpha(\kappa(\rho(\gamma_n))))} \| \pi_{\alpha}^{opp}(v_{y''_n}^{opp}) \| \le \| Ad_{l_n} (\pi_{\alpha}^{opp}(v_{y''_n}^{opp}))\|\le C e^{-\alpha(\kappa(\rho(\gamma_n)))} \| \pi_{\alpha}^{opp}(v_{y''_n}^{opp}) \|.
$$
Using the convergence of $y''_n$ in $U_{\Theta}^{opp,>0}$ again, there exists a $C'>1$ such that
$$
\frac{1}{C'} e^{-\alpha(\kappa(\rho(\gamma_n)))}\le d_{\alpha}(\rho(\gamma_n)\xi^{\alpha}(p_n), \rho(\gamma_n)\xi^{\alpha}(q_n))\le C' e^{-\alpha(\kappa(\rho(\gamma_n)))}
$$
which is a contradiction. 

 \refchange{222}{\changed{Finally, we consider the general case where $l'_n\in L_{\Theta}^{1}-\psi(L_{\Theta}^{\circ})$ might happen.}} Guichard-Wienhard~\cite[Corollary 5.3]{GW2} claims that $L_{\Theta}^{1}$ preserves $\mathring{c}_{\alpha}^{opp}\cup - \mathring{c}_{\alpha}^{opp}$ for each $\alpha\in \Theta$. \refchange{223}{\changed{So although the conjugation of $l'_n$ may no longer preserve $U_{\Theta}^{opp,>0}$ (since there may exist $\alpha \in \Theta$ such that $Ad_{l'_n} \mathring{c}_{\alpha} = -\mathring{c}_{\alpha}$), it still conjugates $U_{\Theta}^{opp,>0}$ to another positive semigroup obtained by replacing some $\mathring{c}_{\alpha}$ with $-\mathring{c}_{\alpha}$ in Definition~\ref{def7}. Hence, the above discussion still applies. }}

\end{proof}

\subsection{Shadow Lemmas}\label{ssec6}
We define the notions of shadows.
\begin{definition}\label{def12}
Let $\mathbb{D}$ be the Kleinian hyperbolic disc with Hilbert metric $d_{\mathbb{D}}$. Given $b_0\in \overline{\mathbb{D}},z\in \mathbb{D}$, define the \emph{shadow} of radius $r$ centered at $z$ from $b_0$ as 
$$
\mathcal{O}_r(b_0,z):=\{x|x\in \partial \mathbb{D}, \changed{\overrightarrow{b_0x}}\cap \overline{B_{d_{\mathbb{D}}}(z,r)}\not = \emptyset \},
$$
where $\changed{\overrightarrow{b_0x}}$ denotes the geodesic (when $b_0\in \partial \mathbb{D}$) or the geodesic ray (when $b_0\in \mathbb{D}$) from $b_0$ passing through $x$, and $B_{d_{\mathbb{D}}}(z,r)$ denote the ball of radius $r$ centered at $z$ with respect to the metric $d_{\mathbb{D}}$. 
\end{definition}
Suppose $\Gamma\subset \mathsf{PSL}(2,\mathbb{R})$ is a discrete subgroup with limit set $\Lambda(\Gamma)$. Denote
$$
\hat{\mathcal{O}}_r(b_0,z) = \mathcal{O}_r(b_0,z)\cap \Lambda(\Gamma).
$$
The \emph{conical limit set} $\Lambda_c(\Gamma)$ is defined as
\[\Lambda_c(\Gamma):=\{x\in \Lambda(\Gamma)\mid\exists \refchange{224}{\changed{\text{ a sequence of }\gamma_n}}\in \Gamma \text{ and }R>0, \text{ s.t. }\gamma_n(b_0)\to x  \text{ and } d_{\mathbb{D}}(\gamma_n(b_0),\changed{\overrightarrow{b_0x}})\le R\}\] and note that $\Lambda_c(\Gamma)$ is independent of $b_0$. The \emph{$R-$uniformly conical limit set} with respect to $b_0\in\mathbb{D}$ is defined as
$$
\Lambda_{b_0,R}(\Gamma) = \{x\mid x\in \Lambda(\Gamma), \changed{\overrightarrow{b_0x}}\subset \Gamma(\overline{B_{d_{\mathbb{D}}}(b_0,R}))\}.
$$
Denote $\hat{\mathcal{O}}_{r,R}(b_0,z) = \mathcal{O}_r(b_0,z)\cap \Lambda_{b_0,R}(\Gamma)$.

For any subset $\Theta\subset \Delta$, in the cases where there is a map $\xi^{\Theta}: \Lambda(\Gamma)\to \mathcal{F}_{\Theta}$, define $\mathcal{O}^{\Theta}_r(b_0,z) := \xi^{\Theta}(\hat{\mathcal{O}}_r(b_0,z))$ and $\mathcal{O}^{\Theta}_{r,R}(b_0,z) := \xi^{\Theta}(\hat{\mathcal{O}}_{r,R}(b_0,z))$.

\vspace{3mm}

Assume $\Gamma$ is non-elementary. Fix \(b_{0}\in\mathbb{D}\) and \(R\ge0\). We say that \(r>0\) is a \emph{sufficiently large radius} with respect to \(b_0\) and \(R\) if there exists \(\epsilon>0\) such that for any \(\gamma\in \Gamma\) and \(z\in \mathbb{D}\), the conditions \(d_{\mathbb{D}}(b_0,z)>r\) and \(d_{\mathbb{D}}(z,\gamma(b_0))\le R\) imply that there exist \(x,y \in \mathcal{O}_{r}(\gamma^{-1}(b_0),\gamma^{-1}(z))\cap \Lambda(\Gamma)\) with \(d_{S^1}(x,y)>\epsilon\).

\begin{claim}
\refchangenew{225,227}{\changednew{For any \(R\) and \(b_0\) there exists a sufficiently large radius \(r\).}}       
\end{claim}

\begin{proof}
Suppose the required \(r\) does not exist. Then there exist a diverging sequence \(r_n\to\infty\), a sequence \(\gamma_n\in\Gamma\), and a sequence \(z_n\in\mathbb{D}\) such that \(d_{\mathbb{D}}(b_0,z_n)>r_n\), \(d_{\mathbb{D}}(z_n,\gamma_n(b_0)) \le R\), and the \(d_{S^1}\)-diameter of \(\refchangenew{225,227}{\changednew{\hat{\mathcal{O}}_{r_n}}}(\gamma_n^{-1}(b_0),\gamma_n^{-1}(z_n))\) is not greater than \(1/n\).

Since \(d_{\mathbb{D}}(b_0,\gamma_n(b_0))\ge d_{\mathbb{D}}(b_0,z_n)-d_{\mathbb{D}}(z_n,\gamma_n(b_0))\ge r_n-R\to\infty\), we may pass to a subsequence so that \(\gamma_n^{-1}(b_0)\to a\in\Lambda(\Gamma)\). Meanwhile, \(d_{\mathbb{D}}(\gamma_n^{-1}(z_n),b_0)=d_{\mathbb{D}}(z_n,\gamma_n(b_0))\le R\), so up to another subsequence we assume \(\gamma_n^{-1}(z_n)\to z\in\mathbb{D}\). But as \(r_n\to\infty\), the shadows \refchangenew{225,227}{\changednew{\(\hat{\mathcal{O}}_{r_n}(\gamma_n^{-1}(b_0),\gamma_n^{-1}(z_n))\) increase to \(\Lambda(\Gamma)-\{a\}\)}}. As \(|\Lambda(\Gamma)|\ge 3\), we get a contradiction.
\end{proof}

{\color{blue}
Let \(r\) be a sufficiently large radius for \(b_0,R\). For any
\(z\in \mathbb{D}\) and \(\gamma \in \Gamma\) such that \(d_{\mathbb{D}}(b_0,\gamma(b_0))>r\) and
\(d_{\mathbb{D}}(\gamma(b_0),z)\le R\), we identify \(\mathcal{O}_r(b_0,z)\) with a closed interval. Then
\(\mathcal{O}_r(b_0,z)\cap \Lambda(\Gamma)\) is a compact subset of $\mathcal{O}_r(b_0,z)$. With respect to the clockwise order on $S^1$, define the \emph{coarse end points} by
\[
x_{z,r}=\inf_{t\in \mathcal{O}_r(b_0,z)\cap \Lambda(\Gamma)} t,\qquad
y_{z,r}=\sup_{t\in \mathcal{O}_r(b_0,z)\cap \Lambda(\Gamma)} t.
\]
By compactness, both extrema are attained, so \(x_{z,r},y_{z,r}\in \Lambda(\Gamma)\cap\mathcal{O}_r(b_0,z)\).

\refchangenew{225,227}{\changednew{Moreover, note that for fixed $r$, \(d_{S^1}\big(\gamma^{-1}(x_{z,r}),\gamma^{-1}(y_{z,r})\big)>\epsilon'\) for some \(\epsilon'>0\) independent of $\gamma$ and $z$. If this were false, then similar to the proof of the existence of a sufficiently large radius, we assume there exist sequences $\gamma_n\in \Gamma$ and $z_n\in\mathbb{D}$ such that
\begin{enumerate}
    \item $d_{\mathbb{D}}(b_0,z_n)>r$ and $d_{\mathbb{D}}(z_n,\gamma_n(b_0))\le R$.
    \item $\gamma_n^{-1}(b_0)\to a\in \Lambda(\Gamma)$ and $\gamma_n^{-1}(z_n)\to z\in \mathbb{D}$.
    \item $d_{S^1}\big(\gamma_n^{-1}(x_{z_n,r}),\gamma_n^{-1}(y_{z_n,r})\big)<\frac{1}{n}$.
\end{enumerate}
By the definition of a sufficiently large radius, the shadows $\hat{\mathcal{O}}_{r}(\gamma_n^{-1}(b_0),\gamma_n^{-1}(z_n))$ contain two points whose distance is bounded from below by some $\epsilon>0$. Thus, there also exist two such points in $\hat{\mathcal{O}}_r(a,z)$, implying that the endpoints of $\hat{\mathcal{O}}_r(a,z)$ are distinct. However, $\gamma_n^{-1}(x_{z_n,r})$ and $\gamma_n^{-1}(y_{z_n,r})$ converge to these distinct endpoints; consequently, the distance between them is bounded from below, a contradiction.
}}

\vspace{3mm}

We prove the following useful proposition:
\begin{proposition}\label{prop4}
Let $\Gamma\subset \mathsf{PSL}(2,\mathbb{R})$ is a \changed{non-elementary} discrete subgroup and $\rho:\Gamma\to G$ is a $\Theta$-transverse representation with the regular distortion property. For any $b_0\in \mathbb{D}$, $r>0$ and $R\ge 0$ where \changed{$r$ is a sufficiently large radius for $b_0$ and $R$},
there exists a constant $C>1$ depending only on $b_0,r,R$ such that for any  $\alpha\in \Theta$, $\gamma\in\Gamma$, and $z\in \mathbb{D}$ with $d_{\mathbb{D}}(b_0,z)>r$ and $d_{\mathbb{D}}(\gamma(b_0),z)\le R$, \refchange{226}{\changed{we have}}
$$
\frac{1}{C} e^{-\alpha(\kappa(\rho(\gamma)))}\le d_{\alpha}(\xi^{\alpha}(x_{z,r}),\xi^{\alpha}(y_{z,r}))\le C e^{-\alpha(\kappa(\rho(\gamma)))}.
$$
\end{proposition}
\begin{proof}
If not, then there exists a sequence $\gamma_n\in \Gamma$, a sequence $z_n\in \mathbb{D}$, and $\alpha\in\Theta$ such that $d_{\mathbb{D}}(b_0, z_n)>r$, $d_{\mathbb{D}}(\gamma_n(b_0),z_n)\le R$, and one of the following holds for all $n$:
$$
d_{\alpha}(\xi^{\alpha}(x_{z_n,r}),\xi^{\alpha}(y_{z_n,r}))> n e^{-\alpha(\kappa(\rho(\gamma_n)))}
$$
$$
d_{\alpha}(\xi^{\alpha}(x_{z_n,r}),\xi^{\alpha}(y_{z_n,r}))<\frac{1}{n} e^{-\alpha(\kappa(\rho(\gamma_n)))}.
$$

Using the same argument as in the proof of Theorem \ref{thm12}, we can see that $\gamma_n$ is unbounded. Pick a subsequence such that \changed{$\omega_{\gamma_n}$} \refchange{228}{\changed{converges}} to a point $a\in S^1$ and \refchange{229}{\changed{from the definition of $\omega_{\gamma_n}$}} we see that $\gamma_n^{-1}(b_0)$ also converge to $a$. As $d_{\mathbb{D}}(\gamma_n^{-1}(z_n),b_0) = d_{\mathbb{D}}(z_n,\gamma_n(b_0))\le R$, \refchange{230}{\changed{$\gamma_n^{-1}(z_n)$ is a uniformly bounded sequence so by taking a subsequence we can assume that}} it converges to some point $z\in \mathbb{D}$.   \refchangenew{225,227}{\changednew{From the paragraph preceding Proposition~\ref{prop4}, the sequences $p_n = \gamma_n^{-1}(x_{z_n,r})$ and $q_n = \gamma_n^{-1}(y_{z_n,r})$ converge \changed{to distinct points} in the shadow $\mathcal{O}_r(a,z)$}}, so we see that for sufficiently large $n$, \changed{$\omega_{\gamma_n},p_n,q_n$} are $d-$bounded for some $d>0$. Since $\rho$ has the regular distortion property, there exists a constant $C>1$ such that:
$$
d_{\alpha}(\xi^{\alpha}(x_{z_n,r}),\xi^{\alpha}(y_{z_n,r}))
$$
$$
=d_{\alpha}(\rho(\gamma_n)\xi^{\alpha}(p_n),\rho(\gamma_n)\xi^{\alpha}(q_n))\in(\frac{1}{C} e^{-\alpha(\kappa(\rho(\gamma_n)))},C e^{-\alpha(\kappa(\rho(\gamma_n)))}),
$$
which is a contradiction.
\end{proof}

 For any $\alpha\in \Theta$, fix a Riemannian metric $d_{\alpha}$ on $\mathcal{F}_{\alpha}$. When $\Lambda(\Gamma) = S^1$ and $\rho:\Gamma\to G$ is a $\Theta$-positive representation, Proposition \ref{prop3} and Remark \ref{rem16} \refchange{231}{\changed{show}} that $\xi^\alpha(\Lambda(\Gamma))$ is a rectifiable curve, so we can define the arc length measure, \refchange{232}{\changed{or namely, the finite and non-zero $1$-dimensional Hausdorff measure $m_{\alpha}$ with respect to the Riemannian metric on $\mathcal{F_{\alpha}}$. (for more details about the properties of rectifiable curves, see for example, \cite[Chapter 3: Section 3.1 and Section 4: Chapter 7: Theorem 2.4]{steinreal})}}.

We have the following corollary:
\begin{corollary}\label{shadowradius}
Let $\Gamma\subset \mathsf{PSL}(2,\mathbb{R})$ be a \changed{non-elementary discrete subgroup}, $G$ be a simple Lie group with $\Theta$-positive structure and $\rho:\Gamma\to G$ be a $\Theta$-positive representation. Fix $b_0\in \mathbb{D}$. For any $r>0, R\ge 0$, there exists a constant $C>1$ such that the following hold:
\begin{enumerate}
    \item 
    If $\Lambda(\Gamma) = S^1$, then for any $\gamma \in \Gamma$ and $z\in \mathbb{D}$, whenever $d_{\mathbb{D}}(\gamma(b_0),z)\le R$, we have
    \[\frac{1}{C} e^{-\alpha(\kappa(\rho(\gamma)))}\le m_{\alpha}(\mathcal{O}_{r}^{\alpha}(b_0,z))\le C e^{-\alpha(\kappa(\rho(\gamma)))}.\]
    \item 
    Let $x = S^1\cap \changed{\overrightarrow{b_0z}}$. \changed{If $x\in \Lambda_{b_0, R}(\Gamma)$, then for any $\gamma \in \Gamma$ and $z\in \mathbb{D}$, whenever  $d_{\mathbb{D}}(\gamma(b_0),z)<r$, we have}
    \[B_{d_{\alpha}}(\xi^{\alpha}(x),\frac{1}{C} e^{-\alpha(\kappa(\rho(\gamma)))})\cap \xi^{\alpha}(\Lambda(\Gamma))\subset \mathcal{O}^{\alpha}_{r}(b_0,z)\subset B_{d_{\alpha}}(\xi^{\alpha}(x),C e^{-\alpha(\kappa(\rho(\gamma)))}). \]
\end{enumerate}
\end{corollary}

\refchange{234}{{\color{blue}{A proof of item~(1) is given as the proof of Lemma~\ref{dichotomy} \refchangenew{233, 235}{\changednew{(we will explain that $\Lambda(\Gamma)$ is $\alpha$-acute when defining $\alpha$-acute sets (see Definition~\ref{acutesubset} and the discussion following it), so the statement in Lemma~\ref{dichotomy} applies here),}} and a rigorous proof of item~(2) follows from Canary-Zhang-Zimmer~\cite[Proof Theorem 7.1]{CaZhZi} with only minor modifications. For completeness, we include a self-contained proof in Appendix~\ref{append1}.}}}

\subsection{Finiteness of ergodic components}\label{sectionfinitenessofergodiccomponents}

{\color{red}
\begin{definition}\label{rankofcone}
Let $c$ be an open acute convex cone.  
The \emph{rank} $r(c)$ is defined as the infimum of all positive integers $N$ with the following property:  
whenever $v_{0}, v_1, \dots, v_{N} \in \overline{c} - \{0\}$ satisfy
\[
\sum_{i=0}^{N} v_i \in c,
\]
there exists an index $0 \le j \le N$ such that
\[
\sum_{\substack{0 \le i \le N \\ i \neq j}} v_i \in c.
\]
\end{definition}

If $c \subset \mathbb{R}^d$, then \(r(c) \le d\). This is a direct consequence of the conic version of Carathéodory’s theorem: any vector in the conic hull of a set in $\mathbb{R}^d$ can be expressed as a nonnegative linear combination of at most $d$ vectors from that set (see Barvinok~\cite[Theorem~2.3]{barvinokcourse}). Applied to $\sum_{i=0}^{N} v_i \in c$ with $v_i \in \overline{c} - \{0\}$, one can discard all but $d$ of the $v_i$ while staying in $c$; in particular, for $N \ge d$ there exists $0 \le j \le N$ such that $\sum_{i \ne j} v_i \in c$. Hence $r(c) \le d$.
}

\medskip

Note that for each $\alpha \in \Theta$, the measure $m_{\alpha}$ is quasi-invariant \refchange{235}{\changed{for $\rho(\Gamma)$}}. We obtain:

\begin{theorem}\label{ergodictheoreminpaper}
If $\Gamma \subset \mathsf{PSL}(2,\mathbb{R})$ is a lattice and $\rho: \Gamma \to G$ is a $\Theta$-positive representation, then for any $\alpha \in \Theta$, the group $\rho(\Gamma)$ {\color{red}acts on $(\xi^{\alpha}(\Lambda(\Gamma)),m_{\alpha})$ with at most $r(\mathring{c}_{\alpha}^{opp})$ ergodic components.}
\end{theorem}

{\color{red}The rest of this Subsection is devoted to prove Theorem \ref{ergodictheoreminpaper}}. The key step is proving a restricted version of Shadow Lemma ({\color{red}see Lemma \ref{dichotomy} below}), whose proof is quite similar to the proof of Theorem \ref{thm12}. \changednew{We first prove a fundamental lemma about the tangent vectors of a positive curve.

Let $I\subset S^1$ be a closed arc, and let $\xi^{\Theta}:I\to \mathcal{F}_{\Theta}$ be a continuous positive map. For any $\alpha\in \Theta$, define $\xi^{\alpha}:I\to \mathcal{F}_{\alpha}$ by $\xi^{\alpha} = \pi^{\Theta}_{\alpha}\circ \xi^{\Theta}$, where $\pi^{\Theta}_{\alpha}: \mathcal{F}_{\Theta}\to \mathcal{F}_{\alpha}$ denotes the equivariant projection. Proposition \ref{prop3} shows that $\xi^{\alpha}(I)$ is a rectifiable curve with respect to $d_{\alpha}$; hence, we can define the Lebesgue measure on it, which we still denote by $m_{\alpha}$. Let $L = m_{\alpha}(\xi^{\alpha}(I))$, and let $f:[0,L] \to I$ be the boundary-preserving arc-length parametrization with respect to the \refchangenew{245}{pullback metric} $(\xi^{\alpha})^*d_{\alpha}$, i.e.,
\[
m_{\alpha}\big( \xi^{\alpha}( f([0,t]) ) \big) = t, \qquad \forall\, t \in [0,L].
\]
We refer to Stein-Shakarchi~\cite[Chapter 3: Section 4 and Theorem~4.3]{steinreal} for details regarding arc-length parametrization. \refchangenew{248}{In particular, the reference establishes that the map $t \mapsto \xi^{\alpha}(f(t))$ is differentiable for almost every $t \in [0,L]$ and that its derivative has unit norm almost everywhere.} Consequently, the map $t \mapsto \xi^{\alpha}(f(t))$ is $1$-Lipschitz.

And we assume there exists $g\in Aut_1(\mathfrak{g})$ and $x: I\to U_{\Theta}^{opp,>0}$ such that
\begin{enumerate}
    \item For any $t\in [0,L]$, $\xi^{\Theta}(f(t)) = gx(f(t)) \mathfrak{p}_{\Theta}$.
    \item If $s_2>s_1$, then $x(f(s_1))^{-1}x(f(s_2))\in U_{\Theta}^{opp,>0} $.
\end{enumerate}

\begin{proposition}\label{propositionoftangentvectors}
    With the notation above, the map $t\mapsto \pi_{\alpha}^{opp}(v^{opp}_{x(f(t))})$ is bi-Lipschitz with respect to the norm $\|\cdot\|$ on $\mathfrak{u}_{\alpha}^{opp}$. Furthermore, the derivative $\frac{d}{dt}\big(\pi_{\alpha}^{opp}(v^{opp}_{x(f(t))})\big)$ exists and lies in $c_{\alpha}^{opp}-\{0\}$ almost everywhere. Moreover, there exists a constant $C'>1$ such that
    \[
        \frac{1}{C'}\le \left\|\frac{d}{dt}\pi_{\alpha}^{opp}\big(v^{opp}_{x(f(t))}\big)\right\|\le C'
    \]
    for almost every $t$.
\end{proposition}

\begin{proof}

First we write
\begin{equation}\label{beginningequation}
x(f(t+\Delta t)) = x(f(t)) \cdot \exp \bigg(\log\bigg(x(f(t))^{-1}\cdot x(f(t+\Delta t))\bigg)\bigg). 
\end{equation}

Recall that the tangent space of $\mathcal{F}_{\Theta}$ (resp. $\mathcal{F}_{\alpha}$) at $\mathfrak{p}_{\Theta}$ (resp. $\mathfrak{p}_{\alpha}$) can be identified with the opposite unipotent subalgebra $\mathfrak{u}_{\Theta}^{opp}$ (resp. $\mathfrak{u}_{\{\alpha\}}^{opp}$), and the exponential map serves as a local diffeomorphism. If we let $\tilde{\pi}^{opp}_{\alpha}$ denote the linear projection from $\mathfrak{u}_{\Theta}^{opp}$ to $\mathfrak{u}_{\{\alpha\}}^{opp}$, then it serves as the differential of the natural projection $\mathcal{F}_{\Theta}\to \mathcal{F}_{\alpha}$, $g\mathfrak{p}_{\Theta}\mapsto g\mathfrak{p}_{\alpha}$, at the point $\mathfrak{p}_{\Theta}$. 

To be brief, we denote
\[
\tilde{v}_{\alpha}(t,\Delta t) = \tilde{\pi}^{opp}_{\alpha}\bigg(\log\bigg( x(f(t))^{-1}\cdot x(f(t+\Delta t))\bigg)\bigg)
\]
and
\[
\tilde{r}_{\alpha}(t,\Delta t) =\log\bigg( x(f(t))^{-1}\cdot x(f(t+\Delta t))\bigg)-\tilde{v}_{\alpha}(t,\Delta t).
\]
Then for fixed $t$, $\lim_{\Delta t\to 0} \tilde{v}_{\alpha}(t,\Delta t) = \lim_{\Delta t\to 0} \tilde{r}_{\alpha}(t,\Delta t) = 0$. Applying both sides of Equation~(\ref{beginningequation}) to the point $\mathfrak{p}_{\alpha} \in \mathcal{F}_{\alpha}$, we obtain
\[
x(f(t+\Delta t)) \mathfrak{p}_{\alpha} = x(f(t)) \cdot \exp \bigg(\tilde{v}_{\alpha}(t,\Delta t)+\tilde{r}_{\alpha}(t,\Delta t)\bigg)\mathfrak{p}_{\alpha}.
\]
Since $\mathfrak{g} = \mathfrak{u}_{\{\alpha\}}^{opp}\oplus \mathfrak{p}_{\alpha}$, it follows that $\tilde{r}_{\alpha}(t,\Delta t)\in \mathfrak{p}_{\alpha}$; thus, $\exp ( \tilde{r}_{\alpha}(t,\Delta t))\mathfrak{p}_{\alpha} = \mathfrak{p}_{\alpha}$. From the Baker–Campbell–Hausdorff formula, we continue to get the equality with an infinitesimal error, i.e.,
 $$
x(f(t)) \cdot \exp \bigg(\tilde{v}_{\alpha}(t,\Delta t)+\tilde{r}_{\alpha}(t,\Delta t)\bigg)\mathfrak{p}_{\alpha} = x(f(t)) \cdot \exp \bigg(\tilde{v}_{\alpha}(t,\Delta t)+o(\tilde{v}_{\alpha}(t,\Delta t))\bigg)\mathfrak{p}_{\alpha}.
$$

On the other hand, note that $x(f(t))^{-1}\cdot x(f(t+\Delta t))\in U_{\Theta}^{opp,>0}$. Applying \refchange{249}{\changed{Theorem \ref{thm10}~(2)~(a)}} and the formal group law (Lemma \ref{lem8}), if we define
\[
    v_{\alpha}(t,\Delta t) = \pi_{\alpha}^{opp}\big(v^{opp}_{(x(f(t)))^{-1}x(f(t+\Delta t))}\big) = \pi_{\alpha}^{opp}\big(v^{opp}_{x(f(t+\Delta t))}\big)- \pi_{\alpha}^{opp}\big(v^{opp}_{x(f(t))}\big),
\]
then we improve the above equality with infinitesimal error to:
\begin{equation}\label{equationforapproximation}
\begin{split}
    x(f(t+\Delta t)) \mathfrak{p}_{\alpha} 
    &= x(f(t))\exp \bigg(v_{\alpha}(t,\Delta t)+o(v_{\alpha}(t,\Delta t))\bigg) \mathfrak{p}_{\alpha} \\
    &= x(f(t))\exp \bigg(\pi_{\alpha}^{opp}\big(v^{opp}_{x(f(t+\Delta t))}\big) \\
    &\qquad\qquad - \pi_{\alpha}^{opp}\big(v^{opp}_{x(f(t))}\big)+o(v_{\alpha}(t,\Delta t))\bigg)\mathfrak{p}_{\alpha}.
\end{split}
\end{equation}
\refchange[2\baselineskip]{251}{\changed{
Note that as $I$ is compact, $x(f(t))$ is bounded with respect to $t$, so there exists a constant $C>1$ such that for any $t$ and $|\Delta t|\le \min(t,L-t)$,
\[
\begin{aligned}
\frac{1+o(1)}{C}\,
\big\|
  \pi_{\alpha}^{opp}\big(v^{opp}_{x(f(t+\Delta t))}\big)
  -\pi_{\alpha}^{opp}\big(v^{opp}_{x(f(t))}\big)
\big\| \\
\le d_{\alpha}\big(x(f(t+\Delta t))\mathfrak{p}_{\alpha},\,x(f(t))\mathfrak{p}_{\alpha}\big) \\
\le C(1+o(1))\,
\big\|
  \pi_{\alpha}^{opp}\big(v^{opp}_{x(f(t+\Delta t))}\big)
  -\pi_{\alpha}^{opp}\big(v^{opp}_{x(f(t))}\big)
\big\|.
\end{aligned}
\]
Here $o(1)$ means a function $o(1)(t,\Delta t)$ such that $\sup_{t}\lim_{\Delta t\to 0} o(1)(t,\Delta t) = 0$.

Recall that $\xi^{\alpha}(f(t)) = g x(f(t))\mathfrak{p}_{\alpha}$. Since the action of $g$ on $\mathcal{F}_{\alpha}$ is bi-Lipschitz, there exists a constant $C'>1$ such that for any $t$ and $|\Delta t|\le \min (t,L-t)$,
\begin{equation}\label{equationforlip}
\begin{aligned}[b]
\frac{1+o(1)}{C'}\,
\big\|
  \pi_{\alpha}^{opp}\big(v^{opp}_{x(f(t+\Delta t))}\big)
  -\pi_{\alpha}^{opp}\big(v^{opp}_{x(f(t))}\big)
\big\| \\
\le d_{\alpha}\big(\xi^{\alpha}(f(t+\Delta t)),\,\xi^{\alpha}(f(t))\big) \\
\le C'(1+o(1))\,
\big\|
  \pi_{\alpha}^{opp}\big(v^{opp}_{x(f(t+\Delta t))}\big)
  -\pi_{\alpha}^{opp}\big(v^{opp}_{x(f(t))}\big)
\big\|.
\end{aligned}
\end{equation}

As the map $t\mapsto \xi^{\alpha}(f(t))$ is $1$-Lipschitz, we have the map $t\mapsto \pi_{\alpha}^{opp}\big(v^{opp}_{x(f(t))}\big)$ is Lipschitz, so the tangent vector $\frac{d}{dt}\pi_{\alpha}^{opp}\big(v^{opp}_{x(f(t))}\big)$ is defined almost everywhere. Moreover, as the parameterization $t\mapsto \xi^{\alpha}(f(t))$ has unit speed, from Inequality~(\ref{equationforlip}) we have for almost every $t$, 
\[
    \frac{1}{C'}\le \|\frac{d}{dt}\pi_{\alpha}^{opp}\big(v^{opp}_{x(f(t))}\big)\|\le C'.
\]
so for almost every $t$, $\frac{d}{dt}\pi_{\alpha}^{opp}\big(v^{opp}_{x(f(t))}\big)$ varies in a compact subset of $c_{\alpha}^{opp}$ away from $0$.

Finally, note that for any $s_1<s_2\in [0,L]$,
\[
\pi_{\alpha}^{opp}\big(v^{opp}_{x(f(s_2))}\big) - \pi_{\alpha}^{opp}\big(v^{opp}_{x(f(s_1))}\big) = \int_{s_1}^{s_2}\frac{d}{dt}\pi_{\alpha}^{opp}\big(v^{opp}_{x(f(t))}\big)dt.
\]
Let $l\in(\mathfrak{u}_{\{\alpha\}}^{opp})^*$ be the functional in the proof of Lemma~\ref{lem15}, i.e., there exists a constant $C_2>1$ such that for any $v\in c_{\alpha}^{opp}$,
\[
\frac{1}{C_2} \|v\|\le l(v)\le C_2 \|v\|.
\]
Then we have
\[
\begin{aligned}
\big\| \pi_{\alpha}^{opp}\big(v^{opp}_{x(f(s_2))}\big) &- \pi_{\alpha}^{opp}\big(v^{opp}_{x(f(s_1))}\big)\big\| \\
&\ge \frac{1}{C_2}l\bigg(\pi_{\alpha}^{opp}\big(v^{opp}_{x(f(s_2))}\big) - \pi_{\alpha}^{opp}\big(v^{opp}_{x(f(s_1))}\big)\bigg) \\[0.5em]
&= \frac{1}{C_2} l\left(\int_{s_1}^{s_2}\frac{d}{dt}\pi_{\alpha}^{opp}\big(v^{opp}_{x(f(t))}\big)dt\right) \\[0.5em]
&= \frac{1}{C_2} \int_{s_1}^{s_2} l\bigg(\frac{d}{dt}\pi_{\alpha}^{opp}\big(v^{opp}_{x(f(t))}\big)\bigg)dt \\[0.5em]
&\ge \frac{1}{(C_2)^2} \int_{s_1}^{s_2} \| \frac{d}{dt}\pi_{\alpha}^{opp}\big(v^{opp}_{x(f(t))}\big)\|dt \\
&\ge \frac{1}{(C_2)^2 C'}(s_2-s_1).
\end{aligned}
\]
So the map $t\mapsto \pi_{\alpha}^{opp}\big(v^{opp}_{x(f(t))}\big)$ is indeed bi-Lipschitz.
}}  
\end{proof}

}

{\color{red} Now we introduce the concept of an \emph{$\alpha$-acute} subset. Let $\Gamma\subset \mathsf{PSL}(2,\mathbb{R})$ be a non-elementary discrete subgroup such that $\Lambda(\Gamma) = S^1$, $G$ be a simple Lie group with $\Theta$-positive structure, $\rho:\Gamma \to G$ be a $\Theta$-positive representation, and $\alpha\in \Theta$.   

\refchangenew{235}{\changednew{
We first set up some notations. For any clockwise quadruple $(a,b,c,d) \subset \Lambda(\Gamma)$, let $(a,d)_{arc}$ and $(b,c)_{arc}$ be the oriented open arcs in $S^1$ joining $a$ to $d$ and $b$ to $c$, respectively, chosen so that $(b,c)_{arc} \subset (a,d)_{arc}$. Let \(L = m_{\alpha}\big(\xi^{\alpha}\big((b,c)_{arc}\big)\big)\).
\begin{definition}
Using the above notations, we say the tuple $(f,g,x)$ is a \emph{chart} for the quadruple $(a,b,c,d)$ if:
\begin{enumerate}
    
    \item $f:[0,L] \to [b,c]_{arc}$ is the boundary preserving arc-length parametrization with respect to the \refchangenew{245}{pull-back metric} $(\xi^{\alpha})^*d_{\alpha}$, i.e.,
\[
f(0) = c,\quad  f(L) = b,
\]
and 
\[
m_{\alpha}\big( \xi^{\alpha}( f([0,t]) ) \big) = t, \qquad \forall\, t \in [0,L].
\]
    \item $g \in Aut_1(\mathfrak{g})$ and $x : [b,c]_{arc} \longrightarrow U_{\Theta}^{opp,>0}$ is a continuous map such that for any $t\in [0,L]$,
\[
\xi^{\Theta}(a, f(t), b) 
= g\big( \mathfrak{p}_{\Theta}^{opp},\, x(f(t)) \,\mathfrak{p}_{\Theta},\, \mathfrak{p}_{\Theta} \big),
\]
and whenever $p,q\in[b,c]_{arc}$ such that $(a,p,q,d)$ is positive, we have $x(q)^{-1}x(p)\in U_{\Theta}^{opp,>0}$.
\end{enumerate}
\end{definition}

Note that for any positive quadruple, a chart $(f,g,x)$ exists. We have talked about the existence of $f$ when setting up the notations of Proposition~\ref{propositionoftangentvectors}. For the existence of $g$ and $x$, see Corollary~\ref{posconv2}~(2). And note that $f$ is uniquely determined, $g$ is unique up to right multiplication by elements of $L^1_{\Theta} \cap \mathsf{Stab}\big(U_{\Theta}^{opp,>0}\big)$ whose adjoint action preserves each cone $\mathring{c}_{\alpha}^{opp},\alpha\in \Theta$, and $x$ is determined by $g$.}}

 \changednew{If $(a,b,c,d)$ is a clockwise quadruple and $(g,f,x)$ is a chart for it, then according to Proposition~\ref{propositionoftangentvectors} the derivative $\frac{d}{dt} \, \pi_{\alpha}^{opp}\!\big( v^{opp}_{x(f(t))} \big)$ is almost everywhere defined.} For any Borel-measurable subset $A\subset S^1$ we define
\[
\changednew{v_{\alpha}\big((a,b,c,d),(f,g,x),A\big)}
:= \int_{f^{-1}(A\cap [b,c]_{arc})} \frac{d}{dt} \, \pi_{\alpha}^{opp}\!\big( v^{opp}_{x(f(t))} \big) \, dt \in c_{\alpha}^{opp}.
\]
We say that $A$ has \emph{$\alpha$-acute support} in $(a,b,c,d)$ if \changednew{there exists a chart $(g,f,x)$ such that}
\[
\changednew{v_{\alpha}\big((a,b,c,d),(g,f,x),A\big) }\in \mathring{c}_{\alpha}^{opp}.
\]
Note that the choice of \changednew{$g$ (respectively, $x$)} is unique up to right multiplication \changednew{(respectively, conjugation)} by elements of $L^1_{\Theta} \cap \mathsf{Stab}\big(U_{\Theta}^{opp,>0}\big)$, whose adjoint action preserves each cone $\mathring{c}_{\alpha}^{opp}$. Thus the condition $\changednew{v_{\alpha}\big((a,b,c,d),(g,f,x),A\big) } \in \mathring{c}_{\alpha}^{opp}$ is independent of the choice of the chart. So the $\alpha$-acute support property is well defined \changednew{for quadruples}.

\begin{definition}\label{acutesubset}
Let $\rho:\Gamma\to G$ be a $\Theta$-positive representation with $\Lambda(\Gamma) = S^1$. For each $\alpha\in\Theta$, a measurable subset $A\subset \Lambda(\Gamma)$ is called \emph{$\alpha$-acute} if:
\begin{enumerate}
    \item $A$ is $\Gamma$-invariant;
    \item $m_{\alpha}(\xi^{\alpha}(A))>0$;
    \item For every clockwise quadruple $(a,b,c,d)$, $A$ has $\alpha$-acute support in $(a,b,c,d)$.
\end{enumerate}
\end{definition}
\refchangenew{235}{\changednew{Note that $\Lambda(\Gamma)$ itself serves as an $\alpha$-acute set.}} Indeed, in the notation above, for any clockwise-ordered quadruple $(a,b,c,d)$, \changednew{pick a chart $(g,f,x)$ for it,} we have
\begin{align*}
\changednew{v_{\alpha}\big((a,b,c,d),(g,f,x),A\big) } &= \int_{f^{-1}((b,c)_{arc})} \frac{d}{dt} \pi_{\alpha}^{opp}(v^{opp}_{x(f(t))}) dt \\
&= \int_{[0,L]} \frac{d}{dt} \pi_{\alpha}^{opp}(v^{opp}_{x(f(t))}) dt \\
&=\pi_{\alpha}^{opp}(v^{opp}_{x(f(L))}) - \pi_{\alpha}^{opp}(v^{opp}_{x(f(0))}) \\
&=\pi_{\alpha}^{opp}(v^{opp}_{x(f(0))^{-1}x(f(L))})\\
&=\pi_{\alpha}^{opp}(v^{opp}_{x(c)^{-1}x(b)}) 
\in \mathring{c}_{\alpha}^{opp}.
\end{align*}
The second last equality is by the formal group law (Lemma~\ref{lem8}).

It is useful to point out that being $\alpha$-acute is a local property. 

\begin{lemma}\label{locality}
Let $\Gamma$ be a lattice and fix $\alpha\in \Theta$. Suppose $A$ is $\Gamma$-invariant and $m_{\alpha}(\xi^{\alpha}(A))>0$. If there exists a clockwise quadruple $(a,b,c,d)$ such that $A$ has $\alpha$-acute support in $(a,b,c,d)$, then $A$ is $\alpha$-acute.
\end{lemma}

\begin{proof}
Let $\mathcal{OQ}$ be the space of clockwise oriented quadruples and $\mathcal{OT}$ the space of clockwise oriented triples of points on $S^1$. Let $\pi:\mathcal{OQ}\to \mathcal{OT}$ be the natural projection sending $(a,b,c,d)\mapsto (a,b,c)$. Since $A$ is $\Gamma$-invariant and $\Gamma$ preserves the orientation of $\Lambda(\Gamma)$, the set $\mathcal{A}$ consisting of clockwise quadruples on which $A$ has $\alpha$-acute support is also $\Gamma$-invariant.

We show several properties of $\mathcal{A}$. Suppose first that $A$ has $\alpha$-acute support in $(a,b,c,d)$. Then we claim that for any $\changednew{p}$ in the clockwise arc $(c,d)_{arc}$ joining $c$ and $d$, $A$ also has $\alpha$-acute support in $(a,b,c,\changednew{p})$. 
Let $[b,\changednew{p}]_{arc}$ be the clockwise closed arc joining $b$ and $\changednew{p}$, and \changednew{let $(g,f,x)$ be a chart for $(a,b,\changednew{p},d)$.} Then for any $e\in [b,\changednew{p}]_{arc}$,
\[
\xi^{\Theta}(a,e,d) = g(\mathfrak{p}_{\Theta}^{opp},\,x(e)\mathfrak{p}_{\Theta},\,\mathfrak{p}_{\Theta}).
\]
Let $(b,c)_{arc}\subset [b,\changednew{p}]_{arc}$ denote the clockwise open arc joining $b,c$. Then for any $\changednew{e'}\in (b,c)_{arc}$,
\[
\xi^{\Theta}(a, \changednew{e'},\changednew{p}) = g(\mathfrak{p}_{\Theta}^{opp},\,x(\changednew{e'})\mathfrak{p}_{\Theta},\,x(\changednew{p})\mathfrak{p}_{\Theta})
= g x(\changednew{p})(\mathfrak{p}_{\Theta}^{opp},\,x(\changednew{p})^{-1}x(\changednew{e'})\mathfrak{p}_{\Theta},\,\mathfrak{p}_{\Theta}).
\]
As $\changednew{p}$ is fixed and $(a,b,c,\changednew{p})$ is positive, the map $\changednew{e'}\mapsto x(\changednew{p})^{-1}x(\changednew{e'})$ on $(b,c)_{arc}$ is still a positive parameterization. \changednew{So $(g x(\changednew{p}),\big(f|_{f^{-1}([b,c]_{arc})}\big)_0,x(\changednew{p})^{-1}x(\cdot))$ is a chart for $(a,b,c,\changednew{p})$. Here if $\tilde{f}$ is a map from interval $[L_1,L_2]$, then $(\tilde{f})_0$ is the map on $[0,L_2-L_1]$ determined by $(\tilde{f})_0(t) = \tilde{f}(t+L_1)$.} Moreover, by the formal group law (Lemma \ref{lem8}),
\[
\pi_{\alpha}^{opp}(v_{x(\changednew{p})^{-1}x(e)}^{opp}) 
= \pi_{\alpha}^{opp}(v_{x(e)}^{opp}) - \pi_{\alpha}^{opp}(v_{x(\changednew{p})}^{opp})
\]
which implies that
\begin{equation}
\frac{d}{dt}\pi_{\alpha}^{opp}(v_{x(\changednew{p})^{-1}x(f(t))}^{opp}) 
= \frac{d}{dt} \pi_{\alpha}^{opp}(v_{x(f(t))}^{opp}).
\end{equation}
\changednew{Thus we have
\[
v_{\alpha}\bigg((a,b,c,\changednew{p}),\big(g x(\changednew{p}),\big(f|_{f^{-1}([b,c]_{arc})}\big)_0, x(\changednew{p})^{-1}x(\cdot)\big|_{[b,c]_{arc}}),A\bigg)  = v_{\alpha}\bigg((a,b,c,d),(g,\big(f|_{f^{-1}([b,c]_{arc})}\big)_0,x|_{[b,c]_{arc}}),A\bigg).
\]
Here we also observed that $(g,\big(f|_{f^{-1}([b,c]_{arc})}\big)_0,x|_{[b,c]_{arc}})$ is a chart for $(a,b,c,d)$.
}
Hence $A$ has $\alpha$-acute support in $(a,b,c,d)$ implies that $A$ also has $\alpha$-acute support in $(a,b,c,\changednew{p})$.

Similarly, if $A$ has $\alpha$-acute support in $(a,b,c,d)$, then for any $\changednew{p}$ in the open arc joining $a,d$ which does not contain $b,c$, $A$ also has $\alpha$-acute support in $(a,b,c,\changednew{p})$. As a result, having $\alpha$-acute support in $(a,b,c,d)$ is a property independent of $d$, so $\mathcal{A} = \mathcal{OQ}$ if and only if $\pi(\mathcal{A}) = \mathcal{OT}$.

Moreover, $\pi(\mathcal{A})$ satisfies a monotonicity property. Let $(a,b)_{arc}$ be the clockwise arc joining $a$ and $b$. From the definition of $\alpha$-acute support it follows that, if $A$ has $\alpha$-acute support in $(a,b,c,d)$, then for any $y\in (a,b)_{arc}$, $A$ has $\alpha$-acute support in $(a,y,c,d)$. Thus if $(a,b,c)\in \pi(\mathcal{A})$, then $(a,y,c)\in \pi(\mathcal{A})$ for all $y\in (a,b)_{arc}$.

Next, we show that $\mathcal{A}$ is open, and hence $\pi(\mathcal{A})$ is also open. \changednew{We note the following continuous dependencies of charts:}
\begin{enumerate}
    \item $L$ varies continuously with respect to $(a,b,c,d)$.
    
    \item $g$ can be chosen to vary continuously with respect to $(a,b,c,d)$ (see Lemma~\ref{posconvergence}).
    
    \item Note that the map $t\mapsto x(f(Lt)) \mathfrak{p}_{\alpha}$ can be written as $t\mapsto g^{-1}\xi^{\alpha}(f(Lt))$. It follows that their tangent maps are equal: $\frac{d}{dt}(x(f(Lt)) \mathfrak{p}_{\alpha}) = \frac{d}{dt} g^{-1}\xi^{\alpha}(f(Lt))$, which means that for any fixed $\changednew{p}\in (b,c)_{arc}$, if $f(Lt_0(L,g)) = \changednew{p}$, then the value of $\frac{d}{dt}(x(f(Lt)) \mathfrak{p}_{\alpha})$ at $t_0(L,g)$ depends continuously only on $g$ and $L$. \changednew{We also recall that in Equation~(\ref{equationforapproximation}) we established that if we let \[
    v_{\alpha}(t,\Delta t) = \pi_{\alpha}^{opp}\big(v^{opp}_{(x(f(t)))^{-1}x(f(t+\Delta t))}\big) = \pi_{\alpha}^{opp}\big(v^{opp}_{x(f(t+\Delta t))}\big)- \pi_{\alpha}^{opp}\big(v^{opp}_{x(f(t))}\big),
    \]
    then
    $$
    x(f(t+\Delta t)) \mathfrak{p}_{\alpha}=x(f(t))\exp \bigg(v_{\alpha}(t,\Delta t)+o(v_{\alpha}(t,\Delta t))\bigg) \mathfrak{p}_{\alpha}.
    $$}
    This implies that if $\changednew{p}\in (b,c)_{arc}$, $f(Lt_0) = \changednew{p}$, and $\xi^{\alpha}(\Lambda(\Gamma))$ is differentiable at $\xi^{\alpha}(\changednew{p})$, then the derivative $\frac{d}{dt}\big(\pi^{opp}_{\alpha}(v^{opp}_{x(f(Lt))})\big)$ also exists at $t_0$, and
\begin{equation}\label{calculationoftangentvectors}
\begin{split}
\frac{d}{dt}\big(\pi^{opp}_{\alpha}(v^{opp}_{x(f(Lt))})\big)(t_0) &= D\changednew{\mathcal{L}}_{x(f(Lt_0))^{-1}} \left( \frac{d}{dt}(x(f(Lt)) \mathfrak{p}_{\alpha}) (t_0) \right) \\
&= \changednew{D\changednew{\mathcal{L}}_{x(f(Lt_0))^{-1}}}\cdot D\changednew{\mathcal{L}}_{g^{-1}} \left( \frac{d}{dt}(\xi^{\alpha}(f(Lt)))(t_0) \right).
\end{split}
\end{equation}
Here, for any $g\in G$, \refchangenew{247}{$D\changednew{\mathcal{L}}_g$} denotes the differential of the left multiplication $\changednew{\mathcal{L}}_g$ acting on $\mathcal{F}_{\alpha}$.
    Thus, under the continuous choice of $g$, it is direct to verify that the map $t\mapsto \pi^{opp}_{\alpha}(v^{opp}_{x(f(Lt))})$ varies continuously with respect to $(a,b,c,d)$ in the space $Lip_{BV}([0,1],\changednew{\mathfrak{u}}_{\alpha}^{opp})$. Here, $Lip_{BV}([0,1],\changednew{\mathfrak{u}}_{\alpha}^{opp})$ consists of all Lipschitz maps $F$ from $[0,1]$ to $\changednew{\mathfrak{u}}_{\alpha}^{opp}$, with topology given by the norm defined by:
    \[
    \| F\|_{BV} = \sup_{0\le t\le 1} \|F(t)\| +\int_{0}^{1} \left\| \frac{d}{dt} F(s)\right\| ds.
    \]
\end{enumerate}

Therefore, for a fixed $A$, \changednew{$v_{\alpha}((a,b,c,d),(g,f,x),A)$} varies continuously with respect to $(a,b,c,d)$, as long as \changednew{the charts $(g,f,x)$ are also chosen to vary continuously as above.} So, if $A$ has $\alpha$-acute support in $(a,b,c,d)$, then there exists a neighborhood $U$ of $(a,b,c,d)$ in $\mathcal{OQ}$ such that for any $(x,y,z,t)\in U$, $A$ has $\alpha$-acute support in $(x,y,z,t)$, so $\mathcal{A}, \pi(\mathcal{A})$ are both open.

Finally we conclude the proof. There is a $\Gamma$-equivariant isomorphism $\mathcal{V}$ between $\mathcal{OT}$ and the unit tangent bundle $T^1\mathbb{D}$, sending $(a,b,c)$ to the unit tangent vector $\mathcal{V}(a,b,c)$ whose forward point is $a$, backward point is $c$, and such that the geodesic from $b$ to the basepoint of $\mathcal{V}(a,b,c)$ is perpendicular to the geodesic from $a$ to $c$. Let $\phi_{t}$ denote the geodesic flow on $T^1\mathbb{D}$. Then $\mathcal{V}\circ\pi(\mathcal{A})$ is a $\Gamma$-invariant and $\phi_{\mathbb{R}^{>0}}$-invariant (which comes from the monotonicity of $\pi(\mathcal{A})$) open set. Since $\Gamma$ is a lattice, the geodesic flow on $T^1\mathbb{D}/\Gamma$ is ergodic with respect to the Bowen–Margulis measure (which is smooth) and conservative. Hence almost every point has a dense forward orbit (for example, see Blayac~\cite[Fact 2.30]{blayac2021pattersonsullivandensitiesconvexprojective}). Therefore $\mathcal{V}\circ\pi(\mathcal{A})$ is open and dense, which must equal to $T^1\mathbb{D}$, so $\pi(\mathcal{A}) = \mathcal{OT}$, so $\mathcal{A} = \mathcal{OQ}$, concluding the proof.
\end{proof}
}

\begin{lemma}\label{dichotomy}
Let $\Gamma\subset \mathsf{PSL}(2,\mathbb{R})$ be a non-elementary discrete subgroup such that $\Lambda(\Gamma)=S^1$, $G$ be a simple Lie group with $\Theta$-positive strucutre, and $\rho:\Gamma\to G$ be a $\Theta$-positive representation.  {\color{red}Let $\alpha\in \Theta$, $A$ be an $\alpha$-acute subset and fix $b_0\in \mathbb{D}$. Then for any $r>0,R\ge 0$, there exists a constant $C>1$ such that if $d_{\mathbb{D}}(\gamma(b_0),z)\le R$, then
$$
\frac{1}{C} e^{-\alpha(\kappa(\rho(\gamma)))}\le m_{\alpha}(\mathcal{O}_{r}^{\alpha}(b_0,z)\cap \xi^{\alpha}(A))\le C e^{-\alpha(\kappa(\rho(\gamma)))}.
$$   }
\end{lemma}

\begin{proof}

     We prove the $\frac{1}{C} e^{-\alpha(\kappa(\rho(\gamma)))}\le m_{\alpha}(\mathcal{O}_{r}^{\alpha}(b_0,z)\cap \xi^{\alpha}(A))$ part first (the other part is similar). If not, then there exists an $\alpha$-acute $A$ and $r>0, R\ge 0$ with the following property: there exists $z_n\in \mathbb{D},\gamma_n\in \Gamma$ such that $d_{\mathbb{D}}(\gamma_n(b_0),z_n)\le R$ and 
    $$
    m_{\alpha}(\mathcal{O}_{r}^{\alpha}(b_0,z_n)\cap \xi^{\alpha}(A))\le \frac{1}{n}e^{-\alpha(\kappa(\rho(\gamma_n)))} .
    $$

   \refchange{238}{\changed{By the definition of $\alpha$-acuteness, we have $m_{\alpha}(\xi^{\alpha}(A))>0$ and $A$ is $\Gamma$-invariant. Note that $m_{\alpha}$ has full support and no atoms, and $\Gamma$ acts topologically transitively on $\Lambda(\Gamma)$. So for any non-empty open set $U \subset \Lambda(\Gamma)$, we have $m_{\alpha}(\xi^{\alpha}(U\cap A))>0$.}} Using similar arguments as we did in the proof of Theorem \ref{thm12}, we claim that $\gamma_n$ is unbounded. So we may assume that $d_{\mathbb{D}}(b_0,z_n)>r$.

   By taking a subsequence, we assume that $\omega_{\gamma_n}\to a$. Define $z'_n = \gamma_n^{-1}(z_n)$, $b'_n = \gamma_n^{-1}(b_0)$, $\mathcal{O}'_n = \gamma_n^{-1}\mathcal{O}_r(b_0,z_n) = \mathcal{O}_r(b'_n,z'_n)$, and $p_n,q_n$ be two end points of $\mathcal{O}'_n$. By passing to a further subsequence, we can assume $z'_n\to z'\in \mathbb{D}$ where $d_{\mathbb{D}}(z',b_0)\le R$, $b'_n\to a$, $p_n\to p, q_n\to q$ and $(a,p,q)$ is a {\color{red}clockwise} triple. Pick $b\in \Lambda(\Gamma)$ such that $(a,p,q,b)$ is {\color{red}clockwise} positive, and we choose an arc $[a,b]_{arc}$ joining $a$ and $b$ such that for $n\gg1$, $\mathcal{O}'_n\subset (a,b)_{arc}$.  Let $E\subset (a,b)_{arc}$ be a closed arc such that for $n\gg1$, $E\subset \mathcal{O}'_n$.

     \changednew{
We first construct a sequence of charts. Let $(g_n)$ be a sequence in $Aut_1(\mathfrak{g})$ and let $(x_n)$ be a sequence of maps $x_n:E\to U_{\Theta}^{opp,>0}$, satisfying the following properties from \refchange{201}{\changed{Corollary \ref{posconv2}~(2)}}:
    \begin{itemize}
        \item $g_n$ converges to $g\in Aut_1(\mathfrak{g})$.
        \item For any $n\in \mathbb{Z}^{>0}$ and any $p\in E$,
        \[
        \xi^{\Theta}(\changed{\omega_{\gamma_n}},p,b) =g_n(\mathfrak{p}_{\Theta}^{opp},x_n(p)\mathfrak{p}_{\Theta},\mathfrak{p}_{\Theta}). 
        \]
        \item $x_n$ converges uniformly to a continuous map $x_{\infty}:E\to U_{\Theta}^{opp,>0}$. And for any $n\in \mathbb{Z}^{>0}\cup \{\infty\}$, and any $p,q\in E$ such that $(a,p,q,b)$ is positive, \refchange{202}{\changed{we have}}
        \[
        x_n(q)^{-1}x_n(p)\in U_{\Theta}^{opp,>0}.
        \]
    \end{itemize}
    Let $L = m_{\alpha}(\xi^{\alpha}(E))$ and choose an parameterization $f:[0,L]\to E$ such that 
\begin{itemize}
    \item For any $t\in [0,L]$, $m_{\alpha}(\xi^{\alpha}(f([0,t]))) = t$.
    \item If $0\le {\color{red} s<t}\le L$, then $(a,f(t),f(s),b)$ is positive.
\end{itemize}
    Then $(g_n,f,x_n)$ is a converging sequence of charts for $(\omega_{\gamma_n},f(L), f(0), b)$, which converges to the chart $(g,f,x_{\infty})$ for $(a,f(L), f(0), b)$ (for details of the convergence, see the enumeration of discussions in the proof of Lemma~\ref{locality}).
}  

\changednew{
Let {\color{red}$E' = E\cap A$ and let $F = f^{-1} (E')$}. \refchangenew{251}{From Proposition~\ref{propositionoftangentvectors} we have for all the $n\in \mathbb{Z}^{>0}\cup \{\infty\}$ the sequence of maps $t\mapsto \pi_{\alpha}^{opp}(v^{opp}_{x_{n}(f(t))})$ is bi-Lipschitz}, and the sequence of derivatives $\frac{d}{dt}(\pi_{\alpha}^{opp}(v^{opp}_{x_{n}(f(t))}))$ is defined almost everywhere.} Let
{\color{red}
$$
v= \int_{F} \frac{d}{dt}(\pi_{\alpha}^{opp}(v^{opp}_{\changednew{x_{\infty}}( f(t))}))dt.
$$
As $A$ is $\alpha$-acute, $v\in \mathring{c}_{\alpha}^{opp}$.} \refchange[-\baselineskip]{247}{\changednew{On the other hand, the third argument in the enumeration appeared in the proof of Lemma~\ref{locality} shows that
\[
\int_{F} \frac{d}{dt}\big(\pi_{\alpha}^{opp}(v^{opp}_{\changednew{x_{n}( f(t))}})\big)\, dt \;\to\; v.
\]
}}
    
\changednew{  
    Let $\rho(\gamma_n) = \mu_n l_n\nu_n$ be a Cartan decomposition of $\rho(\gamma_n)$, and up to a subsequence, we assume that $\mu_n,\nu_n$ converges. From Lemma~\ref{l2}, there exist $u_n^{-}\in U_{\Theta}^{opp}$, $l'_n\in L_{\Theta}^{1}$, and continuous maps $x'_n: E\to U_{\Theta}^{opp,>0}$ such that $u_n^{-},l'_n$ converges, $x'_n$ \refchangenew{244}{uniformly converges to $x_{\infty}:E\to U_{\Theta}^{opp,>0}$}, and the following holds
    \begin{enumerate}
        \item  For any $p\in E$
        \[
                \rho(\gamma_n)\xi^{\Theta}(p) = \mu_n l_n u_n^{-} l_n'x'_n(p)\mathfrak{p}_{\Theta},
        \]
        or equivalently,
        \begin{equation}\label{equationsimplify}
                \xi^{\Theta}(p) = \nu_n^{-1}u_n^{-} l_n'x'_n(p)\mathfrak{p}_{\Theta}.
        \end{equation}

        \item  Moreover, if  $p,q\in E$ and $(a,p,q,b)$ is positive, then for any $n\in \mathbb{Z}^{>0}$, \refchange{199}{\changed{we have}} $x'_n(q)^{-1}x'_n(p)\in U_{\Theta}^{opp,>0}$.
    \end{enumerate}
    Again from Proposition~\ref{propositionoftangentvectors} we have for all the $n\in \mathbb{Z}^{>0}$ the sequence of maps $t\mapsto \pi_{\alpha}^{opp}(v^{opp}_{x'_{n}(f(t))})$ is bi-Lipschitz, and the sequence of derivatives $\frac{d}{dt}(\pi_{\alpha}^{opp}(v^{opp}_{x'_{n}(f(t))}))$ is defined almost everywhere. Note that $\nu_n^{-1}u_n^{-}l'_n$ converges and $x'_n$ converges to $x_{\infty}$ uniformly. \changednew{Applying the third argument of the enumeration appeared in the proof of Lemma~\ref{locality} to Equation~(\ref{equationsimplify}), we obtain
\[
\int_{F} \frac{d}{dt}\big(\pi_{\alpha}^{opp}(v^{opp}_{\changednew{x'_{n}( f(t))}})\big)\, dt \;\to\; v.
\]  
    }}

Now we have
$$
 m_{\alpha}(\mathcal{O}_{r}^{\alpha}(b_0,z_n)\cap \xi^{\alpha}(A))=m_{\alpha}(\rho(\gamma_n) \xi^{\alpha}(\mathcal{O}'_n\cap A))
$$
$$
\ge m_{\alpha}(\rho(\gamma_n) \xi^{\alpha}(E\cap A))= m_{\alpha}(\rho(\gamma_n) \xi^{\alpha}({\color{red}E'}))
$$
$$
=\int_{{\color{red}F}} \| \frac{d}{dt} (\rho(\gamma_n)\xi^{\alpha}(f(t))) \| dt
=\int_{{\color{red}F}} \| \frac{d}{dt} (\mu_nl_nu_n^{-}l'_nx'_n(f(t))\mathfrak{p}_{\alpha}) \| dt
$$
where the norms are the Riemannian norms on the tangent spaces, and $\frac{d}{dt}$ means taking the tangent vector of the curves $t\mapsto \rho(\gamma_n)\xi^{\alpha}(f(t))$ and $t\mapsto \mu_nl_nu_n^{-}l'_nx'_n(f(t))\mathfrak{p}_{\alpha}$ in $\mathcal{F}_{\alpha}$, which are defined almost everywhere.

Note that $l_n,l'_n$ both stabilizes $\mathfrak{p}_{\alpha}$, so we have
$$
\mu_nl_nu_n^{-}l'_nx'_n(f(t))\mathfrak{p}_{\alpha} = \mu_n (l_n u_n^{-} l_n^{-1}) l_n(l'_n\refchange{254}{\changed{x'_n(f(t))}}{l'_n}^{-1}) l'_n  \mathfrak{p}_{\alpha}
$$
$$
=\mu_n (l_n u_n^{-} l_n^{-1}) l_n(l'_n\changed{x'_n(f(t))}{l'_n}^{-1})\mathfrak{p}_{\alpha} = \mu_n (l_n u_n^{-} l_n^{-1}) l_n(l'_n\changed{x'_n(f(t))}{l'_n}^{-1})l_n^{-1}\mathfrak{p}_{\alpha}
$$ 
and note that both sequences $\mu_n$ and $l'_n$ converge (by previous construction), and the sequence $(l_n u_n^{-} l_n^{-1})$ also converges (as $u_n^-\in U_{\Theta}^{opp}$). \refchange{255}{\changed{Recall that $l'_n\in L_{\Theta}^1$, so similar to the proof of Theorem \ref{thm12} (recall that the argument was as follows: when $l'_n \in \psi(L_{\Theta}^{\circ})$, we let the sequence 
$x''_n(f(t)) := l'_n x'_n(f(t)) {l'_n}^{-1}$ be the new sequence of converging maps; 
and when $l'_n \in L_{\Theta}^1 - \psi(L_{\Theta}^{\circ})$, the map is instead conjugated to another identification of the positive semigroup)}}, we can remove $\mu_n,l'_n,u_n^{-}$ and the estimate continues as
$$
\approx \int_{{\color{red}F}} \| \frac{d}{dt} (\mathsf{c}_{l_n}(x''_n(f(t)))\mathfrak{p_{\alpha}}) \| dt.
$$
\changednew{
Here we assume $x''_n$ converges to a map $x''_{\infty}$ conjugate to $x_{\infty}$ and we define
\[
v''= \int_{F} \frac{d}{dt}(\pi_{\alpha}^{opp}(v^{opp}_{\changednew{x''_{\infty}}( f(t))}))dt\in \mathring{c}_{\alpha}^{opp}.
\]
We also have
\[
\frac{d}{dt}(\pi_{\alpha}^{opp}(v^{opp}_{\changednew{x''_{n}}( f(t))}))dt\to v''
\]
as $n\to \infty$.
}

\changednew{From Equation~(\ref{calculationoftangentvectors}) we can continue the estimate of $\int_{{\color{red}F}} \| \frac{d}{dt} (\mathsf{c}_{l_n}(x''_n(f(t)))\mathfrak{p_{\alpha}}) \| dt$ as}
$$
\approx \int_{{\color{red}F}} \| \frac{d}{dt} (\pi_{\alpha}^{opp}(v^{opp}_{\mathsf{c}_{l_n}(x''_n(f(t)))})) \| dt
$$
$$
\ge \| \int_{{\color{red}F}}  \frac{d}{dt} (\pi_{\alpha}^{opp}(v^{opp}_{\mathsf{c}_{l_n}(x''_n(f(t)))})) dt\| = \| Ad_{l_n}\int_{{\color{red}F}}  \frac{d}{dt} (\pi_{\alpha}^{opp}(v^{opp}_{x''_n(f(t))})) dt\| .
$$
\refchange{257}{\changed{(Here we used two equalities: $v_{\mathsf{c}_l (x)}^{opp} = Ad_l v_x^{opp}$ and $\pi_{\alpha} (Ad_l v_x^{opp}) = Ad_l \pi_{\alpha}(v_x^{opp})$ which can be verified directly.)}}

Since $\int_{{\color{red}F}}  \frac{d}{dt} (\pi_{\alpha}^{opp}(v^{opp}_{x''_n(f(t))})) dt\to \changednew{v''}\in \mathring{c}_{\alpha}^{opp}$, \refchange{253}{\changed{the sequence 
$\int_{ {\color{red} F} } \frac{d}{dt} \big(\pi_{\alpha}^{opp}(v^{opp}_{x''_n(f(t))})\big)\, dt$
is contained in a properly bounded subset of $\mathring{c}_{\alpha}^{opp}$. 
We then apply Lemma~\ref{lem7} to the cone $\mathring{c}_{\alpha}^{opp}$ and the maps $Ad_{l_n}$ (as we did in the proof of Theorem~\ref{thm12}),}}
so we get a constant $C'>1$ such that \[\frac{1}{C'} e^{-\alpha(\kappa(\rho(\gamma_n)))}\le \|Ad_{l_n}(\int_{{\color{red}F}}  \frac{d}{dt} (\pi_{\alpha}^{opp}(v^{opp}_{x''_n(f(t))})) dt)\|\le  C' e^{-\alpha(\kappa(\rho(\gamma_n)))}.\] The above integral estimate shows that there exists a constant $C>0$ such that $m_{\alpha}(\mathcal{O}_{r}^{\alpha}(b_0,z_n)\cap \xi^{\alpha}(A))\ge C e^{-\alpha(\kappa(\rho(\gamma_n)))}$ holds for all $n$, which is a contradiction. So we have finished the proof.
\end{proof}

{{\color{red}
Now we show Lemma \ref{dichotomy} implies that $\alpha-$acute sets are of full measure, that is:
\begin{proposition}\label{acutemeansfull}
Let $\rho:\Gamma\to G$ be a $\Theta$-positive representation from a lattice and $\alpha\in\Theta$.  
If $A\subset \Lambda(\Gamma)$ is $\alpha$-acute, then $\xi^{\alpha}(A)$ is of full $m_{\alpha}$-measure. 
\end{proposition}
}

The proof of {\color{red}Proposition \ref{acutemeansfull}} is inspired by Canary-Zhang-Zimmer~\cite[Theorem 13.1]{CaZhZi3}, whose proof uses a shadow version Vitali covering lemma in Roblin~\cite{shadowvitali}.
\begin{proof}
{\color{red}Suppose the proposition is not true, then there exists an $\alpha$-acute $A\subset \Lambda(\Gamma)$ such that $m_\alpha(\xi^{\alpha}(A^c))>0$.} Define the measure $m_{\alpha}'$ on $\xi^{\alpha}(\Lambda(\Gamma))$ by $m_{\alpha}'(E) = m_{\alpha}(E\cap \xi^{\alpha}(A))$. 

Note that when $\Gamma$ is a lattice, $\Lambda(\Gamma)-\Lambda_c(\Gamma)$ is the set of parabolic fixed points which is countable. Also note that $m_{\alpha}$ has no atoms, so $\xi^{\alpha}(\Lambda_c(\Gamma))$ is of $m_{\alpha}$ full measure and $\xi^{\alpha}(\Lambda_c(\Gamma)\cap A)$ is of $m'_{\alpha}-$full measure. Let $B = \xi^{\alpha}(A^c\cap \Lambda_c(\Gamma))$. Clearly $m'_\alpha(B) = 0$ and $m_{\alpha}(B) = m_{\alpha}(\xi^{\alpha}(A^c))>0$. Let \( S_K \) denote the compact complement of the finite-area hyperbolic surface \refchange{258}{\changed{$\mathbb{D}/\Gamma$}} obtained by removing pairwise disjoint cusp neighborhoods. Then any geodesic that does not converge to a cusp (i.e., whose lifts to the universal cover do not go to a parabolic fixed point) returns to \( S_K \) infinitely often. Let $R$ be the diameter of $S_K$ \refchange{259}{\changed{and let $b_0\in \mathbb{D}$ be chosen whose projection lies in $S_K$}}, then $\cup_{\gamma\in\ \Gamma}\mathcal{O}_{R}(b_0,\gamma(b_0))$ covers every point in  $\Lambda_c(\Gamma)$ for infinitely many times.  Lemma \ref{dichotomy} shows that both $m'_{\alpha}$ and $m_{\alpha}$ satisfy the shadow lemma, so there exists a constant $C>1$ such that
$$
\frac{1}{C} m'_\alpha(\mathcal{O}_R^{\alpha}(b_0,\gamma(b_0)))\le m_\alpha(\mathcal{O}_R^{\alpha}(b_0,\gamma(b_0)))\le Cm'_\alpha(\mathcal{O}_R^{\alpha}(b_0,\gamma(b_0))) 
$$
$$
 m_\alpha(\mathcal{O}_{5R}^{\alpha}(b_0,\gamma(b_0)))\le Cm_\alpha(\mathcal{O}_{R}^{\alpha}(b_0,\gamma(b_0))),\quad  m'_\alpha(\mathcal{O}_{5R}^{\alpha}(b_0,\gamma(b_0)))\le Cm'_\alpha(\mathcal{O}_{R}^{\alpha}(b_0,\gamma(b_0))).
$$

Note that $m_\alpha$ and $m'_{\alpha}$ are both outer regular measures. For any $\epsilon>0$, as $m'_{\alpha}(B)=0$, we can pick an open set $U$ containing $B$ such that $m'_{\alpha}(U)<\epsilon$. Let $\mathcal{I}=\{\gamma\mid \refchange{260}{\changed{\gamma\in \Gamma}},  \mathcal{O}^{\alpha}_{R}(b_0,\gamma(b_0))\subset U\}$. For any diverging sequence $\gamma_n\in \Gamma$, the diameter of $\mathcal{O}^{\alpha}_R(b_0,\gamma_n(b_0))$ converges to $0$. This implies that the \refchange{261}{\changed{interior of shadows $\{\mathring{\mathcal{O}}^{\alpha}_{R}(b_0,\gamma(b_0))\}_{\gamma \in \mathcal{I}}$}} is an open covering of $B$. \refchange{262}{Roblin~\changed{\cite[page. 23]{shadowvitali}}} claims that there exists $\mathcal{J}\subset \mathcal{I}$ such that $\{\mathcal{O}_{R}(b_0,\gamma(b_0))|\gamma\in \mathcal{J}\}$ is a disjoint family and 
\[
\cup_{\gamma\in \mathcal{I}}\mathcal{O}_{R}(b_0,\gamma(b_0))\subset \cup_{\gamma\in \mathcal{J}}\mathcal{O}_{5R}(b_0,\gamma(b_0)),
\]
which is equivalent to say $\{\mathcal{O}^{\alpha}_{R}(b_0,\gamma(b_0))|\gamma\in \mathcal{J}\}$ is a disjoint family and 
\[
\cup_{\gamma\in \mathcal{I}}\mathcal{O}^{\alpha}_{R}(b_0,\gamma(b_0))\subset \cup_{\gamma\in \mathcal{J}}\mathcal{O}^{\alpha}_{5R}(b_0,\gamma(b_0)).
\]
As a result,
$$
m_{\alpha}(B)\le m_{\alpha}(\cup_{\gamma\in \mathcal{I}}\mathcal{O}^{\alpha}_{R}(b_0,\gamma(b_0)))\le m_{\alpha}(\cup_{\gamma\in \mathcal{J}}\mathcal{O}^{\alpha}_{5R}(b_0,\gamma(b_0)))\le \sum_{\gamma\in \mathcal{J}} m_{\alpha}(\mathcal{O}^{\alpha}_{5R}(b_0,\gamma(b_0))).
$$
Since
$$
m_\alpha(\mathcal{O}_{5R}^{\alpha}(b_0,\gamma(b_0)))\le Cm_\alpha(\mathcal{O}_{R}^{\alpha}(b_0,\gamma(b_0)))
$$
and
$$
\frac{1}{C} m'_\alpha(\mathcal{O}_R^{\alpha}(b_0,\gamma(b_0)))\le m_\alpha(\mathcal{O}_R^{\alpha}(b_0,\gamma(b_0)))\le Cm'_\alpha(\mathcal{O}_R^{\alpha}(b_0,\gamma(b_0))), 
$$
we continue the estimate as
$$
m_{\alpha}(B)\le C\sum_{\gamma\in \mathcal{J}} m_{\alpha}(\refchange{263}{\changed{\mathcal{O}^{\alpha}_{R}}}(b_0,\gamma(b_0)))\le C^2 \sum_{\gamma\in \mathcal{J}} m'_{\alpha}(\changed{\mathcal{O}^{\alpha}_{R}}(b_0,\gamma(b_0)))
$$
$$
=C^2 m'_{\alpha}(\cup_{\gamma \in \mathcal{J}} \changed{\mathcal{O}^{\alpha}_{R}}(b_0,\gamma(b_0)))
$$
where the equality is from the disjointness of \refchange{264}{\changed{the collection of shadows $\mathcal{O}^{\alpha}_{R}(b_0,\gamma(b_0)),\gamma \in \mathcal{J}$.}} As $\cup_{\gamma\in \mathcal{J}}\changed{\mathcal{O}^{\alpha}_{R}}(b_0,\gamma(b_0))\subset U$, we continue the estimate as
$$
m_{\alpha}(B) \le C^2 m'_{\alpha}(U)\le C^2\epsilon.
$$
As $\epsilon$ is arbitrary, $m_{\alpha}(\xi^{\alpha}(A^c)) = m_{\alpha}(B)=0 $, a contradiction. So we have finished the proof.
\end{proof}

{\color{red}

\begin{proof}[Proof of Theorem \ref{ergodictheoreminpaper}]

Suppose not, then let $A_0, A_1, \dots, A_{r(\mathring{c}_{\alpha}^{opp})}$ be a partition of $\Lambda(\Gamma)$ such that each $A_i$ is $\Gamma$-invariant and satisfies $m_{\alpha}(\xi^{\alpha}(A_i)) > 0$. 

Pick a clockwise quadruple $(a,b,c,d)$ \changednew{and a chart $(g,f,x)$ for it}, we have
\[
\sum_{i=0}^{r(\mathring{c}_{\alpha}^{opp})} v_{\alpha}\big((a,b,c,d),\changednew{(g,f,x)},A_i\big) 
= v_{\alpha}\big((a,b,c,d), \changednew{(g,f,x)}, \Lambda(\Gamma)\big) \in \mathring{c}_{\alpha}^{opp}.
\]
For each $0 \le i \le r(\mathring{c}_{\alpha}^{opp})$, since $A_i$ is $\Gamma-$invariant and $m_{\alpha}(\xi^{\alpha}(A_i)) > 0$, we have 
\[
v_{\alpha}\big((a,b,c,d), \changednew{(g,f,x)}, A_i\big) \in c_{\alpha}^{opp} - \{0\}.
\]
By the definition of $r(\mathring{c}_{\alpha}^{opp})$, there exists some index $j$ such that
\[
\sum_{\substack{0 \le i \le r(\mathring{c}_{\alpha}^{opp}) \\ i \neq j}} 
v_{\alpha}\big((a,b,c,d),\changednew{(g,f,x)},A_i\big) \in \mathring{c}_{\alpha}^{opp}.
\]
But this sum is exactly $v_{\alpha}\big((a,b,c,d),\changednew{(g,f,x)}, A_j^c\big)$, which means that $A_j^{c}$ has $\alpha$-acute support for the chosen quadruple.  
By Lemma~\ref{locality}, $A_j^{c}$ is $\alpha$-acute, so by Proposition \ref{acutemeansfull}, $m_{\alpha}(A_j) = 0$, a contradiction.
\end{proof}
}

\section{\changed{Critical exponent upper bound and rigidity for lattices}}\label{sec5}

In this section we prove Theorem \ref{thmA}~(1) and (2), following the strategy in Pozzetti-Sambarino-Wienhard~\cite{Liplim}.

Assume \refchange{274}{\changed{$\Gamma\subset \mathsf{PSL}(2,\mathbb{R})$ is a non-elementary discrete subgroup}} first. Fix $\alpha\in \Theta$ and a Riemannian metric $d_{\alpha}$ on $\mathcal{F}_{\alpha}$. As $\Gamma$ is discrete, we can also pick a base point $b_0\in \mathbb{D}$ such that $\mathsf{Stab}_{\Gamma}(b_0) = \{id\}$. So for any $z=\gamma(b_0)$,  $c_{\alpha}(z) := e^{-\alpha(\kappa(\rho(\gamma)))}$ is well defined.

Proposition \ref{prop3} shows that $\rho$ is $\alpha-$transverse, and Proposition \ref{cartanproperty} shows that for each $n\in \mathbb{Z}^{>0}$, $\mathcal{A}_n = \{z\in \Gamma(b_0)| e^{-(n+1)} < c_{\alpha}(z)\le e^{-n}\}$ is finite. We invoke a separation lemma of shadows. \refchange{275}{\changed{Since the proof can be reproduced directly from \cite[Lemma~8.5]{CaZhZi}, which treats the case $G=\mathsf{PGL}(d,\mathbb{R})$, we postpone it to Appendix~\ref{appendB}, where a slightly generalized version of this lemma (see Lemma~\ref{lem11refined}) will also be used frequently.}}

\begin{lemma}{\refchange{275}{\changed{\cite[Lemma 8.5]{CaZhZi}}}}\label{lem11}
For any $r>0$, there exists $C_0 = C_0(r)>0$ such that if $n\ge 0$, $z,w\in\mathcal{A}_n$ and $d_{\mathbb{D}}(z,w)>C_0$, then
$$
\mathcal{O}_r(b_0,z)\cap \mathcal{O}_r(b_0,w) = \emptyset.
$$
\end{lemma}

\refchange{276}{\changed{Next we state a property of the critical exponent, whose proof will be postponed until the proof of a more refined version (see Lemma~\ref{lem13refined}).}}

\begin{lemma}{\refchange{276}{\changed{\cite[Lemma 8.7]{CaZhZi}}}}\label{lem13}    
    For any $0<\delta< \delta^{\alpha}_{\rho}(\Gamma)$, we have
    $$
\limsup_{n\to \infty} \sum_{z\in \mathcal{A}_n} c_{\alpha}(z)^{\delta} = \infty.
    $$
\end{lemma}

\changed{We establish an upper bound for the critical exponents, which proves Theorem~\ref{thmA}~(1).}
\begin{proposition}\label{whenS1}
If $\Gamma\subset \mathsf{PSL}(2,\mathbb{R})$ is a non-elementary discrete subgroup, $G$ is a simple Lie group with $\Theta$-positive structure and $\rho:\Gamma\to G$ is a $\Theta$-positive representation, then for any $\alpha\in \Theta$, $\delta^{\alpha}_{\rho}(\Gamma) \le 1$.
\end{proposition}
\begin{proof}
    Suppose $\delta^{\alpha}_{\rho}(\Gamma)>1$, then from Lemma \ref{lem13}, we have
    $$
    \limsup _{n\to \infty} \sum_{z\in \mathcal{A}_n} c_{\alpha}(z) = \infty.
    $$

\changed{Fix \(b_{0}\in \mathbb{D}\). Let $r$ be a sufficiently large radius (see the discussions preceding Proposition~\ref{prop4}) with respect to $b_0$ and $R= 0$.}
 As $\Gamma$ is discrete, we can make a finite partition $\Gamma = \sqcup_{i = 1}^{L} \Gamma_i$ such that $\Gamma_i(b_0)$ are all $C_0$-separated where $C_0$ is the constant in Lemma \ref{lem11}. Then there exists an $i\in \{1,2,...\refchange{277}{\changed{,}}L\}$, such that
    $$
 \limsup_{n\to \infty} \sum_{z\in \mathcal{A}_n\cap \Gamma_i(b_0)} c_{\alpha}(z) = \infty,
    $$
    and for any $z,w\in \mathcal{A}_n\cap \Gamma_i(b_0)$, $\mathcal{O}_r(b_0,z)\cap \mathcal{O}_r(b_0,w) = \emptyset$. \changed{Recall that $x_{z,r},y_{z,r}\in \Lambda(\Gamma)$ are the coarse end points of $\mathcal{O}_r(b_0,z)$.} Proposition \ref{prop4} implies that there exists a constant $C_1>0$ such that for \changed{any $\gamma\in \Gamma$ where $d_{\mathbb{D}}(b_0,\gamma(b_0))>r$, we have
    $$
    c_{\alpha}(\gamma(b_0))\le  C_1d_{\alpha}(\xi^{\alpha}(x_{\gamma(b_0),r}),\xi^{\alpha}(y_{\gamma(b_0),r})).
    $$}

    \refchange{279}{\changed{As the shadows are disjoint, order the elements of $\mathcal{A}_n \cap \Gamma_i(b_0)$ as 
$z_{1,n}, \dots, z_{N_n,n}$, where $N_n = |\mathcal{A}_n \cap \Gamma_i(b_0)|$, so that 
\[
(x_{z_{1,n},r}, y_{z_{1,n},r}, x_{z_{2,n},r}, y_{z_{2,n},r}, \dots, x_{z_{N_n,n},r}, y_{z_{N_n,n},r})
\]
is positive.}} Since the limit map $\xi^{\alpha}$ is rectifiable \refchange{278}{\changed{(see Proposition~\ref{prop3})}}, by the definition of rectifiability there exists a constant $\tilde{C}>0$ such that for each $n \in \mathbb{Z}^+$,
\[
\sum_{z \in \mathcal{A}_n \cap \Gamma_i(b_0)} c_{\alpha}(z)
= \sum_{k = 1}^{N_n} c_{\alpha}(z_{k,n})
\le C_1 \sum_{k = 1}^{N_n} d_{\alpha}\big(\xi^{\alpha}(x_{z_{k,n},r}),\xi^{\alpha}(y_{z_{k,n},r})\big)
\le \tilde{C},
\]
which is a contradiction. Therefore, $\delta^{\alpha}_{\rho}(\Gamma) \le 1$.

\end{proof}

Canary-Zhang-Zimmer~\refchange{280}{\changed{\cite[Theorem 1.2]{CaZhZi}}} proved a lower-bound for the critical exponent, which is a generalization of Pozzetti-Sambarino-Wienhard's result for Anosov subgroups in \refchange{280}{\changed{\cite[Theorem 3.1]{Liplim}}}.

\begin{theorem}{\cite[Theorem 1.2 and \refchange{282}{\changed{Corollary 5.2}}]{CaZhZi}}\label{thmB}

If $\Gamma\subset \mathsf{PSL}(2,\mathbb{R})$ is a discrete subgroup and $\rho:\Gamma\to G$ is an $\alpha-$transverse representation, then 
$$
\delta^{\alpha}_{\rho}(\Gamma)\ge\dim \xi^{\alpha}(\Lambda_c(\Gamma))
$$
where $\dim$ denotes the Hausdorff dimension with respect to any Riemannian metric $d_\alpha$ on $\mathcal{F}_{\alpha}$.
\end{theorem}

\begin{remark}\label{dimup}
  Although the original statement is for transverse subgroups in $\mathsf{PGL}(d,\mathbb{R})$, \refchange{281}{\changed{by applying Proposition~\ref{linearize} to linearize transverse representations, we can use the adapted representation to apply the cited theorem to a general semisimple group $G$. }}
\end{remark}

\begin{proof}[Proof of Theorem \ref{thmA} (2)]
 Recall again that $\xi^{\alpha}$ is rectifiable \refchange{283}{\changed{(see Proposition \ref{prop3})}}, so \refchange{284}{\changed{the image $\xi^{\alpha}(\Lambda(\Gamma))=\xi^{\alpha}(S^1)$ is a rectifiable curve whose Hausdorff dimension is $1$ (see Stein-Shakarchi~\cite[Chapter~3: Section 4 and Chapter~7: Theorem 2.4]{steinreal})}}. As $\Gamma$ is a lattice, $\Lambda(\Gamma)-\Lambda_{c}(\Gamma)$ is the set of parabolic fixed points which is countable, so we have $\dim(\xi^{\alpha}(\Lambda_{c}(\Gamma))) = 1$. By Proposition \ref{thm17}, $\rho$ is $\alpha-$transverse (indeed relatively $\alpha-$Anosov), so Theorem \ref{thmB} implies that $\delta^{\alpha}_{\rho}(\Gamma)\ge1$. Combining with  \refchange{285}{\changed{Proposition \ref{whenS1}}}, we have $\delta^{\alpha}_{\rho}(\Gamma) = 1$.
\end{proof}

\section{Doubling $\Theta$-positive representations}\label{sec6}

In this Section we prove \changed{Theorem \ref{thmA}~(3) and we will assume that $\Gamma$ is geometrically finite.} The strategy is as follows. If \(\Lambda(\Gamma)\neq S^{1}\), then
\refchange{286}{\changed{the convex core of \(\mathbb{D}/\Gamma\)}} is a hyperbolic
\changed{surface with geodesic boundary. Doubling it along this boundary produces a
finite–area hyperbolic surface (see for example, Alhfors~\cite[Page.26, 13H]{ahlforsrs}) — which corresponds to a lattice
\(\Gamma^{D}\supset \Gamma\),} and we compare the critical exponent of the doubled representation and the original representation.

We assume that $\Gamma \subset \mathsf{PSL}(2,\mathbb{R})$ is a non-elementary discrete subgroup. \refchange{287}{\changed{And we assume that $G$ is a real simple Lie group with finitely many components, whose identity component $G^{\circ}$ has finite center. Moreover, if $\mathfrak{g}$ denotes the Lie algebra of $G$ with adjoint representation $\psi:G \to Aut(\mathfrak{g})$, then $\psi(G) \subset Aut_1(\mathfrak{g})$ (see Section~\ref{assumption}). We further assume that $G$ has $\Theta$-positive sturcuture, and note that $Aut_1(\mathfrak{g})$ has $\Theta$-positive structure compatible with the one on $G$.}} And if $\rho:\Gamma\to G$ is a $\Theta$-positive representation, then $\psi\circ \rho$ is also $\Theta$-positive. Moreover, as the Lie algebras of $G$ and $Aut_1(\mathfrak{g})$ are isomorphic, one can choose Cartan projection $\kappa'$ on $Aut_1(\mathfrak{g})$ such that for any $g\in G,\alpha \in \Delta$, we have $\alpha(\kappa(g)) = \alpha (\kappa'(\psi(g)))$.

\refchange{288}{\changed{Labourie–McShane~\cite[Section~9.2]{labourie2009cross} first constructed the double of Hitchin representations for convex cocompact Fuchsian groups, and Canary–Zhang–Zimmer~\cite[Appendix~A]{CaZhZi} extended this construction to Hitchin representations arising from geometrically finite Fuchsian groups. We adapt their proofs to the $\Theta$-positive setting.}}

\begin{proposition}\label{prop6}
    Let $\Gamma\subset \mathsf{PSL}(2,\mathbb{R})$ \refchange{313}{\changed{be a non-elementary and geometrically finite}} subgroup such that $\Lambda(\Gamma)\not = S^1$, \refchange{293}{\changed{G be a simple Lie group with $\Theta$-positive structure}} and $\rho:\Gamma\to G$ be a $\Theta$-positive representation, then there exists a \changed{lattice} $\Gamma^{D}\subset \mathsf{PSL}(2,\mathbb{R})$ \changed{containing $\Gamma$ as an infinite index subgroup}, and a $\Theta$-positive representation $\rho^{D}:\Gamma^D\to Aut_1(\mathfrak{g})$ such that $\rho^{D}|_{\Gamma} =\psi \circ\rho $. Furthermore, we have the following properties:
    
    \begin{enumerate}

        \item If $\xi^{\Theta}$ is the limit map for $\rho$ and $\xi^{\Theta,D}$ is the limit map for $\rho^{D}$, then $\xi^{\Theta, D}|_{\Lambda(\Gamma)} = \xi^{\Theta}$.

        \item  If $G$ is connected and $Z(G)$ is trivial, then one can lift $\rho^{D}$ to $\tilde{\rho}^{D}:\Gamma^{D}\to G$ so that $\tilde{\rho}^{D}|_{\Gamma} = \rho$.

    \end{enumerate} 
    
\refchangenew{304}{\changednew{Before the proof, we establish a Lie theoretic lemma: 
\begin{lemma}\label{stabofaut}
Let $\mathfrak{g}$ be a semisimple Lie algebra and $\Theta$ be a subset of its positive simple roots. If $\phi\in Aut_1(\mathfrak{g})$ preserves both $\mathfrak{p}_{\Theta}$ and $\mathfrak{p}_{\Theta}^{opp}$, then for any $\alpha\in \Theta$, $\phi$ preserves $\mathfrak{u}_{\alpha}$ and $\mathfrak{u}_{\alpha}^{opp}$.
\end{lemma}
\begin{proof}
    Recall that $\mathfrak{u}_{\Theta}$ (respectively, $\mathfrak{u}_{\Theta}^{opp}$) is the unique maximal nilpotent ideal of $\mathfrak{p}_{\Theta}$ (respectively, $\mathfrak{p}_{\Theta}^{opp}$). Since $\phi$ is a Lie algebra automorphism, it must preserve these unique ideals, whence:
     \[
    \phi(\mathfrak{u}_{\Theta}) = \mathfrak{u}_{\Theta} \text{ and } \phi(\mathfrak{u}_{\Theta}^{opp}) = \mathfrak{u}_{\Theta}^{opp}.
    \]
    Furthermore, the Levi subalgebra $\mathfrak{l}_{\Theta}$ is characterized by the intersection:
    \[
    \mathfrak{l}_{\Theta} = \mathfrak{p}_{\Theta}\cap \mathfrak{p}_{\Theta}^{opp},
    \]
    which is consequently preserved by $\phi$. Since $\phi$ preserves $\mathfrak{l}_{\Theta}$, it necessarily preserves its center $\mathfrak{z}_{\Theta}$. It follows that $\phi$ preserves the weight space decomposition of $\mathfrak{u}_{\Theta}$ associated with $\mathfrak{z}_{\Theta}$. In particular, this implies that $\phi$ preserves each $\mathfrak{u}_{\alpha}$ and $\mathfrak{u}_{\alpha}^{opp}$ for $\alpha \in \Theta$  as described in Theorem~\ref{thm6}.
\end{proof}
}}

\end{proposition}

\begin{proof}
\textbf{Step 1: Notation setup.}

Let $\mathcal{C}(\Gamma)$ be the convex hull of $\Lambda(\Gamma)$ and $\mathcal{S}(\Gamma)$ be the set of boundary components of $\mathcal{C}(\Gamma)$. Every $b\in\mathcal{S}(\Gamma)$ is the axis of some hyperbolic element $\beta_{b}\in \Gamma\subset \mathsf{PSL}(2,\mathbb{R})$, and let $r_{b}\in \mathsf{PGL}(2,\mathbb{R})$ be the reflection with respect to $b$ (\refchange{289}{\changed{Note that the action of $r_b$ on $\mathbb{D}$ is not orientation preserving, so $r_b\not \in \mathsf{PSL}(2,\mathbb{R}) $).}} \refchange[\baselineskip]{290}{\changed{Define $\hat{\Gamma} = \langle \Gamma,\{r_b\}_{b\in\mathcal{S}(\Gamma)}\rangle\subset \mathsf{PGL}(2,\mathbb{R})$ and let $\Gamma^D = \hat{\Gamma}\cap \mathsf{PSL}(2,\mathbb{R})$ be the subgroup whose action on $\mathbb{D}$ is orientation preserving.}} \refchangenew{290}{\changednew{It can be shown that \(\mathbb{D}/\Gamma^{D}\) is the finite-area surface obtained by gluing the convex core of \(\mathbb{D}/\Gamma\) to itself along its geodesic boundary. Specifically, applying the arguments in \cite[Proposition 9.2.2 and Section 9.2.3]{labourie2009cross} (which, although stated for convex cocompact surfaces, apply equally to geometrically finite surfaces) to $\mathbb{D}/\Gamma$ with $J = \begin{pmatrix}1&0\\0&-1 \end{pmatrix} \in \mathsf{PGL}(2,\mathbb{R})$ yields the $J$-extension (see \cite[Definition 9.2.1]{labourie2009cross}) of $\Gamma$, which coincides with $\Gamma^{D}$ in our context.}} Hence \(\Gamma^{D}\) is a lattice and \(\Lambda(\Gamma^{D})=S^{1}\), and $\Gamma\subset \Gamma^D$ is an infinite index subgroup (as $\mathbb{D}/\Gamma^{D}$ is of finite area but $\mathbb{D}/\Gamma$ is not).

Let $P$ be a \changed{finite-sided} Dirichlet polygon of the $\Gamma$ action on $\mathbb{D}$ \changed{(which is a closed subset of $\mathbb{D}$)}, chosen so that $P$ intersects each $\Gamma$-orbit in $\mathcal{S}(\Gamma)$ exactly once, and $\mathcal{E}(\Gamma)$ be the collection of geodesics in $\mathcal{S}(\Gamma)$ which intersects $P$. The side identification of $P$ gives a finite presentation of $\Gamma = \langle X:R\rangle$. \refchange{292}{\changed{Then
\[
\big\langle X,\ \{r_b\}_{b\in\mathcal{E}(\Gamma)} \ \big| \ R,\ \{r_b^{2}\}_{b\in\mathcal{E}(\Gamma)}\big\rangle
\]
is also a finite presentation of $\hat{\Gamma}$. Indeed, $X\cup\{r_b\}_{b\in\mathcal{E}(\Gamma)}$ generates $\hat{\Gamma}$, since every reflection in $\hat{\Gamma}$ is $\Gamma$–conjugate to some $r_b$ with $b\in\mathcal{E}(\Gamma)$; moreover, no additional relations are needed beyond those in $R$ together with $r_b^2=1$ for $b\in\mathcal{E}(\Gamma)$.}}

\vspace{2mm}

\textbf{Step 2: Construction of the involution $x_I$.}

\refchange{294}{\changed{We use a construction that appeared in Guichard-Wienhard~\cite[Proposition 5.2]{GW2}. Theorem \ref{thm6}~(2), (3) shows that the following properties uniquely determine an element $x_I \in Aut_1(\mathfrak{g})$: $x_I|_{\mathfrak{l}_{\Theta}} = id$; for each $\alpha \in \Theta$, $x_I|_{\mathfrak{u}_{\alpha}} = -id$, and $x_I|_{\mathfrak{u}_{\alpha}^{opp}} = -id$. Clearly, $x_I$ also satisfies the additional properties:
\begin{itemize}
  \item $x_I^2 = id \in Aut_1(\mathfrak{g})$;
  \refchange{295}{\changed{\item $x_I \psi(U_{\Theta}^{>0})x_I^{-1} = \psi(U_{\Theta}^{>0})^{-1}$}}.
  \item \refchangenew{304}{\changednew{$x_I$ lies in the center of $\mathrm{Stab}_{Aut_1(\mathfrak{g})}(\mathfrak{p}_{\Theta}) \cap \mathrm{Stab}_{Aut_1(\mathfrak{g})}(\mathfrak{p}_{\Theta}^{opp})$;}}
\end{itemize}}}
The last item is due to Lemma~\ref{stabofaut}.

\vspace{2mm}

\textbf{Step 3: Construction of the double representation.}

Let $G^{\circ}$ be the identity component of $G$. For each $b\in \mathcal{E}(\Gamma)$, pick $\refchange{297}{\changed{g_b}}\in G^{\circ}$ such that $\xi^{\Theta}(\beta_{b}^{-},\beta_{b}^{+}) = \changed{g_b}(\mathfrak{p}_{\Theta}^{opp}, \mathfrak{p}_{\Theta})$, and set $R_b = \psi(\changed{g_b}) x_{I} \psi(\changed{g_b^{-1}})\in Aut_1(\mathfrak{g})$. One can see that \refchange{296}{\changed{$R_b^2 = id$}}, so we can extend the representation $\rho:\Gamma \to G$ to $\hat{\rho}: \hat{\Gamma}\to Aut_1(\mathfrak{g})$ by requiring $\hat{\rho}|_{\Gamma} = \rho$ and $\hat{\rho}(r_b) = R_b$. And denote the restriction as $\rho^{D}:\Gamma^{D}\to Aut_1(\mathfrak{g})$.

An element $\gamma$ in $\hat{\Gamma}$ is orientation preserving if and only if its expression involves even number of elements in $\{r_b\}_{b\in\mathcal{S}(\Gamma)}$. In this case, $\hat{\rho}(\gamma)$ involves even number of $x_{I}$. And note that if $g\in G^{\circ}$, then $x_I \psi(g) x_I\in Inn(\mathfrak{g})$ (as if one deform $g$ to $id$, then $x_I \psi(g) x_I$ is also deformed to $id$). So if $\psi(G)=Inn(\mathfrak{g})$, then we have $\rho^D(\Gamma^D)\subset Inn(\mathfrak{g})$. \refchangenew{303}{\changednew{Specially, if $G$ is connected and $Z(G)$ is trivial, then $G\cong Inn(\mathfrak{g})$, so $\rho^D$ can be lifted to $\tilde{\rho}:\Gamma^{D}\to G$. We proceed to show that $\rho^{D}$ is $\Theta$-positive, which verifies the second item in the ``furthermore'' part of the proposition.}}

\vspace{2mm}

\textbf{Step 4: Extension of $\xi^{\Theta}$.}

Let $\Gamma_{R} = \langle\{r_b\}_{b\in \mathcal{S}(\Gamma)}\rangle\subset \mathsf{PGL(2,\mathbb{R}})$. Note that the boundary components \refchange{298}{\changed{of $\mathcal{C}(\Gamma)$}} are all disjoint, a Ping-Pong argument shows every element in $\Gamma_{R}$ has a unique reduced word expression in $r_b$:  \refchange{299}{\changed{If $x \in \Gamma_{R}$ has two different reduced expressions, then by removing the common initial words, we may assume that the first words of the two expressions are different, say one is $r_{b_1}$ and the other is $r_{b_2}$. Choose $b_0 \in \mathcal{C}(\Gamma)$. The fact that $x$ has a reduced expression starting with $r_{b_1}$ implies that $x(b_0)$ lies in the component of $\mathbb{D}-\mathcal{C}(\Gamma)$ bounded by $b_1$; on the other hand, the reduced expression starting with $r_{b_2}$ implies that $x(b_0)$ lies in the component bounded by $b_2$. These two components are disjoint whenever $b_1 \neq b_2$, giving a contradiction.}} As a result, we have the following claim (see also Canary-Zhang-Zimmer~\cite[Lemma A.2]{CaZhZi}):
\begin{claim}
For any $r_1\not = r_2\in \Gamma_{R}$, one of the following holds 
\begin{enumerate}
    \item There exists $b\in\mathcal{S}(\Gamma)$, such that $r_1 = r_2 r_b$, in which case $r_1(\mathcal{C}(\Gamma))\cap r_2(\mathcal{C}(\Gamma)) = r_1(b) = r_2(b)$.

    \item $r_1(\mathcal{C}(\Gamma))\cap r_2(\mathcal{C}(\Gamma)) = \emptyset$.
\end{enumerate}  
As a result, for any $r_1,r_2\in \Gamma_{R}$, $r_1(\Lambda(\Gamma))\cap r_2(\Lambda(\Gamma)) $ is non-empty if and only if $r_1 = r_2$ or $r_1 = r_2 r_b$ for some $b\in \mathcal{S}(\Gamma)$. In the latter case, the intersection is $\{\beta_b^{-},\beta_b^{+}\}$.
\end{claim}

\refchange{300}{\changed{Recall that $\hat{\Gamma}$ is generated by $\Gamma$ and $\{r_b\}_{b\in \mathcal{E}(\Gamma)}$ and $\Lambda(\hat\Gamma) = \Lambda(\Gamma^D) = S^1$}}, so one can then \refchangenew{303}{\changednew{extend the limit map}} to $\xi^{\Theta, D}: \changednew{\Gamma_R(\Lambda(\Gamma))}\to \mathcal{F}_{\Theta}$ \refchange{300}{\changed{in the $\hat{\rho}(\hat{\Gamma})$-equivariant way}} by requiring \refchangenew{303}{\changednew{$\xi^{\Theta, D}|_{\Lambda(\Gamma)} = \xi^{\Theta}$}} and  $\xi^{\Theta, D}(r_b x) = R_b \xi^{\Theta,D}(x)$. Obviously $\xi^{\Theta, D}$ is also $\rho^D(\Gamma^D)$-equivariant. \refchangenew{304}{\changednew{Next we show that it is positive on $\Gamma_{R}\big(\Lambda(\Gamma)\big)$.}}

\vspace{2mm}

\textbf{Step 5: Verification of the positivity of $\xi^{\Theta,D}$ on $\Gamma_R(\Lambda(\Gamma))$.}

\changed{Recall that $r_b$ flips $\mathbb{D}$ with respect to the axis with endpoints $\beta_b^{+}, \beta_b^{-} \in \Lambda(\Gamma)$.} Note that $\beta_{b}^{+}, \beta_{b}^{-}$ divide $\partial \mathbb{D}$ into \refchange{301}{\changed{two closed arcs, say $I_{b,1}$ and $I_{b,2}$, with $r_b(I_{b,1}) = I_{b,2}$ and $r_b(I_{b,2}) = I_{b,1}$, and 
\[
R_b\big(\refchangenew{301}{\changednew{g_b}}(\mathfrak{p}_{\Theta}^{opp}, U_{\Theta}^{>0}\mathfrak{p}_{\Theta}^{opp}, \mathfrak{p}_{\Theta})\big) 
= \refchangenew{301}{\changednew{g_b}}(\mathfrak{p}_{\Theta}^{opp}, (U^{>0}_{\Theta})^{-1}\mathfrak{p}_{\Theta}^{opp}, \mathfrak{p}_{\Theta}).
\]
}} 

Fix $r_b$ and we first check that $\xi^{\Theta,D}|_{\Lambda(\Gamma)\cup r_b(\Lambda(\Gamma))}$ is positive. As $\xi^{\Theta,D}$ is equivariant and $b\in \Gamma\cdot \mathcal{E}(\Gamma)$, up to a $\Gamma$ conjugation we can assume that $b\in \mathcal{E}(\Gamma)$. \refchange{303,304,305}{{\color{blue}{It suffices to show that for any positive tuple in \refchangenew{302}{\changednew{$\Lambda(\Gamma)\cup r_b(\Lambda(\Gamma))$}}, its $\xi^{\Theta,D}$-image is positive in $\mathcal{F}_{\Theta}$. Since subtuples of positive tuples remain positive, it is enough to check that the $\xi^{\Theta,D}$-images of positive tuples containing \refchange{302}{$\beta_b^-$ and $\beta_b^+$} are positive. 

Since $\Lambda(\Gamma)$ is contained in one of $I_{b,1}$ and $I_{b,2}$, we can assume the tuple can be written in cyclic order as
\[
(\beta_b^-,x_1,\ldots,x_n,\beta_b^+,y_1,\ldots,y_m),
\]
where $x_i \in \Lambda(\Gamma)$ for $i=1,\ldots,n$ and $y_j = r_b y'_j$ with $y'_j \in \Lambda(\Gamma)$ for $j=1,\ldots,m$.  
\refchangenew{304}{\changednew{
From Corollary~\ref{posconv2}~(2), there exist $\tilde{g}_b\in Aut_1(\mathfrak{g})$, $p_i, p'_j \in U_{\Theta}^{opp,>0}$ for $i=1,\ldots,n$ and $j=1,\ldots,m$ such that
\[
\xi^{\Theta}(\beta_b^-,\beta_b^+) = \tilde{g}_b(\mathfrak{p}_{\Theta}^{opp},\mathfrak{p}_{\Theta})
\]
\[
\xi^{\Theta}(x_i) = \tilde{g}_b p_i \mathfrak{p}_{\Theta}, \quad 
\xi^{\Theta}(y'_j) = \tilde{g}_b p'_j \mathfrak{p}_{\Theta},
\]
with $p_{i+1}^{-1}p_i \in U_{\Theta}^{opp,>0}$ and $(p'_j)^{-1}p'_{j+1} \in U_{\Theta}^{opp,>0}$ for $i=1,\ldots,n-1$ and $j=1,\ldots,m-1$.  
}}
Recall that $g_b \in G^{\circ}$ is chosen so that 
\[
\xi^{\Theta}(\beta_b^-,\beta_b^+) = g_b(\mathfrak{p}_{\Theta}^{opp},\mathfrak{p}_{\Theta}),
\]
so
\changednew{
\[
\psi(g_b^{-1})\tilde{g}_b \in \mathrm{Stab}_{Aut_1(\mathfrak{g})}(\mathfrak{p}_{\Theta}) \cap \mathrm{Stab}_{Aut_1(\mathfrak{g})}(\mathfrak{p}_{\Theta}^{opp}).
\]
}

\changednew{\refchangenew{304}{Write
\[
\xi^{\Theta}(y_j) = R_b\xi^{\Theta}(y_j') = \psi(g_b) x_I \big(\psi (g_b^{-1}) \tilde{g}_b\big)p'_j \mathfrak{p}_{\Theta} = \tilde{g}_b x_I  p'_j \mathfrak{p}_{\Theta} .\] The last equality is due to the third property of $x_I$ in Step 2. Since $x_I \psi(U_{\Theta}^{opp,>0}) x_I^{-1} = \psi(U_{\Theta}^{opp,>0})^{-1}$, if we let $q_j =  \psi^{-1}(x_I \psi(p'_j) x_{I}^{-1}) \in (U_{\Theta}^{opp,>0})^{-1}$ for $j=1,\ldots,m$, then}
\[
\xi^{\Theta}(y_j) = \tilde{g}_b q_j \mathfrak{p}_{\Theta}, \quad 
q_j^{-1}q_{j+1} \in (U_{\Theta}^{opp,>0})^{-1}.
\]
}
(recall that we abbreviate $\psi(g)$ as $g$ when acting $G$ on flag manifolds, see Section \ref{assumption}.) Consequently,
\[
\xi^{\Theta,D}(\beta_b^-,x_1,\ldots,x_n,\beta_b^+,y_1,\ldots,y_m) 
= \changednew{\tilde{g}_b}(\mathfrak{p}_{\Theta}^{opp},p_1 \mathfrak{p}_{\Theta},\ldots,p_n \mathfrak{p}_{\Theta},\mathfrak{p}_{\Theta},q_1 \mathfrak{p}_{\Theta},\ldots,q_m \mathfrak{p}_{\Theta}),
\]
\[
= \changednew{\tilde{g}_b} q_m(\mathfrak{p}_{\Theta}^{opp},q_m^{-1}p_1 \mathfrak{p}_{\Theta},\ldots,q_m^{-1}p_n \mathfrak{p}_{\Theta},q_m^{-1}\mathfrak{p}_{\Theta},q_m^{-1}q_1 \mathfrak{p}_{\Theta},\ldots,\mathfrak{p}_{\Theta}).
\] And the properties of the sequences $p_i$ and $q_j$ established above show that the tuple is positive. Hence, $\xi^{\Theta,D}$ is positive on $\Lambda(\Gamma)\cup r_b(\Lambda(\Gamma))$.
}}}

\vspace{3mm}

Next we show that $\xi^{\Theta, D}$ is positive on \refchange{306,307}{\changed{$\Lambda(\Gamma)\cup \big(\bigcup_{r\in \mathcal{S}'} r(\Lambda(\Gamma))\big)$}} for any finite subset \refchange{307}{\changed{$\mathcal{S}'\subset \Gamma_R -\{id\}$}} such that \refchange{306,307}{\changed{$\mathcal{C}(\Gamma) \cup \big(\bigcup_{r\in \mathcal{S'}} r(\mathcal{C}(\Gamma))\big)$}} is connected. We prove this by induction on the cardinality of $\mathcal{S}'$.

When $\# \mathcal{S}' = 1$, it has been completed above, \refchange{306,307}{\changed{as $\mathcal{C}(\Gamma)\cup r \mathcal{C}(\Gamma), r\in\Gamma_R$ is connected if and only if $r = r_b$ for some $b\in \mathcal{S}(\Gamma)$.}}

\refchangenew{304}{\changednew{
For the inductive step, let $\mathcal{S}''\subset \mathcal{S}'$ be a subset such that $\#\mathcal{S}'' = \#\mathcal{S}'-1$ and \changed{$\mathcal{C}(\Gamma) \cup \big(\bigcup_{r\in \mathcal{S''}} r(\mathcal{C}(\Gamma))\big)$} is connected. \changed{Such an $\mathcal{S}''$ exists. Indeed, the Claim in Step 3 shows that for any distinct $r_1, r_2$, the sets $r_1(\mathcal{C}(\Gamma))$ and $r_2(\mathcal{C}(\Gamma))$ overlap along at most one boundary component. Thus, we can choose an element $r_1\in \mathcal{S}'$ such that $r_1(\mathcal{C}(\Gamma))$ intersects $\mathcal{C}(\Gamma) \cup \big(\bigcup_{r\in \mathcal{S'}- \{r_1\}} r(\mathcal{C}(\Gamma))\big)$ at exactly one boundary component, and let $\mathcal{S}'-\mathcal{S}''=\changed{\{r_1\}}$.} Moreover, using the same claim, we can find $r_2\in \mathcal{S}''$ such that $\changed{r_{1}}(\mathcal{C}(\Gamma))$ and $\changed{r_{2}}(\mathcal{C}(\Gamma))$ \refchange{309}{\changed{share}} a common boundary component $d$. Moreover, if $d = r_2(b')$ and $ b'\in \mathcal{S}(\Gamma)$, then $r_1 = r_2 r_{b'}$.

Let
\[
\Lambda'' = \Lambda(\Gamma)\cup \bigcup_{r\in (\mathcal{S}''\cup\{id\}) - \{r_2\}} r_2^{-1} r\Lambda(\Gamma).
\]
By the choice of $r_1,r_2$, for any $r\in (\mathcal{S}''\cup\{id\})$, the set $r_2^{-1} r\mathcal{C}(\Gamma)$ does not lie in the component of $\mathbb{D}-\mathcal{C}(\Gamma)$ bounded by $b'$. This implies that $\Lambda''$ and $r_{b'} \Lambda(\Gamma)$ belong to the distinct sets $I_{b',1}$ and $I_{b',2}$, respectively.}} Since $\xi^{\Theta,D}$ is $\hat{\rho}(\hat{\Gamma})$-equivariant, the induction hypothesis implies that $\xi^{\Theta,D}$ is positive on $\Lambda''$. Moreover, the base case of the induction carries over verbatim, showing that $\xi^{\Theta,D}$ is positive on $\Lambda'' \cup r_{b'}(\Lambda'')$. On the other hand, we have
\[
r_1\Lambda(\Gamma)=r_2r_{b'}\Lambda(\Gamma)\subset r_2r_{\refchangenew{311}{\changednew{b'}}}\big( \Lambda''\big).
\]
So $$\Lambda(\Gamma)\cup \bigcup_{r\in \mathcal{S}'} r\Lambda(\Gamma)\subset r_2 \big( \Lambda'' \cup r_{b'}(\Lambda'')\big).$$ From the $\hat{\rho}(\hat{\Gamma})$-equivariance again, $\xi^{\Theta,D}$ is positive on $\Lambda(\Gamma)\cup \bigcup_{r\in \mathcal{S}'} r\Lambda(\Gamma)$.

\textbf{Step 6: Concluding the proof.}

\refchangenew{306}
{
\changednew{
Thus far, we have shown that $\rho^{D}:\Gamma^{D}\to Aut_1(\mathfrak{g})$ is a representation admitting a positive and equivariant map $\xi^{\Theta,D}:\Gamma_{R}(\Lambda(\Gamma))\to \mathcal{F}_{\Theta}$. To conclude the proof, it remains to show that $\xi^{\Theta,D}$ extends to a continuous positive map (such an extension is automatically equivariant).

Note that by the construction of the doubling $\Gamma^{D}$, every $\Gamma^{D}$-cusp point in $S^1$ lies in the $\hat{\Gamma}$-orbit of a $\Gamma$-cusp point in $\Lambda(\Gamma)$. Since $\xi^{\Theta,D}$ is $\hat{\Gamma}$-equivariant and $\xi^{\Theta}$ is continuous on $\Lambda(\Gamma)$ which is a perfect set, it follows that $\xi^{\Theta,D}$ is continuous at all $\Gamma^{D}$-cusp points in $S^1$.

By Flamm-Tholozan-Wang-Zhang~\cite[Theorem~C]{flamm2026thetapositiverepresentationsrealclosed}, there exists a unique left-continuous, $\rho^{D}$-equivariant, and positive map $\xi^{\Theta,l}:S^1\to \mathcal{F}_{\Theta}$ that is continuous at conical limit points. The equivariance implies that for any hyperbolic $\gamma\in \Gamma$, the points $\xi^{\Theta,l}(\gamma^{\pm})\in \mathcal{F}_{\Theta}$ are the attractor and repeller of $\rho(\gamma)$. The continuity at the conical limit sets implies that $\xi^{\Theta,l}|_{\Lambda(\Gamma)} = \xi^{\Theta}$, and the $\Gamma^{D}$-equivariance implies that $\xi^{\Theta,D} = \xi^{\Theta,l}$ on $\Gamma^{D}(\Lambda(\Gamma))$, which is a dense subset of $S^1$. The continuity of $\xi^{\Theta,D}$ at cusp points and the continuity of $\xi^{\Theta,l}$ at conical limit points together imply that $\xi^{\Theta,D}$ and $\xi^{\Theta,l}$ are continuous and coincide everywhere. Since $\xi^{\Theta,l}$ is positive, so is $\xi^{\Theta,D}$, and this concludes the proof.
}
}
\end{proof}

To show Theorem \ref{thmA} (3), We invoke the strict entropy drop result when $\Gamma$ is geometrically finite. \refchange{314}{\changed{Although the original statement was for relatively Anosov representations (and note that in their setting they did not distinguish between relatively Anosov and Anosov) into $\mathsf{PGL}(d,\mathbb{R})$, we observe that $\Theta$-transverse representations of geometrically finite groups are $\Theta$-relatively Anosov. Moreover, Lemma~\ref{lem4} and Proposition~\ref{linearize} show that we can linearize transverse representations to general Lie groups just as in the case of $\mathsf{PGL}(d,\mathbb{R})$.}}

\begin{proposition}[\protect{\cite[Proposition 11.5]{CaZhZi}}]\label{prop7} 
Suppose $\Gamma$ is \refchange{314}{\changed{a non-elementary geometrically finite Fuchsian group}} and $\rho:\Gamma\to G $ is a relatively $\alpha-$Anosov representation. If $\Gamma'\subset \Gamma$ is an infinite index, geometrically finite subgroup, then we have
$$
\delta^{\alpha}_{\rho|\Gamma'}(\Gamma')<\delta^{\alpha}_{\rho}(\Gamma).
$$
\end{proposition}
Now Theorem \ref{thmA}~(3) is a corollary of Proposition \ref{whenS1}, Proposition \ref{prop6}, Proposition \ref{prop7}, and Theorem \ref{thmA} (2).

\vspace{3mm}

\section{Proof of Theorem \ref{thmC}}\label{append2}
By Theorem \ref{thmB}, we have  
\[
\delta^{\alpha}_{\rho}(\Gamma) \ge \dim \xi^{\alpha}(\Lambda_c(\Gamma)).
\]  
Therefore, it suffices to establish the corresponding lower bound. Specifically, we aim to prove that for any \emph{non-boundary root} \(\alpha \in \Theta\),  
$
\dim(\xi^{\alpha}(\Lambda_c(\Gamma))) \ge \delta^{\alpha}_{\rho}(\Gamma),
$ 
and for the \emph{boundary root} \(\alpha_{\Theta}\),  
$
\dim(\xi^{\alpha_{\Theta}}(\Lambda_c(\Gamma))) \ge \frac{r}{2} \delta^{\omega'_{\alpha_{\Theta}}}_{\rho}(\Gamma),
$  
where \(\omega'_{\alpha_{\Theta}}\) denotes the fundamental weight of \(\alpha_{\Theta}\) considered as a root of \(S_{\Theta - \alpha_{\Theta}}\), and \(r\) is the real rank of \(S_{\Theta - \alpha_{\Theta}}\).

Let \(\Gamma \subset \mathsf{PSL}(2, \mathbb{R})\) be a \refchange{318}{\changed{non-elementary}} discrete subgroup, and let \(G\) be a real simple Lie group as assumed at the beginning of Section \ref{sec3}, that is, \refchange{319}{\changed{$G$ has finitely many components, its identity component $G^{\circ}$ has finite center, and if $\mathfrak{g}$ is the Lie algebra of $G$ with adjoint representation $\psi:G \to Aut(\mathfrak{g})$, then $\psi(G) \subset Aut_1(\mathfrak{g})$ (see Section~\ref{assumption}).}}

Fix a basepoint \(b_0 \in \mathbb{D}\). For any \(x \ne b_0 \in \mathbb{D}\), \refchange{320}{\changed{define 
$
l_{b_0}(x) = \overrightarrow{b_0x} \cap \partial \mathbb{D},
$ 
where \(\overrightarrow{b_0x}\) denotes the geodesic ray from \(b_0\) toward \(x\).}}

\begin{definition}
Let \(\rho: \Gamma \to G\) be a \(\Theta\)-transverse representation, and let \(\phi \in \mathfrak{a}^*\). We say that \(\phi\) has \emph{strong coarse additivity} if for any \(\epsilon > 0\), there exists a constant \(C > 0\) such that whenever \refchange{321}{\changed{$\gamma,\eta\in \Gamma$ satisfy}} 
\[
d_{S^1}(l_{b_0}(\gamma^{-1}(b_0)), l_{b_0}(\eta(b_0))) > \epsilon,
\]  
the following inequality holds:
\[
\refchange{333,335,336,342}{\changed{|\phi(\kappa(\rho(\gamma\eta))) -\phi(\kappa(\rho(\gamma))) - \phi(\kappa(\rho(\eta)))| \le  C.}}
\]
\end{definition}
\refchange{333,335,336,342}{\changed{Note that the inequality in the definition involves absolute values, so linear combinations of strongly coarse additive functionals remain strongly coarse additive.}}

\refchangenew{19}{\changednew{Recall that $(\mathfrak{a}^*)^+ \subset \mathfrak{a}^*$ denotes the cone of functionals positive on $\mathfrak{a}^+$ (which is a closed cone with $\{0\}$ removed)}}. For any simple root $\alpha\in \Delta$, let $\omega_{\alpha}\in (\mathfrak{a}^*)^+$ denote the associated fundamental weight.  We have

\begin{proposition}{\refchange{322-328}{\changed{\cite[\refchangenew{322-328}{\changednew{Proposition B.1 and Lemma B.2}}]{CaZhZi3}}}}\label{stronglycoarseaddprop}
Let $\rho \colon \Gamma \to G$ be a $\Theta$-transverse representation. For any $\alpha \in \Theta$, the weight $\omega_\alpha$ satisfies strong coarse additivity.
\end{proposition}
\changed{
Note that the cited Lemma~B.2 is stated for transverse representations into $\mathsf{PGL}(d,\mathbb{R})$, 
but Proposition~B.1~(2) and~(6) extends this result to our setting.
And although the reference assumes that $G$ is connected, when $G$ has finitely many components we can pass to the finite-index subgroup 
\[
\Gamma_0 = \{\gamma \in \Gamma \mid \rho(\gamma) \in G^{\circ}\},
\] 
as in the discussion following Remark~\ref{convexcocompact}. In this way, Proposition~\ref{stronglycoarseaddprop} extends to the general case.
}

\vspace{3mm}

The strong coarse additivity is used in proving the following proposition:

\begin{proposition}\label{prop10}
Let $\Gamma\subset \mathsf{PSL}(2,\mathbb{R})$ be a \refchange{329}{\changed{non-elementary}} discrete subgroup and $\rho:\Gamma\to G$ be a $\Theta$-transverse representation satisfying the regular distortion property. Fix a root $\alpha\in \Theta$, suppose $\phi\in (\mathfrak{a}^*)^{+}$ satisfies the strong coarse additivity, \changed{$\delta_{\rho}^{\phi}(\Gamma) <\infty$}, and $\phi \ge A\alpha$ \refchange{330}{\changed{when evaluating on the positive Weyl Chamber}} for some $A>0$. Then $\dim \xi^{\alpha}(\Lambda_{c}(\Gamma))\ge A\delta_{\rho}^{\phi}(\Gamma)$.
\end{proposition}

The proof of Proposition \ref{prop10} is repeating the Bishop-Jones~\refchange{331}{\changed{\cite{bishop1994hausdorffdimensionkleiniangroups}}} argument carried out in Canary-Zhang-Zimmer~\refchange{332}{\changed{\cite[Section 8]{CaZhZi}}}. We will introduce the idea briefly after proving Theorem \ref{thmC}.

\begin{proof}[Proof of Theorem \ref{thmC}]
\refchangenew{343}{\changednew{Recall that Theorem~\ref{thmA} established that every $\alpha\in \Theta$ has a finite critical exponent.}}

First we prove for each non-boundary root $\alpha\in \Theta$, it has the strong coarse additivity. We discuss it case by case according to Theorem \ref{thm7}.

\begin{enumerate}
    \item If $\Delta = \Theta$, then $\mathrm{Span}(\alpha,\alpha\in\Theta) = \mathrm{Span}(\omega_{\alpha},\alpha\in \Theta)$. As each $\omega_{\alpha}$ has the strong coarse additivity, so does $\alpha$.

    \item If $\Theta=\{\alpha_r\}$, then there is no non-boundary root so we do not need to prove anything.

    \item If $G$ is locally isomorphic to $\mathsf{SO}(p,q)$, then \refchange{334}{\changed{the non-boundary roots are $\alpha_k, k<p-1$}}. \refchangenew{334}{\changednew{For such $k$, one can check from the Cartan matrix}} that $\alpha_k = 2\omega_{\alpha_{k}}-\omega_{\alpha_{k-1}}-\omega_{\alpha_{k+1}}$, which is a linear combination of strongly coarse additive elements, so $\alpha_k$ is strongly coarse additive. 

    \item If the restricted root system of $G$ is of type $\mathsf{F}_4$, then \refchange{336}{\changed{the only non-boundary root is $\alpha_1$}}. From the Cartan matrix of $\mathsf{F}_4$ (for example, see Humphreys~\cite[Seciton 11.4]{humphreys1994introduction}), $\alpha_1 =2\omega_{\alpha_1}-\omega_{\alpha_2} $, so $\alpha_1$ is strongly coarse additive.
\end{enumerate}

As $\Theta$-positive representations \refchange{339}{\changed{are $\Theta$-transverse (Proposition \ref{thm17})}} and \refchange{337}{\changed{have}} the regular distortion property (Theorem \ref{thm12}), Proposition \ref{prop10} and the above discussions show that $\dim \xi^{\alpha}(\Lambda_{c}(\Gamma))\ge \delta_{\rho}^{\alpha}(\Gamma)$. Together with Theorem \ref{thmB}, we have proved Theorem \ref{thmC}~(1). 

Moreover, if the representation is transverse on the adjacent roots of $\Theta$,  
\refchange{340}{{\color{blue}{then, by the properties of the Dynkin diagram and the Cartan matrix, two roots $\alpha$ and $\beta$ are linked if and only if $\alpha$ has a nonzero $\omega_{\beta}$-coefficient in the basis of fundamental weights. Hence, transversality on the adjacent roots of $\Theta$ implies that for any $\alpha \in \Theta$, the root $\alpha$ can be expressed as a linear combination of weights with strong coarse additivity. So any $\alpha\in \Theta$ is strongly coarse additive. From Theorem \ref{thmB} and Proposition \ref{prop10}, we have proved Theorem \ref{thmC}~(3). }}}

As to Theorem \ref{thmC}~(2), recall from Theorem \ref{thm8} that $S_{\Theta-\alpha_{\Theta}}$ is Hermitian of tube type with root system generated by $(\Delta-\Theta)\cup \{\alpha_{\Theta}\}$ \refchange{341}{\changed{(see Guichard-Wienhard~\cite[Section 2.7 and Section 3.5]{GW2})}}, and $\alpha_{\Theta}$ is the long root. If $\omega'_{\alpha_{\Theta}}$ is the fundamental weight of $\alpha_{\Theta}$ in $S_{\Theta-\alpha_{\Theta}}$, then $\omega'_{\alpha_{\Theta}}\in \mathrm{Span}(\omega_{\alpha},\alpha\in \Theta)\cap \mathrm{Span}(\beta,\beta\in (\Delta-\Theta)\cup \{\alpha_{\Theta}\})$, which is strongly coarse additive.

According to Theorem \ref{thmB} and Proposition \ref{prop10}, \refchange{343}{\changed{it suffices to show the constant $A$ in Proposition \ref{prop10} can be chosen as $\frac{r}{2}$, that is to prove $\omega_{\alpha_{\Theta}}'\ge \frac{r}{2} \alpha_{\Theta}$. \refchangenew{343}{\changednew{Note that once this is established, the fact that $\alpha_{\Theta}$ has a finite critical exponent directly implies that $\omega_{\alpha_{\Theta}}'$ also has a finite critical exponent.}}

And this inequality only depends on the structure of the root system, so from the discussions in the last paragraph, we only need to prove:}} For Hermitian tube type Lie groups, $\omega_{\alpha_r}\ge \frac{r}{2} \alpha_{r}$ where $r$ is the real rank and $\alpha_{r}$ is the long root. And as the root system is of type $\mathsf{C}_r$ (see Guichard-Wienhard~\cite[Section 3.5]{GW2}), we only need to check the case when $G = \mathsf{Sp}(2r,\mathbb{R})$. Now if the Cartan subalgebra of $\mathfrak{sp}(2r,\mathbb{R})$ is identified as $\mathrm{diag}(\lambda_1,...,\lambda_r,-\lambda_1,...,-\lambda_r)$ where $\lambda_1>...>\lambda_r>0$, \refchange{344}{\changed{then we have $\alpha_r: \mathrm{diag}(...)\mapsto  2\lambda_r$ and $\omega_{\alpha_r}:\mathrm{diag}(...)\mapsto   \lambda_1+\cdots+\lambda_r$}}, so clearly $\omega_{\alpha_r}\ge \frac{r}{2} \alpha_r$.

\end{proof}

Now we sketch the proof of Proposition \ref{prop10}. As $\Gamma$ is discrete, one can pick $b_0\in \mathbb{D}$ such that $\mathsf{Stab}_{\Gamma}(b_0) = \{id\}$, and for any $\phi\in \mathfrak{a}^{*}$, define $c_{\phi}:\Gamma(b_0)\to \mathbb{R}$ by $c_{\phi}(\gamma(b_0)) = e^{-\phi(\kappa(\rho(\gamma)))}$.

Essentially, the strong coarse additivity is used in the following construction in Canary-Zhang-Zimmer~\cite{CaZhZi} (although the reference did not write it as an assumption).

\begin{proposition}[\protect{\cite[Proposition 8.2]{CaZhZi}}]\label{prop8}
\changed{Let $\Gamma\subset \mathsf{PSL}(2,\mathbb{R})$ be a non-elementary discrete subgroup and $\rho:\Gamma\to G$ be a $\Theta$-transverse representation.} Suppose $\phi\in (\mathfrak{a}^*)^+$ has strong coarse additivity \changed{and $\delta^{\phi}_{\rho}(\Gamma)<\infty$} . For any $0<\delta<\delta^{\phi}_{\rho}(\Gamma)$, there exists $r_0,D_0>0$ such that if $z\in \Gamma(b_0)$, then there exists a finite subset $\mathcal{C}_{\delta}(z)$ of $\Gamma_0(b_0)-\{z\}$ with the following properties
\begin{enumerate}
    \item if $w\in \mathcal{C}_{\delta}(z)$, then $\mathcal{O}_{2r_0}(b_0,w)\subset \mathcal{O}_{r_0}(b_0,z)$.
    \item $\{\mathcal{O}_{2r_0}(b_0,w)|w\in \mathcal{C}_{\delta}(z)\}$ are pairwise disjoint.
    \item if $w\in \mathcal{C}_{\delta}(z)$, then $d_{\mathbb{D}}(z,w)\le D_0$.
    \item $\sum_{w\in\mathcal{C}_{\delta}(z)} c_{\phi}^{\delta}(w)\ge c_{\phi}^{\delta}(z)$.
\end{enumerate}
\end{proposition}
\refchange{350}{\changed{As the proof is almost the same as the proof of \cite[Proposition 8.2]{CaZhZi}, we attach it in Appendix~\ref{appendB}.}} \refchangenew{350}{\changednew{The only difference is that we replace their Theorem~7.1 (which is a key feature of the hyperconvexity therein) with our Corollary~\ref{shadowradius}.}}

\refchange{346}{\changed{Fix $0<\delta<\delta^{\phi}_{\rho}(\Gamma)$}}, let $\mathcal{T}_{\delta}$ be the directed tree whose set of vertices is in $\Gamma(b_0)$, and for all $z,w\in \Gamma(b_0)$, there is a directed edge from $z$ to $w$ if and only if $w\in \mathcal{C}_\delta(z)$. Let $E_{\delta}\subset \Lambda(\Gamma)$ be the set of accumulation points of \refchange{347}{\changed{the vertices of $\mathcal{T}_{\delta}$}}. Recall that in Section \ref{ssec6} we defined the $R-$uniformly conical limit set $\Lambda_{b_0, R}(\Gamma)$ with respect to $b_0$.

\refchange{348}{\changed{Once we have Proposition \ref{prop8}, the same proof of Canary-Zhang-Zimmer~\cite[Lemma 8.3 and Lemma 8.4~(1)]{CaZhZi} (which only depend on the properties in Proposition~\ref{prop8}) works and gives the following lemmas: }}
\begin{lemma}[\protect{\cite[Lemma 8.3]{CaZhZi}}]\label{lem17}

For any $x\in E_{\delta}$, there exists a piecewise geodesic ray in $\mathcal{T}_{\delta}$ such that if $\{x_n\}$ is the sequence of vertices of $\mathcal{T}_{\delta}$ along the ray, then $x_0 = b_0, x_n\to x$ and $x \in \cap_{n = 0}^{\infty} \mathcal{O}_{r_0}(b_0,x_n)$. Moreover, there exists $R>0$, such that $E_{\delta}\subset \Lambda_{b_0,R}(\Gamma)$.

\end{lemma}

\begin{lemma}[\protect{\cite[Lemma 8.4~(1)]{CaZhZi}}]\label{lem16} 

There exists a probability measure $\mu$ supported on $E_{\delta}$ such that for any $z\in \mathcal{T}_{\delta}$, one has:
$$
\mu(\mathcal{O}_{2r_0}(b_0,z)) \le c_{\phi}(z)^{\delta}.
$$
\end{lemma}

\begin{proof}[Proof of Proposition \ref{prop10}]

Pick any $x\in E_{\delta}$. Let $\{x_n\}$ be the sequences of vertices in Lemma \ref{lem17}. Assume $x_n = \gamma_n(b_0)$, then $d_{\mathbb{D}}(\gamma_n(b_0),\changed{\overrightarrow{b_0x}})<r_0$. \refchange{353}{\changed{Recall that $c_{\alpha}(x_n) = e^{-\alpha(\kappa(\rho(\gamma_n)))}$.}}

\refchange{350,351,352}{\changed{Corollary \ref{shadowradius}~(2)}} shows that there exists some $C_1>0$, such that for each $n$,
$$
B_{d_{\alpha}}(\xi^{\alpha}(x), C_1c_{\alpha}(x_n))\cap \xi^{\alpha}(\Lambda(\Gamma)) \subset  \mathcal{O}^{\alpha}_{2r_0}(b_0,x_n).
$$
By Proposition \ref{prop8} (3), $\gamma_{n+1}^{-1}\gamma_n$ is bounded independent of $x$ and $n$. The uniform continuity of Cartan projection (see Corollary \ref{cor1}) shows that there exists a constant $C_2>1$ such that $\frac{1}{C_2} c_{\alpha}(x_n) \le c_{\alpha}(x_{n+1})\le C_2 c_{\alpha}(x_n)$. As \refchange{354}{\changed{$\rho$ is $\Theta$-transverse and $\gamma_n(b_0)\to x$, Proposition \ref{cartanproperty}~(2) implies that}} $c_{\alpha}(x_n)\to 0$. So for any $t$ small enough, there always exists a $n$ such that
$$
\frac{1}{(C_2)^2} c_{\alpha}(x_n)\le  t\le  (C_2)^2 c_{\alpha}(x_n).
$$

Let $\mu$ be the probability measure in Lemma \ref{lem16}. Since $\phi\ge A\alpha$ \refchange[2\baselineskip]{355}{\changed{ when evaluating on the positive Weyl chamber}}, \refchange{356}{\changed{we have}} $(c_{\alpha}(z))^{A}\ge c_{\phi}(z)$ \refchange{357}{\changed{for any $z\in \Gamma(b_0)$.}}  So by Lemma \ref{lem16}
$$
\xi^{\alpha}_{*}\mu(B_{d_{\alpha}}(\xi^{\alpha}(x), C_1c_{\alpha}(x_n)) \le c_{\phi}(x_n)^{\delta} \le c_{\alpha}(x_n)^{A\delta}.
$$
Thus, for any $t$ small enough, \refchange{358}{\changed{let $t' = \frac{C_1}{(C_2)^2}t$. Suppose $\frac{1}{(C_2)^2} c_{\alpha}(x_n)\le t\le (C_2)^2 c_{\alpha}(x_n)$, then $\frac{C_1}{(C_2)^4} c_{\alpha}(x_n)\le t' \le C_1 c_{\alpha}(x_n)$. Moreover, we have $$ \xi^{\alpha}_*\mu (B_{d_{\alpha}}(\xi^{\alpha}(x), t')) \le \xi^{\alpha}_*\mu (B_{d_{\alpha}}(\xi^{\alpha}(x), C_1 c_{\alpha}(x_n))) \le c_{\alpha}(x_n)^{A\delta} \le \left(\frac{(C_2)^4}{C_1}\right)^{A\delta} t'^{A\delta}. $$ Let $C' = (\frac{C_1}{(C_2)^4})^{A\delta}$, then the above discussions show that for any $y\in \xi^{\alpha}(E_{\delta})$ and any $t'$ small enough, $$t'^{A\delta}\ge C'\xi^{\alpha}_*\mu (B_{d_{\alpha}}(y, t')).$$}} 

Pick any countable covering $B_{d_{\alpha}}(y_n,t_n)$ of $\xi^{\alpha}(E_{\delta})$. Then $\sum t_n^{A{\delta}}\ge C'\xi^{\alpha}_*\mu(\xi^{\alpha}(E_\delta))=C'$, \refchange{359}{\changed{so the $A\delta$-Hausdorff measure of $\xi^{\alpha}(E_{\delta})$ is non-zero}}, and therefore $\dim \xi^{\alpha}(\Lambda_{c}(\Gamma))\ge \dim \xi^{\alpha}(E_{\delta})\ge A\delta$.

 As $\delta$ is arbitrary in $(0, \delta^{\phi}_{\rho}(\Gamma))$, we have finished the proof.
\end{proof}

\appendix

{\color{blue}{

\section{Proof of Corollary \ref{shadowradius}~(2)}\label{append1}
We restate the Corollary in a stronger version (which implies Corollary \ref{shadowradius}~(2) by Theorem \ref{thm12}):
\begin{corollary}\label{shadowradiusstrong}
Let $\Gamma\subset \mathsf{PSL}(2,\mathbb{R})$ be a non-elementary discrete subgroup, $G$ be a semisimple Lie group $\rho:\Gamma\to G$ be a $\Theta$-transvserse representation with regular distortion property. Fix $b_0\in \mathbb{D}$. For any $r>0,R\ge0$, there exists a constant $C>1$ such that: Let $x = S^1\cap \changed{\overrightarrow{b_0z}}$. If $x\in \Lambda_{b_0, R}(\Gamma)$, then for any $\gamma\in \Gamma$ and $z\in \mathbb{D}$, whenever $d_{\mathbb{D}}(\gamma(b_0),z)<r$, we have
    \[B_{d_{\alpha}}(\xi^{\alpha}(x),\frac{1}{C} e^{-\alpha(\kappa(\rho(\gamma)))})\cap \xi^{\alpha}(\Lambda(\Gamma))\subset \mathcal{O}^{\alpha}_{r}(b_0,z)\subset B_{d_{\alpha}}(\xi^{\alpha}(x),C e^{-\alpha(\kappa(\rho(\gamma)))}). \]
\end{corollary}

The proof follows Canary-Zhang-Zimmer~\cite[Theorem~5.1 and Theorem~7.1]{CaZhZi}.

Recall that we work in the \emph{Klein model} of $\mathbb{D}$, so that $\mathbb{D}$ is identified with a properly convex domain in $\mathbb{RP}^2$. 
For $a,b\in \mathbb{RP}^2$ with $a\neq b$, let $l_{ab}$ denote the projective line through $a$ and $b$ (which, when $a,b\in \overline{\mathbb{D}}$, restricts to the geodesic in $\mathbb{D}$ joining them). 
If $x,y\in \overline{\mathbb{D}}$, we write $[x,y]$, $(x,y)$, $[x,y)$, and $(x,y]$ for the closed, open, and the two half-open geodesic segments with endpoints $x$ and $y$.

Fix $b_0\in \mathbb{D}$. For $x\in \partial\mathbb{D}$, let $x^{opp}$ be the second intersection point of $l_{b_0x}$ with $\partial\mathbb{D}$. 
Let $\mathring{x}$ be the intersection of the two tangent lines to $\partial\mathbb{D}$ at $x$ and at $x^{opp}$ (this may be the point at infinity if the tangents are parallel). 
Define $\pi_x:\partial\mathbb{D}\to [x,x^{opp}]$ by
\[
\pi_x(y)\;=\; l_{\mathring{x}y}\cap [x,x^{opp}].
\]
(When $\mathring{x}=\infty$, interpret $l_{\mathring{x} y}$ as the line through $y$ parallel to these tangents.)

We invoke some lemmas from Canary-Zhang-Zimmer~\cite{CaZhZi}.

\begin{lemma}{\cite[Lemma~7.5]{CaZhZi}}\label{lem9}
For any $r,r'>0$ there exists $T>0$ with the following property: if $x,y\in \Lambda(\Gamma)$, $z\in [b_0,x)$, $d_{\mathbb{D}}\big(z,\Gamma( b_0)\big)\le r'$, $\pi_x(y)\in (z,x)$, and $d_{\mathbb{D}}\big(\pi_x(y),z\big)\ge T$, then $y\in \mathcal{O}_r(b_0,z)$.
\end{lemma}

\begin{lemma}{\cite[Lemma~7.6]{CaZhZi}}\label{lem10}
Let $\rho:\Gamma\to G$ be a $\Theta$-transverse representation with regular distortion property and let $\alpha\in \Theta$. For any $R\ge 0$ and $0<\delta<1$, there exists a constant $C_1>1$ such that for any $x,y\in \Lambda(\Gamma)$ and $\gamma\in \Gamma$, if
\[
d_{\mathbb{D}}\big(\gamma(b_0),\pi_x(y)\big)\le R
\quad\text{and}\quad
d_{S^1}(x,y)\le \delta,
\]
then
\[
d_{\alpha}\big(\xi^\alpha(x),\xi^{\alpha}(y)\big)\ \ge\ \frac{1}{C_1}\,e^{-\alpha(\kappa(\rho(\gamma)))}.
\]
\end{lemma}

\begin{proof}
Assume for contradiction that the claim fails. Then there exist sequences $x_n,y_n\in \Lambda(\Gamma)$ and $\gamma_n\in \Gamma$ such that
\[
0< d_{S^1}(x_n,y_n)\le \delta,\qquad
d_{\mathbb{D}}\big(\gamma_n(b_0),\pi_{x_n}(y_n)\big)\le R,
\]
but
\[
d_{\alpha}\big(\xi^{\alpha}(x_n),\xi^{\alpha}(y_n)\big)
\ \le\ \frac{1}{n}\, e^{-\alpha(\kappa(\rho(\gamma_n)))}.
\]
Set $p_n=\gamma_n^{-1}(y_n)$ and $q_n=\gamma_n^{-1}(x_n)$. Passing to a subsequence, assume
\[
x_n\to x\in \Lambda(\Gamma),\quad y_n\to y\in \Lambda(\Gamma),\qquad
\gamma_n(b_0)\to b\in \overline{\mathbb{D}},\quad \gamma_n^{-1}(b_0)\to a\in \overline{\mathbb{D}},\quad
\gamma_n^{-1}\!\big(\pi_{x_n}(y_n)\big)\to z\in \overline{B_{d_{\mathbb{D}}}(b_0,R)},
\]
and
\[
p_n\to p\in \Lambda(\Gamma),\qquad q_n\to q\in \Lambda(\Gamma),\qquad \gamma_n^{-1}\!\big(x_n^{opp}\big)\to \hat a\in \Lambda(\Gamma).
\]

Since the right-hand side of the inequality above tends to $0$ as $n\to\infty$, we have
$d_{\alpha}\big(\xi^{\alpha}(x_n),\xi^{\alpha}(y_n)\big)\to 0$, and because the limit map $\xi^{\alpha}$ is injective, it follows that $x=y$. Hence $d_{S^1}(x_n,y_n)\to 0$ and $\pi_{x_n}(y_n)\to x$. As $d_{\mathbb{D}}\big(\gamma_n(b_0),\pi_{x_n}(y_n)\big)\le R$, we obtain $\gamma_n(b_0)\to x$, so $\{\gamma_n\}$ is unbounded and therefore $a,b\in \Lambda(\Gamma)$; moreover, $\hat a=a$.

Let $l_n := \gamma_n^{-1}\!\big(l_{b_0 x_n}\big)$. Then $\gamma_n^{-1}(b_0),\, \gamma_n^{-1}\!\big(\pi_{x_n}(y_n)\big)\in l_n$. Since $\gamma_n^{-1}\!\big(\pi_{x_n}(y_n)\big)\to z\in \overline{B_{d_{\mathbb{D}}}(b_0,R)}$, the triple $(a,p,q)$ is distinct. Hence there exists $d>0$ such that, for all sufficiently large $n$, the triple $(\omega_{\gamma_n},p_n,q_n)$ is $d$–bounded. By the regular distortion property, there exists $C_2>0$ with
\[
d_{\alpha}\big(\xi^{\alpha}(x_n),\xi^{\alpha}(y_n)\big)
= d_{\alpha}\big(\rho(\gamma_n)\xi^{\alpha}(p_n),\,\rho(\gamma_n)\xi^{\alpha}(q_n)\big)
\ \ge\ C_2\, e^{-\alpha(\kappa(\rho(\gamma_n)))},
\]
which contradicts $d_{\alpha}\big(\xi^{\alpha}(x_n),\xi^{\alpha}(y_n)\big)\le \frac{1}{n} e^{-\alpha(\kappa(\rho(\gamma_n)))}$. 
\end{proof}

\begin{proof}[Proof of Corollary~\ref{shadowradiusstrong}.]
We first show that for any $r>0,R\ge0$ there exists a constant $C>1$ such that for any $\alpha\in \Theta$, $x\in \Lambda_{b_0,R}(\Gamma)$, $z\in [b_0,x)$, and $\gamma\in \Gamma$ with $d_{\mathbb{D}}(z,\gamma(b_0))<r$, one has
\[
B_{d_{\alpha}}\!\left(\xi^{\alpha}(x),\frac{1}{C}\,e^{-\alpha(\kappa(\rho(\gamma)))}\right)\cap \xi^{\alpha}(\Lambda(\Gamma))
\ \subset\ \mathcal{O}^{\alpha}_{r}(b_0,z).
\]

Assume not. Then there exist $\alpha$ and sequences $x_n\in \Lambda_{b_0,R}(\Gamma)$, $y_n\in \Lambda(\Gamma)$, $z_n\in [b_0,x_n)$, $\gamma_n\in \Gamma$ such that
\[
d_{\mathbb{D}}(z_n,\gamma_n(b_0))<r,\qquad y_n\notin \mathcal{O}_{r}(b_0,z_n),\qquad
d_{\alpha}\!\big(\xi^{\alpha}(x_n),\xi^{\alpha}(y_n)\big)<\frac{1}{n}\,e^{-\alpha(\kappa(\rho(\gamma_n)))}.
\]
Passing to a subsequence, assume $x_n\to x\in \Lambda_{b_0,R}(\Gamma)$, $y_n\to y\in \Lambda(\Gamma)$, and $\gamma_n(b_0)\to b\in \overline{\mathbb{D}}$. Since $\xi^{\alpha}$ is continuous and injective and $d_{\alpha}\!\big(\xi^{\alpha}(x_n),\xi^{\alpha}(y_n)\big)<\frac{1}{n}\,e^{-\alpha(\kappa(\rho(\gamma_n)))}$, we must have $x=y$.

Let $w_n:=\pi_{x_n}(y_n)$. Because $x_n\in \Lambda_{b_0,R}(\Gamma)$, there exists $\beta_n\in \Gamma$ with $d_{\mathbb{D}}(w_n,\beta_n(b_0))\le R$. By Lemma~\ref{lem10} there is $C_1>1$ such that
\[
d_{\alpha}\!\big(\xi^{\alpha}(x_n),\xi^{\alpha}(y_n)\big)\ \ge\ \frac{1}{C_1}\,e^{-\alpha(\kappa(\rho(\beta_n)))}.
\]
Together with $d_{\alpha}\!\big(\xi^{\alpha}(x_n),\xi^{\alpha}(y_n)\big)<\frac{1}{n}\,e^{-\alpha(\kappa(\rho(\gamma_n)))}$, this yields
\[
\alpha\!\big(\kappa(\rho(\beta_n))\big)\ \ge\ \alpha\!\big(\kappa(\rho(\gamma_n))\big)+\log n-\log C_1. 
\]

Set $\eta_n:=\beta_n^{-1}\gamma_n\in \Gamma$. By Corollary~\ref{cor1} (uniform continuity of the Cartan projection), the above sequence of inequalities implies that $\{\eta_n\}$ is unbounded, hence $d_{\mathbb{D}}(b_0,\eta_n b_0)\to\infty$. It follows that $d_{\mathbb{D}}(\gamma_n(b_0),\beta_n(b_0))\to\infty$, and therefore $d_{\mathbb{D}}(z_n,w_n)\to\infty$ since $d_{\mathbb{D}}(z_n,\gamma_n(b_0))< r$ and $d_{\mathbb{D}}(w_n,\beta_n(b_0))\le R$.

By Lemma~\ref{lem9}, $w_n$ cannot lie in $(z_n,x_n)$ for infinitely many $n$ (otherwise $y_n$ would be in the shadow), so after passing to a subsequence we may suppose $w_n\in (b_0,z_n]$. Using $d_{\mathbb{D}}(\gamma_n(b_0),z_n)< r$ and $d_{\mathbb{D}}(\beta_n(b_0),w_n)\le R$, we estimate
$$
d_{\mathbb{D}}\!\big(\beta_n(b_0),[b_0,\gamma_n(b_0)]\big)
\le d_{\mathbb{D}}\!\big(\beta_n(b_0),w_n\big)+ d_{\mathbb{D}}\!\big(w_n,[b_0,\gamma_n(b_0)]\big)
$$
$$
\le R+ d_{\mathbb{D}}\!\big(z_n,[b_0,\gamma_n(b_0)]\big) \\
\le R+r.
$$
Here we used the fact from hyperbolic geometry that for any $y\in \overline{\mathbb{D}}, z\in \mathbb{D}$, $w\in[b_0,z]\to  d_{\mathbb{D}}\big(w,[b_0,y)\big)\le  d_{\mathbb{D}}\big(z,[b_0,y)\big)$.

Applying Lemma~\ref{partialcoarseadditive}, we obtain a constant $C_0>0$ such that
\[
\alpha\!\big(\kappa(\rho(\gamma_n))\big)\ \ge\ \alpha\!\big(\kappa(\rho(\beta_n))\big) + \alpha\!\big(\kappa(\rho(\beta_n^{-1}\gamma_n))\big)-C_0\ge \ \alpha\!\big(\kappa(\rho(\beta_n))\big)-C_0,
\]
which contradicts inequality $\alpha\!\big(\kappa(\rho(\beta_n))\big)\ \ge\ \alpha\!\big(\kappa(\rho(\gamma_n))\big)+\log n-\log C_1$ for all sufficiently large $n$. This proves the first inclusion.

\medskip
For the converse inclusion,
\[
\mathcal{O}^{\alpha}_{r}(b_0,z)\ \subset\ B_{d_{\alpha}}\!\big(\xi^{\alpha}(x),\, C\, e^{-\alpha(\kappa(\rho(\gamma)))}\big),
\]
we appeal to \cite[Theorem~5.1]{CaZhZi}. Although the result there is stated for transverse representations into $\mathsf{PGL}(d,\mathbb{R})$, Lemma \ref{lem4} and Proposition~\ref{linearize} allow us to linearize general transverse representations into that setting, with the flag manifold embedded equivariantly via the adapted representation. This yields the desired bound and completes the proof. 
\end{proof}
}}

{\color{blue}{
\section{Proof of Proposition~\ref{prop8}}\label{appendB}
In this appendix, we reproduce the proof by Canary-Zhang-Zimmer~\cite[Proposition~8.2]{CaZhZi}.

We begin with several lemmas.  
Let $\Gamma \subset \mathsf{PSL}(2,\mathbb{R})$ be a discrete subgroup, $G$ a semisimple Lie group, and $\rho:\Gamma \to G$ a $\Theta$-transverse representation.  
Suppose $\phi \in (\mathfrak{a}^*)^+$. Fix $b_0\in\mathbb{D}$ which satisfies $\mathrm{Stab}_{\Gamma}(b_0) = \{id\}$, and for each $w=\gamma(b_0)\in \Gamma(b_0)$, recall that we define
\(
c_{\phi}(w) \;=\; e^{-\phi(\kappa(\rho(\gamma)))}.
\)
Moreover, in this appendix we define
\(
\mathcal{A}_n \;:=\; \Bigl\{\,z\in \Gamma(b_0)\ \Bigm|\ e^{-(n+1)} < c_{\phi}(z) \le e^{-n}\,\Bigr\}.
\)

We say that $\phi \in (\mathfrak{a}^*)^+$ is \emph{divergent} if, for every sequence of distinct points $w_n \in \Gamma(b_0)$, one has $c_{\phi}(w_n)\to 0$.  
In this case, each set $\mathcal{A}_n$ is finite.  
Clearly, the condition $\delta_{\rho}^{\phi}(\Gamma)<\infty$ implies that $\phi$ is divergent.  
Moreover, for $\Theta$-transverse representations, every $\alpha \in \Theta$ is also divergent.

\begin{lemma}[{\cite[Lemma~8.5]{CaZhZi}}]\label{lem11refined}
Let $\rho:\Gamma\to G$ be a $\Theta$-transverse representation.  
Suppose $\phi$ is either a positive multiple of a root in $\Theta$, or a divergent functional in $(\mathfrak{a}^*)^+$ that satisfies strong coarse additivity.  
Then, for any $r>0$, there exists a constant $C_0=C_0(r)>0$ such that for all $n\ge 0$ and all $z,w\in \mathcal{A}_n$ with $d_{\mathbb{D}}(z,w)>C_0$, one has
\[
\mathcal{O}_r(b_0,z)\,\cap\, \mathcal{O}_r(b_0,w)\;=\;\varnothing.
\]
\end{lemma}

\begin{proof}
We prove the contrapositive. Suppose $\gamma(b_0),\eta(b_0)\in \mathcal{A}_n$ and
\[
x\in \mathcal{O}_r(b_0,\gamma(b_0))\cap \mathcal{O}_r(b_0,\eta(b_0)).
\]
Then there exist $z',w'\in [b_0,x)$ with $d_{\mathbb{D}}(\gamma(b_0),z')\le  r$ and $d_{\mathbb{D}}(\eta(b_0),w')\le r$. 
Without loss of generality we assume that $z'\in [b_0,w']$. Then
\[
d_{\mathbb{D}}\big(\gamma(b_0),[b_0,\eta(b_0)]\big)
\ \le\ d_{\mathbb{D}}(\gamma(b_0),z')+d_{\mathbb{D}}\big(z',[b_0,\eta(b_0)]\big)
\ \le\ 2r.
\]
Here we again used the fact from hyperbolic geometry that for any $y\in \overline{\mathbb{D}}, z\in \mathbb{D}$, $w\in[b_0,z]\to  d_{\mathbb{D}}\big(w,[b_0,y)\big)\le  d_{\mathbb{D}}\big(z,[b_0,y)\big)$.

If $\phi$ is a positive multiple of a root in $\Theta$, then by Lemma~\ref{partialcoarseadditive}, there exists $C=C(2r)>0$ such that
\[
\phi\!\big(\kappa(\rho(\eta))\big)
\ \ge\ 
\phi\!\big(\kappa(\rho(\gamma))\big)
\;+\;
\phi\!\big(\kappa(\rho(\gamma^{-1}\eta))\big)
\;-\; C.
\]
If $\phi$ has strong coarse additivity, then the condition 
$
d_{\mathbb{D}}\big(\gamma(b_0),[b_0,\eta(b_0)]\big)\le 2r
$
is equivalent to
$
d_{\mathbb{D}}\big(b_0,[\gamma^{-1}(b_0),\gamma^{-1}\eta(b_0)]\big)\le 2r
$
which implies that 
$
d_{S^1}\big(l_{b_0}(\gamma^{-1}(b_0)),\,l_{b_0}(\gamma^{-1}\eta(b_0))\big) > \epsilon
$
for some $\epsilon=\epsilon(2r)>0$.  
Hence, by the definition of strong coarse additivity, there still exists a constant $C=C(2r)>0$ such that
\[
\phi\!\big(\kappa(\rho(\eta))\big)
\;\ge\;
\phi\!\big(\kappa(\rho(\gamma))\big)
\;+\;
\phi\!\big(\kappa(\rho(\gamma^{-1}\eta))\big)
\;-\; C.
\]

In both cases, since $\gamma,\eta\in \mathcal{A}_n$, we have $e^{-(n+1)}<c_{\phi}(\gamma),c_{\phi}(\eta)\le e^{-n}$, hence
\[
\phi\!\big(\kappa(\rho(\gamma^{-1}\eta))\big)
\ \le\ C \;+\; \phi\!\big(\kappa(\rho(\eta))\big)-\phi\!\big(\kappa(\rho(\gamma))\big)
\ \le\ C+1.
\]
Because $\phi$ is divergent in both cases, a uniform bound on $\phi(\kappa(\rho(\gamma^{-1}\eta)))$ implies a uniform bound on 
$d_{\mathbb{D}}\!\big(b_0,(\gamma^{-1}\eta)(b_0)\big)$. Thus there exists $C_0>0$ such that
\[
\gamma(b_0),\eta(b_0)\in \mathcal{A}_n
\quad\text{and}\quad
\mathcal{O}_r(b_0,\gamma(b_0))\cap \mathcal{O}_r(b_0,\eta(b_0))\neq \emptyset
\ \Longrightarrow\
d_{\mathbb{D}}\big(\gamma(b_0),\eta(b_0)\big)\le C_0.
\]
This proves the contrapositive.
\end{proof}

The second lemma holds for general projectively visible subgroups (a class that includes discrete subgroups of $\mathsf{PSL}(2,\mathbb{R})$ acting on the Klein model of $\mathbb{D}$); see Canary--Zhang--Zimmer~\cite{CaZhZi}.
\begin{lemma}[{\cite[Lemma~8.6]{CaZhZi}}]\label{lem12}
Let $\Gamma\subset \mathsf{PSL}(2,\mathbb{R})$ be a discrete subgroup. Given $t>0$, there exist $r_0=r_0(t)>0$ and $N_0=N_0(t)>0$ such that if $z,w\in \Gamma(b_0)-\{b_0\}$, $d_{\mathbb{D}}(b_0,w)\ge N_0$, and $d_{S^1}\big(l_{b_0}(z),l_{b_0}(w)\big)\ge t$, then
\[
\mathcal{O}_{2r_0}(z,w)\ \subset\ \mathcal{O}_{r_0}(z,b_0)\cap \mathcal{O}_{2r_0+1}(b_0,w).
\]
\end{lemma}

Finally we refine Lemma~\ref{lem13} with a proof from Canary-Zhang-Zimmer~\cite[Lemma~8.7]{CaZhZi}. 
For $t>0$ and $y\in S^1$, define
\[
\mathcal{B}(y,t)\ :=\ \bigl\{\,w\in \Gamma(b_0)-\{b_0\}\ \bigm|\ d_{S^1}\big(l_{b_0}(w),y\big)<t\,\bigr\}.
\]

Suppose $\phi \in (\mathfrak{a}^*)^+$ satisfy $\delta^{\phi}_{\rho}(\Gamma)<\infty$. Fix $\delta,\epsilon>0$ such that $0<\delta<\delta+\epsilon<\delta^{\phi}_{\rho}(\Gamma)$. By definition of the critical exponent,
\[
\sum_{w\in \Gamma(b_0)} c_{\phi}(w)^{\delta+\epsilon}\ =\ \infty.
\]
Hence for each $n>0$ there exists $x_n\in \Lambda(\Gamma)$ such that
\[
\sum_{w\in \mathcal{B}(x_n,\tfrac{1}{n})} c_{\phi}(w)^{\delta+\epsilon}\ =\ \infty.
\]
Passing to a subsequence, suppose $x_n\to x$. Then, for every $t>0$,
\[
\sum_{w\in \mathcal{B}(x,t)} c_{\phi}(w)^{\delta+\epsilon}\ =\ \infty.
\]

\begin{lemma}{\cite[Lemma~8.7]{CaZhZi}}\label{lem13refined}
In the above setting, if $y\in \Gamma(x)$ and $t>0$, then
\[
\limsup_{n\to\infty}\ \sum_{w\in \mathcal{A}_n\cap \mathcal{B}(y,t)} c_{\phi}(w)^{\delta}\ =\ \infty.
\]
\end{lemma}

\begin{proof}
Choose $\gamma\in \Gamma$ with $\gamma(x)=y$. By the uniform continuity of the Cartan projection (Corollary~\ref{cor1}), there exists $D>1$ such that for all $w\in \Gamma(b_0)$,
\[
\frac{1}{D}\,c_{\phi}(\gamma(w))\ \le\ c_{\phi}(w)\ \le\ D\,c_{\phi}(\gamma(w)).
\]

\refchangenew{276}{\changednew{Assuming that the action of $\gamma$ on $S^1$ is $D'$-bi-Lipschitz, we have $\gamma\big(\mathbb{B}_{d_{S^1}}(x,\frac{t}{2D'})\big)\subset \mathbb{B}_{d_{S^1}}(y,\frac{t}{2})$. Thus, for $n\gg 1$,
\[
\mathcal{A}_n\cap \gamma \mathcal{B}\left(x,\frac{t}{2D'}\right) \subset \mathcal{A}_n\cap \mathcal{B}(y,t).
\]
Consequently, for $n\gg1$,
\[
\sum_{w\in \mathcal{A}_n\cap \mathcal{B}(y,t)} c_{\phi}(w)^{\delta}
\ge \sum_{w\in \mathcal{A}_n\cap \gamma\left( \mathcal{B}(x,\frac{t}{2D'})\right)} c_{\phi}(w)^{\delta}
\ge \frac{1}{D^{\delta}} \sum_{w\in (\gamma^{-1}\mathcal{A}_n)\cap \mathcal{B}(x,\,\frac{t}{2D'})} c_{\phi}(w)^{\delta}.
\]
}}
If, contrary to the claim,
\[
\limsup_{n\to\infty}\ \sum_{w\in \mathcal{A}_n\cap \mathcal{B}(y,t)} c_{\phi}(w)^{\delta} < \infty,
\]
then there exists a constant $C>0$ such that for all $n$,
\changednew{
\[
\sum_{w\in (\gamma^{-1}\mathcal{A}_n)\cap \mathcal{B}(x,\,\frac{t}{2D'})} c_{\phi}(w)^{\delta} < C.
\]
}
Since $\Gamma (b_0) = \bigcup_{n} \gamma^{-1}(\mathcal{A}_n)$, it follows that
\changednew{
\[
\sum_{w\in \mathcal{B}(x,\frac{t}{2D'})} c_{\phi}(w)^{\delta+\epsilon}
\le \sum_{n=0}^{\infty}\left(D^{\epsilon}e^{-\epsilon n}\!\!\sum_{w\in (\gamma^{-1}\mathcal{A}_n)\cap \mathcal{B}(x,\frac{t}{2D'})} c_{\phi}(w)^{\delta}\right)
\le C D^{\epsilon} \sum_{n=0}^{\infty} e^{-\epsilon n} < \infty,
\]
}
a contradiction.
\end{proof}

Now we conclude the proof of Proposition~\ref{prop8}. Still assume $x\in \Lambda(\Gamma)$ as above, choose $x'\in \Gamma(x)-\{x\}$ and fix
\[
0<t_0<\tfrac14\, d_{S^1}(x,x').
\]
By the strong coarse additivity of $\phi$, there exists $C_1=C_1(t_0)>0$ so that if $\gamma,\eta\in \Gamma$ satisfy
\[
d_{S^1}\big(l_{b_0}(\eta(b_0)),\,l_{b_0}(\gamma^{-1}(b_0))\big)\ \ge\ t_0,
\]
then
\begin{equation}\label{eqappendB5}
e^{-\phi(\kappa(\rho(\gamma\eta)))}\ \ge\ C_1\;e^{-\phi(\kappa(\rho(\gamma)))}\,e^{-\phi(\kappa(\rho(\eta)))}.  
\end{equation}

Also, let $r_0=r_0(t_0)>0$ and $N_0=N_0(t_0)>0$ be the constants given by Lemma~\ref{lem12}. Then let $C_0=C_0(2r_0+1)>0$ be the constant from Lemma~\ref{lem11refined}. Since $\Gamma$ is a discrete group, there exist a finite partition
\[
\Gamma(b_0)\;=\;P_1\ \cup\ \cdots\ \cup\ P_L
\]
such that each $P_i$ is $C_0$–separated. Since $\Gamma$ is discrete and $\phi$ has finite critical exponent,
\[
\lim_{n\to\infty}\ \min_{z\in \mathcal{A}_n} d(b_0,z)\;=\;\infty.
\]
So by Lemma~\ref{lem13refined}, there exist $n,n'\ge 1$ such that
\begin{equation}\label{eqappendB2}
    \min_{z\in \mathcal{A}_n\cup \mathcal{A}_{n'}} d(b_0,z)\ \ge\ N_0,\qquad
\sum_{w\in \mathcal{A}_n\cap \mathcal{B}(x,t_0)} c_{\phi}(w)^{\delta}\ \ge\ \frac{L}{C_1^{\delta}},
\qquad
\sum_{w\in \mathcal{A}_{n'}\cap \mathcal{B}(x',t_0)} c_{\phi}(w)^{\delta}\ \ge\ \frac{L}{C_1^{\delta}}.
\end{equation}

Thus, there exist $i_x,i_{x'}\in\{1,\dots,L\}$ such that
\begin{equation}\label{eqappendB4}
 \sum_{w\in S_x} c_{\phi}(w)^{\delta}\ \ge\ \frac{1}{C_1^{\delta}}
\qquad\text{and}\qquad
\sum_{w\in S_{x'}} c_{\phi}(w)^{\delta}\ \ge\ \frac{1}{C_1^{\delta}},   
\end{equation}
where $S_x=P_{i_x}\cap \mathcal{A}_n\cap \mathcal{B}(x,t_0)$ and $S_{x'}=P_{i_{x'}}\cap \mathcal{A}_{n'}\cap \mathcal{B}(x',t_0)$. Set
\[
D_0\ :=\ \max_{z\in \mathcal{A}_n\cup \mathcal{A}_{n'}} d(b_0,z).
\]

Pick $z=\gamma(b_0)\in \Gamma(b_0)-\{b_0\}$. Since $d_{S^1}(x,x')>4t_0$, there exists $y\in\{x,x'\}$ such that
\begin{equation}\label{eqappendB1}
d_{S^1}\big(y,\,l_{b_0}(\gamma^{-1}(b_0))\big)\ \ge\ 2t_0.
\end{equation}

Let $\mathcal{C}_{\delta}(z)=\gamma(S_y)\subset \Gamma(b_0)-\{z\}$. We check that $\mathcal{C}_{\delta}(z)$ satisfies all the requirements in Proposition~\ref{prop8}. 

Since $S_y\subset \mathcal{B}(y,t_0)$, Inequality~(\ref{eqappendB1}) implies that
\begin{equation}\label{eqappendB3}
  d_{S^1}\big(l_{b_0}(\gamma^{-1}(w)),l_{b_0}(\gamma^{-1}(b_0))\big)\ >\ t_0
\quad\text{for all } w\in \mathcal{C}_{\delta}(z).  
\end{equation}

Since $S_y\subset \mathcal{A}_n\cup \mathcal{A}_{n'}$, from Equation~(\ref{eqappendB2}) we have
\[
d_{\mathbb{D}}\big(b_0,\,\gamma^{-1}(w)\big)\ \ge\ N_0
\quad\text{for all } w\in \mathcal{C}_{\delta}(z).
\]
So for any $w\in \mathcal{C}_{\delta}(z)$, Inequality~(\ref{eqappendB3}) and Lemma~\ref{lem12} shows:
\[
\mathcal{O}_{2r_0}(\gamma^{-1}(b_0),\gamma^{-1}(w))\ \subset\ \mathcal{O}_{r_0}(\gamma^{-1}(b_0),b_0) \cap \mathcal{O}_{2r_0+1}\big(b_0,\gamma^{-1}(w)\big).
\]
Applying $\gamma$ to both sides we get
\[
\mathcal{O}_{2r_0}(b_0,w)\ \subset\ \mathcal{O}_{r_0}(b_0,z) \cap \mathcal{O}_{2r_0+1}\big(z,w\big)\subset \mathcal{O}_{r_0}(b_0,z) \quad\text{for all } w\in \mathcal{C}_{\delta}(z).
\]
So Proposition~\ref{prop8}~(1) holds.

Since $S_y\subset P_{i_y}$, $S_y$ is $C_0$–separated, so by Lemma~\ref{lem11refined},
\[
\mathcal{O}_{2r_0+1}\big(b_0,\gamma^{-1}(w)\big) \cap \mathcal{O}_{2r_0+1}\big(b_0,\gamma^{-1}(w')\big)\;=\;\varnothing
\quad\text{for }w\ne w'\in \mathcal{C}_{\delta}(z)
\]
From the same application of Lemma~\ref{lem12} as above, for $w\ne w'\in \mathcal{C}_{\delta}(z)$,
\[
\mathcal{O}_{2r_0}(b_0,w)\cap \mathcal{O}_{2r_0}(b_0,w') = \gamma \bigg( \mathcal{O}_{2r_0}\big(\gamma^{-1}(b_0),\gamma^{-1}(w)\big)\cap \mathcal{O}_{2r_0}\big(\gamma^{-1}(b_0),\gamma^{-1}(w')\big)\bigg)
\]
\[
\subset \gamma\bigg(\mathcal{O}_{2r_0+1}\big(b_0,\gamma^{-1}(w)\big)\cap \mathcal{O}_{2r_0+1}\big(b_0,\gamma^{-1}(w')\big)\bigg)=\emptyset.
\]
So Proposition~\ref{prop8}~(2) holds. 

Since $S_y\subset \mathcal{A}_n\cup \mathcal{A}_{n'}$, we have
\[
d_{\mathbb{D}}(z,w) = d_{\mathbb{D}}\big(b_0,\,\gamma^{-1}(w)\big)\ \le\ D_0
\quad\text{for all } w\in \mathcal{C}_{\delta}(z).
\]
So Proposition~\ref{prop8}~(3) holds. 

Finally, for each $w\in \mathcal{C}(z)$, choose $\eta_w\in \Gamma$ so that
\[
\eta_w(b_0)=\gamma^{-1}(w)\in S_y .
\]
Since $\eta_{w}(b_0)\in \mathcal{B}(y,t_0)$, from Inequality~(\ref{eqappendB1}) we have
\[
d_{S^1}\!\big(l_{b_0}(\eta_w(b_0)),\, l_{b_0}(\gamma^{-1}(b_0))\big) = d_{S^1}\big(l_{b_0}(\gamma^{-1}(w)),l_{b_0}(\gamma^{-1}(b_0))\big)\ >\ t_0,
\]
so by Inequality~(\ref{eqappendB5}) we have
\[
c_{\phi}(w)
= e^{-\phi(\kappa(\rho(\gamma\eta_w)))}
\ \ge\ C_1\, e^{-\phi(\kappa(\rho(\gamma)))}\, e^{-\phi(\kappa(\rho(\eta_w)))}
= C_1\, c_{\phi}(z)\, c_{\phi}\!\big(\gamma^{-1}(w)\big).
\]
Inequality~(\ref{eqappendB4}) implies that
\[
\sum_{w\in \mathcal{C}(z)} c_{\phi}(w)^{\delta}
\ \ge\ C_1^{\delta}\, c_{\phi}(z)^{\delta}\! \sum_{w\in S_y} c_{\phi}(w)^{\delta}
\ \ge\ c_{\phi}(z)^{\delta}.
\]
So Proposition~\ref{prop8}~(4) holds, which completes the proof of Proposition~\ref{prop8}.

}}

{\color{red}
{
\section{Ergodic components of multi-Fuchsian representations}\label{append3}
This appendix provides a more detailed explanation of Example~\ref{ergodicitycounterexample}.  
We use products of Fuchsian representations to construct maximal representations into $\mathsf{Sp}(2n,\mathbb{R})$ whose limit curves have $n$ ergodic components with respect to the Lebesgue measure.  
Moreover, \refchangenew{2.1.1.}{\changednew{recall that $\mathsf{Pos}(n)$ is the cone consisting of $n\times n$ positive definite matrices.}} We show that the rank (see Definition~\ref{rankofcone}) of $\mathsf{Pos}(n)$ is exactly $n$, which, by Theorem~\ref{ergodictheoreminpaper}, implies that $D(\mathfrak{sp}(2n,\mathbb{R}),\alpha)=n$ is optimal in Theorem~\ref{ergodictheorem}.

Let $\Gamma$ be a torsion-free closed surface group \refchangenew{2.1.2.}{\changednew{of genus at least two}}. For any $n\in\mathbb{Z}^{>0}$, pick
pairwise nonconjugate Fuchsian representations
\[
\rho_i:\Gamma\to \mathsf{SL}(2,\mathbb{R}),\qquad i=1,\dots,n.
\]
For each $1\le i\le n$ write, for $\gamma\in\Gamma$,
\[
\rho_i(\gamma)=
\begin{pmatrix}
a_i(\gamma)& b_i(\gamma)\\
c_i(\gamma)& d_i(\gamma)
\end{pmatrix}.
\]
Define $\rho:\Gamma\to \mathsf{SL}(2,\mathbb{R})^n \hookrightarrow \mathsf{Sp}(2n,\mathbb{R})$
to be the composition of $(\rho_1,\dots,\rho_n)$ with the block-diagonal embedding
relative to $\mathbb{R}^{2n}\cong \bigoplus_{i=1}^n \mathbb{R}^2$. In coordinates,
\[
\rho(\gamma)=
\begin{pmatrix}
\operatorname{diag}\!\big(a_i(\gamma)\big) & \operatorname{diag}\!\big(b_i(\gamma)\big)\\[2pt]
\operatorname{diag}\!\big(c_i(\gamma)\big) & \operatorname{diag}\!\big(d_i(\gamma)\big)
\end{pmatrix}\in \mathsf{Sp}(2n,\mathbb{R}),
\]
where $\operatorname{diag}(\cdot)$ denotes the $n\times n$ diagonal matrix with the indicated entries.
Since each $\mathsf{SL}(2,\mathbb{R})$ factor preserves the standard symplectic form on $\mathbb{R}^2$,
the block-diagonal image lies in $\mathsf{Sp}(2n,\mathbb{R})$.

$\rho$ is maximal \refchangenew{2.1.3.}{\changednew{(see Burger, Iozzi, Labourie and Wienhard~\cite[Example 3.9]{maximal})}} with respect to the unique long simple root
$\alpha$ of $\mathfrak{sp}(2n,\mathbb{R})$, characterized by
\[
\alpha\!\left(\operatorname{diag}(\lambda_1,\dots,\lambda_n,-\lambda_1,\dots,-\lambda_n)\right)=2\lambda_n.
\]
Moreover, let $\kappa$ denote the Cartan projection on $\mathsf{Sp}(2n,\mathbb{R})$,
and let $\kappa'$ denote the Cartan projection on $\mathsf{SL}(2,\mathbb{R})$.
Let $\alpha'$ be the (unique) simple root of $\mathfrak{sl}(2,\mathbb{R})$.
Then, for all $\gamma\in\Gamma$,
\[
\alpha\big(\kappa(\rho(\gamma))\big)
 \;=\; \min_{1\le i\le n}\,\alpha'\big(\kappa'(\rho_i(\gamma))\big).
\]

We continue to use $\xi^{\alpha}:\Lambda(\Gamma)\to \mathcal{F}_{\alpha}$ to denote the positive limit map and $m_{\alpha}$ to denote the Lebesgue measure on $\xi^{\alpha}(\Lambda(\Gamma))$, as in Section~\ref{sectionfinitenessofergodiccomponents}. \refchangenew{2.1.4.}{\changednew{We denote the pullback of $m_{\alpha}$ via $\xi^{\alpha}$ by $m'$, which is a Borel measure on $\Lambda(\Gamma)$.}}

\refchangenew{2.1.4.}{\changednew{Note that the $\alpha'$-flag manifold $\mathcal{F}_{\alpha'}$ of $\mathsf{SL}(2,\mathbb{R})$ is isomorphic to the boundary $\partial\mathbb{D}\cong S^1$. For each $i=1,\dots,n$, let $\xi^{\alpha',i}:\Lambda(\Gamma)\to \partial \mathbb{D}$ denote the limit map associated with $\rho_i$, and let $m'_i$ denote the pullback of the Lebesgue measure on $\partial\mathbb{D}$ via $\xi^{\alpha',i}$, which is a measure on $\Lambda(\Gamma)$.}}

\begin{theorem}\label{ncomponents}
$m_{\alpha}$ has $n$ ergodic components under the $\rho(\Gamma)$ action.
\end{theorem}

\begin{proof}
\refchangenew{2.1.4.}{\changednew{It is equivalent to prove that $m'$ has $n$ ergodic components under the $\Gamma$ action.}}

Fix $r>0$ and $b_0\in \mathbb{D}$.  
Since Fuchsian representations are maximal when we identify $\mathsf{SL}(2,\mathbb{R})$ with $\mathsf{Sp}(2,\mathbb{R})$, by applying Corollary~\ref{shadowradius}\,(1) to $\rho_i$ ($1\le i\le n$) and to $\rho$, there exists a constant $C>1$ such that for any $\gamma\in \Gamma$ with $d_{\mathbb{D}}(b_0,\gamma(b_0))>r$, we have
\[
\frac{1}{C}\, e^{-\alpha'(\kappa'(\rho_i(\gamma)))} 
   \le \refchangenew{2.1.4.}{\changednew{m'_{i}\big(\mathcal{O}_r(b_0,\gamma(b_0))\big)}}
   \le C\, e^{-\alpha'(\kappa'(\rho_i(\gamma)))}, 
   \qquad i = 1,\dots,n,
\]
and
\[
\frac{1}{C}\, e^{-\alpha(\kappa(\rho(\gamma)))} 
   \le \refchangenew{2.1.4.}{\changednew{m'\big(\mathcal{O}_r(b_0,\gamma(b_0))\big)}}
   \le C\, e^{-\alpha(\kappa(\rho(\gamma)))}.
\]

Since
\[
\alpha\big(\kappa(\rho(\gamma))\big)
   = \min_{1\le i\le n}\,\alpha'\big(\kappa'(\rho_i(\gamma))\big),
\]
we deduce:
\begin{enumerate}
    \item For any $i$,
    \[
    \refchangenew{2.1.4.}{\changednew{m' \big(\mathcal{O}_r(b_0,\gamma(b_0))\big)}}
       \ge \frac{1}{C^2}\,
         \refchangenew{2.1.4.}{\changednew{m'_{i}\big(\mathcal{O}_r(b_0,\gamma(b_0))\big)}}.
    \]
    \item We also have
    \[
    \refchangenew{2.1.4.}{\changednew{m'\big(\mathcal{O}_r(b_0,\gamma(b_0))\big)}}
       \le \sum_{i=1}^n C^2\,
         \refchangenew{2.1.4.}{\changednew{m'_{i}\big(\mathcal{O}_r(b_0,\gamma(b_0))\big)}}.
    \]
\end{enumerate}
Hence,
\[
   \refchangenew{2.1.4.}{\changednew{\frac{1}{nC^2}\sum_{i=1}^n \,
m'_{i}\big(\mathcal{O}_r(b_0,\gamma(b_0))\big)
   \;\le\;
m' \big(\mathcal{O}_r(b_0,\gamma(b_0))\big)
   \;\le\;
  C^2 \sum_{i=1}^n \,
   m'_{i}\big(\mathcal{O}_r(b_0,\gamma(b_0))\big)}}.
\]

Repeating the same Vitali-covering argument as in the proof of Proposition~\ref{acutemeansfull} we see that    \refchangenew{2.1.4.}{\changednew{$m'$ and $\sum_{i = 1}^{n} m'_{i}$}} are mutually absolutely continuous measures.  
On the other hand, it is standard that each \refchangenew{2.1.4.}{\changednew{$m'_{i}$}} is ergodic, and since the $\rho_i$ are pairwise nonconjugate, the measures    \refchangenew{2.1.4.}{\changednew{$m'_{i}$ are mutually singular (see Kuusalo~\cite[Theorem~1]{KUUSALO}).}}  
Consequently,    \refchangenew{2.1.4.}{\changednew{$\sum_{i = 1}^{n} m'_{i}$}} has $n$ ergodic components, and so does    \refchangenew{2.1.4.}{\changednew{$m'$}}.
\end{proof}

Next we show the rank of $\mathsf{Pos}(n)$ is $n$ (see Definition~\ref{rankofcone}), so the above example implies that the estimate of ergodic components in Theorem~\ref{ergodictheoreminpaper} is optimal.
\begin{proposition}\label{rankofpos}
$r(\mathsf{Pos}(n)) = n$.
\end{proposition}
\begin{proof}
Firstly, $r(\mathsf{Pos}(n)) \ge n$, as $\sum_{i = 1}^{n} E_{i,i}=I_n$ is positive definite, but deleting any $E_{i,i}$ the remaining sum is not,  
where $E_{i,i}$ denotes the matrix whose only nonzero entry is $1$ on the $i$-th row and the $i$-th column.  
Thus, it remains to prove the following statement:

\medskip
\noindent
\textbf{Claim.}
Let $A_0,\dots,A_n$ be $n+1$ symmetric positive semidefinite matrices.  
If $\sum_{i=0}^{n} A_i$ is positive definite, then there exists an index $0 \le k \le n$ such that  
\[
\sum_{\substack{0\le i\le n \\ i\ne k}} A_i
\]
is also positive definite.

\vspace{3mm}

Define $S:=\sum_{i=0}^n A_i$. By the assumption that $S$ is positive definite, \refchangenew{2.1.5.}{\changednew{the square root $S^{\frac{1}{2}}\in \mathsf{Pos}(n)$ is well-defined. Define}}
\[
B_i \;:=\; S^{-1/2} A_i S^{-1/2}\ \ (\text{positive semidefinite}), 
\qquad\text{so}\qquad \sum_{i=0}^n B_i \;=\; I_n .
\]
If the maximal eigenvalue $\lambda_{\max}(B_k)<1$ for some $k$, then
\[
S-A_k \;=\; S^{1/2}\bigl(I_n-B_k\bigr)S^{1/2}
\]
is positive definite, and we are done.  
Thus it suffices to show that there exists $k$ with $\lambda_{\max}(B_k)<1$.

Suppose on the contrary that $\lambda_{\max}(B_i)\ge 1$ for all $i=0,\dots,n$.  
Since each $B_i$ is positive semidefinite, $\operatorname{tr}(B_i)\ge \lambda_{\max}(B_i)\ge 1$, hence
\[
n \;=\; \operatorname{tr}\!\Bigl(\sum_{i=0}^n B_i\Bigr)
   \;=\; \sum_{i=0}^n \operatorname{tr}(B_i)
   \;\ge\; \sum_{i=0}^n 1
   \;=\; n+1,
\]
a contradiction.  
Therefore some $k$ satisfies $\lambda_{\max}(B_k)<1$, and consequently $\sum_{i\ne k} A_i = S-A_k$ is positive definite.
\end{proof}

}
}

\bibliographystyle{amsplain}
\bibliography{ref}

\end{document}